\documentclass [12pt]{amsart}
\usepackage[utf8]{inputenc}
\usepackage{mathdots}
\usepackage{amsmath,amssymb,amsthm,amsfonts}
\usepackage{mathtools}

\usepackage[english]{babel}
\usepackage{comment}
\usepackage{hyperref}
\usepackage{bbm}
\usepackage{mathrsfs}

\usepackage{tikz,graphicx,color}
\usepackage{tikz-cd}
\usepackage{tikz-3dplot}
\usetikzlibrary{calc}
\usetikzlibrary{arrows}
\usetikzlibrary{shapes}
\usetikzlibrary{patterns}
\usetikzlibrary{positioning}
\usepackage{epstopdf}

\usepackage{enumerate}
\usepackage[arrow]{xy}
\usepackage{diagbox}
\usepackage{subfig}
\usepackage{chngcntr} 
\usepackage{arcs}
\usepackage{xcolor}

\newtheorem{theorem}{Theorem}[section]

\newtheorem{proposition}[theorem]{Proposition}

\theoremstyle{definition}
\newtheorem{remark}[theorem]{Remark}
\theoremstyle{definition}
\newtheorem{definition}[theorem]{Definition}
\newtheorem{conjecture}[theorem]{Conjecture}

\theoremstyle{definition}

\theoremstyle{definition}
\newtheorem{example}[theorem]{Example}

\def\Acal{\mathcal{A}}\def\Ccal{\mathcal{C}}\def\Hcal{\mathcal{H}}\def\Pcal{\mathcal{P}}

\def\ubf{\mathbf{u}}\def\vbf{\mathbf{v}}\def\wbf{\mathbf{w}}\def\xbf{\mathbf{x}}

\def\Sfr{{ \mathfrak{S}}}

\def\R{\mathbb{R}}
\def\N{\mathbb{N}}
\def\Z{\mathbb{Z}}
\def\Q{\mathbb{Q}}

\newcommand\parr[1]{{({#1})}}

\def\<{{\langle}}
\def\>{{\rangle}}
\def\e{{\epsilon}}
\def\l{{\lambda}}
\def\m{{\mu}}

\def\multiset#1#2{\left(\!\left({#1\atopwithdelims..#2}\right)\!\right)}

\def\Id{\operatorname{Id}}

\def\det{{ \operatorname{det}}}

\def\proj{ \operatorname{proj}}

\usetikzlibrary{arrows,decorations.markings}

\def\TS{{\mathcal{S}}}
\def\TC{{\mathcal{O}}}
\def\x{{\xbf}}

\def\Cyc{{ \Ccal}}

\def\Pbf{\mathbf{P}}
\def\Cbf{\mathbf{C}}
\def\sgn{ \operatorname{sgn}}
\def\g{{\tilde{g}}}
\def\Nt{{\tilde{N}}}
\def\Vt{{\tilde{V}}}
\def\ut{{\tilde{u}}}
\def\vt{{\tilde{v}}}

\def\et{{\tilde{e}}}

\def\Et{{\tilde{E}}}
\def\Pt{{\tilde{P}}}
\def\pt{{\tilde{p}}}

\def\wtt{{\tilde{w}}}
\def\wind{{ \operatorname{wind}}}
\def\ee{{\Nt}}
\def\gen{h}
\def\pathst{\tilde\Pcal}
\def\outg{ {\operatorname{out}_\Nt}}
\def\inc{ {\operatorname{in}_\Nt}}

\def\ubft{{\tilde{\ubf}}}

\def\vbft{{\tilde{\vbf}}}
\def\Pbft{{\tilde{\Pbf}}}

\newcommand\pleh[1]{{[h_{#1}]}}
\newcommand\plee[1]{{[e_{#1}]}}
\def\Kbar{{\overline{K}}}

\def\wt{ \operatorname{wt}}

\usepackage{tabu}
\usepackage{booktabs}

\begin{document}
\numberwithin{equation}{section}

\title{Linear recurrences for cylindrical networks}
\author{Pavel Galashin}
\address{Department of Mathematics, Massachusetts Institute of Technology,
Cambridge, MA 02139, USA}
\email{{\href{mailto:galashin@mit.edu}{galashin@mit.edu}}}

\author{Pavlo Pylyavskyy}
\address{Department of Mathematics, University of Minnesota,
Minneapolis, MN 55414, USA}
\email{{\href{mailto:ppylyavs@umn.edu}{ppylyavs@umn.edu}}}

\date{\today}

\thanks{P.~P. was partially supported by NSF grants  DMS-1148634, DMS-1351590, and Sloan Fellowship.}



\begin{abstract}
We prove a general theorem that gives a linear recurrence for tuples of paths in every cylindrical network. This can be seen as a cylindrical analog of the Lindstr\"om-Gessel-Viennot theorem. We illustrate the result by applying it to Schur functions, plane partitions, and domino tilings. 
\end{abstract}
\maketitle

\setcounter{tocdepth}{1}
\tableofcontents

\section{Introduction}
 
 Given a weighted directed graph $N$, the celebrated Lindstr\"om-Gessel-Viennot theorem~\cite{GV,Lindstrom} gives a combinatorial interpretation for a certain determinant in terms of tuples of vertex-disjoint paths in $N$. Its applications in several different contexts have been studied extensively, such as: 
 \begin{itemize}
 	\item alternating sign matrices (equivalently, totally symmetric self-complementary plane partitions)~\cite{BdFZJ,ZJdF},
 	\item domino tilings of the Aztec diamond~\cite{BvL} (equivalently, states of the six-vertex model~\cite{FerrariSpohn,ZJ}, or monomials in the octahedron recurrence~\cite{Sp}),
 	\item vicious walkers model~\cite{VW1,VW2,VW3},
 	\item super-Schur functions~\cite{BrentiS}, $Q$-Schur functions~\cite{StembridgeQ}, etc.
 \end{itemize}

 In a series of papers~\cite{GP1,GP2,GP3}, we studied Zamolodchikov phenomena associated with \emph{finite $\boxtimes$ finite}, \emph{affine $\boxtimes$ finite}, and \emph{affine $\boxtimes$ affine} quivers. In particular, in the affine $\boxtimes$ finite case~\cite{GP2} we showed that the values of the $T$-system of type $\hat A\otimes A$ (equivalently, the values of the octahedron recurrence in a cylinder) satisfy a simple linear recurrence relation whose coefficients admit a nice combinatorial interpretation: they are sums over domino tilings of the cylinder with fixed \emph{Thurston height}. Similarly, in~\cite{GPCube} we show using a formula given by Carroll and Speyer~\cite{CS} that the values of the cube recurrence in a cylinder satisfy a similar linear recurrence relation, where the coefficients now are sums over \emph{groves}. 
 
 This led us to the formulation of a general theorem (Theorem~\ref{thm:recurrences}) that applies to most of the networks mentioned above. We state two different versions of it, one for arbitrary \emph{cylindrical networks} (see Definition~\ref{dfn:network}) and one for \emph{planar cylindrical networks} (see Definition~\ref{dfn:network_planar}), i.e., cylindrical networks that can be drawn in a cylinder without self intersections. The first version is more general, however, the second version gives a stronger result, that is, a shorter recurrence relation which is conjecturally minimal (see Conjecture~\ref{conj:minimal}). A closely related but different result on rationality of certain measurements taken in cylindrical networks has previously appeared in \cite[Proposition 2.7]{LP}.
 
In Section~\ref{sect:main}, we define cylindrical networks and related notions and state our main results, Theorem~\ref{thm:single_path} and its generalization, Theorem~\ref{thm:recurrences}. We then prove Theorem~\ref{thm:single_path} in Section~\ref{sect:single_path}. After that, in Section~\ref{sect:LGV} we give some background on plethysms of symmetric functions and the Lindstr\"om-Gessel-Viennot method which we then use to prove Theorem~\ref{thm:recurrences}.

In Section~\ref{sect:applications}, we give several applications of Theorem~\ref{thm:recurrences}. We start with the case of Schur functions, see Section~\ref{sect:Schur}. We show how Theorem~\ref{thm:recurrences} gives a linear recurrence for sequences of Schur polynomials of the form $f(\ell)=s_{(\l+\ell\m)}(x_1,\dots,x_n)$ where $\m$ is of rectangular shape. This gives an alternative proof to a recent result of Alexandersson~\cite{Per,Per2} which however applies to more general sequences of the form $s_{(\l+\ell\m)/(\kappa+\ell\nu)}$ where the partitions $\l,\m,\nu,\kappa$ are not assumed to be of rectangular shape. We also apply our results to lozenge tilings (see Section~\ref{sect:lozenge}) and domino tilings (Section~\ref{sect:domino}) in a strip.\footnote{We thank Per Alexandersson for pointing out his paper~\cite{Per2} to us.}   

Finally, we give two conjectures for planar cylindrical networks in Section~\ref{sect:conjectures}. The first one states that the recurrence polynomials in the second part of Theorem~\ref{thm:recurrences} have nonnegative integer coefficients (Conjecture~\ref{conj:Schur_Polya}) and positive real roots (Conjecture~\ref{conj:roots}). The second one (Conjecture~\ref{conj:minimal}) asks whether these polynomials are always \emph{minimal} if the network is strongly connected and has algebraically independent edge weights.

 \section{Main results}\label{sect:main}
 
 We start by briefly introducing our main objects of study. More precise definitions are given in Sections~\ref{sect:cyl_networks} and~\ref{sect:planar}. 
 
 Consider an acyclic directed graph $\Nt$ drawn in some horizontal strip $\TS=\{(x,y)\in\R^2\mid A\leq y\leq B\}$ in the plane such that its vertex set $\Vt$ and edge set $\Et$ are invariant with respect to the shift by some horizontal vector $\g=(m,0)\in\R^2$. Suppose in addition that we are given a shift-invariant function $\wt:\Et\to K$ assigning weights from some field $K$ of characteristic zero\footnote{For us $K$ will be either $\R$ or the field of rational functions in several variables.} to the edges of $\Nt$. We call such a weighted directed graph \emph{a cylindrical network}. We also define in an obvious way the \emph{projection} $N$ of $\Nt$ to the cylinder $\TC=\TS/\Z\g$. Thus $N$ is a weighted directed graph drawn in the cylinder. We require that all the vertices of $\Nt$ have finite degree and that for every directed path in $\Nt$ connecting a vertex $\vt$ to its shift $\vt+\ell\g$ we have $\ell>0$. 
 
 We say that a cycle $C$ in $N$ is \emph{simple} if it passes through each vertex of $N$ at most once. For a simple cycle $C$ in $N$, we define its \emph{winding number} $\wind(C)$ to be the unique integer $\ell$ such that any lift of $C$ to a path in $\Nt$ connects a vertex $\vt$ to $\vt+\ell\g$. Thus for any cylindrical network $\Nt$, \emph{any cycle in $N$ has a positive winding number}. An example of a cylindrical network can be found in Figure~\ref{fig:network_example}.

\begin{figure}
 \centering

 \scalebox{0.6}{
\begin{tikzpicture}[scale=3.0]
\tikzset{myptr/.style={decoration={markings,mark=at position 1 with %
    {\arrow[scale=3,>=stealth]{>}}},postaction={decorate}}}
\node[] (node0D) at (0.00,0.00) {};
\node[] (node0U) at (0.00,1.00) {};
\node[scale=1.25,anchor=south] (zzz1U) at (1.00,1.05) {$\ut$};
\node[scale=1.25,anchor=north] (zzz1D) at (1.00,-0.05) {$\vt-\g$};
\node[draw,ellipse] (node1D) at (1.00,0.00) {};
\node[draw,ellipse] (node1U) at (1.00,1.00) {};
\node[scale=1.25,anchor=south] (zzz2U) at (2.00,1.05) {$\ut+\g$};
\node[scale=1.25,anchor=north] (zzz2D) at (2.00,-0.05) {$\vt$};
\node[draw,ellipse] (node2D) at (2.00,0.00) {};
\node[draw,ellipse] (node2U) at (2.00,1.00) {};
\node[scale=1.25,anchor=south] (zzz3U) at (3.00,1.05) {$\ut+2\g$};
\node[scale=1.25,anchor=north] (zzz3D) at (3.00,-0.05) {$\vt+\g$};
\node[draw,ellipse] (node3D) at (3.00,0.00) {};
\node[draw,ellipse] (node3U) at (3.00,1.00) {};
\node[scale=1.25,anchor=south] (zzz4U) at (4.00,1.05) {$\ut+3\g$};
\node[scale=1.25,anchor=north] (zzz4D) at (4.00,-0.05) {$\vt+2\g$};
\node[draw,ellipse] (node4D) at (4.00,0.00) {};
\node[draw,ellipse] (node4U) at (4.00,1.00) {};
\node[] (node5D) at (5.00,0.00) {};
\node[] (node5U) at (5.00,1.00) {};
\draw[myptr] (node0U) to node[midway,below,scale=1.25] {$a$} (node1U);
\draw[myptr] (node0D) to node[midway,above,scale=1.25] {$e$} (node1D);
\draw[myptr] (node0U) to node[pos=0.25,below,scale=1.25] {$b$} (node1D);
\draw[myptr] (node0D) to node[pos=0.35,below,scale=1.25] {$d$} (node1U);
\draw[myptr] (node1U) to node[midway,below,scale=1.25] {$a$} (node2U);
\draw[myptr] (node1D) to node[midway,above,scale=1.25] {$e$} (node2D);
\draw[myptr] (node1U) to node[midway,right,scale=1.25] {$c$} (node1D);
\draw[myptr] (node1U) to node[pos=0.25,below,scale=1.25] {$b$} (node2D);
\draw[myptr] (node1D) to node[pos=0.35,below,scale=1.25] {$d$} (node2U);
\draw[myptr] (node2U) to node[midway,below,scale=1.25] {$a$} (node3U);
\draw[myptr] (node2D) to node[midway,above,scale=1.25] {$e$} (node3D);
\draw[myptr] (node2U) to node[midway,right,scale=1.25] {$c$} (node2D);
\draw[myptr] (node2U) to node[pos=0.25,below,scale=1.25] {$b$} (node3D);
\draw[myptr] (node2D) to node[pos=0.35,below,scale=1.25] {$d$} (node3U);
\draw[myptr] (node3U) to node[midway,below,scale=1.25] {$a$} (node4U);
\draw[myptr] (node3D) to node[midway,above,scale=1.25] {$e$} (node4D);
\draw[myptr] (node3U) to node[midway,right,scale=1.25] {$c$} (node3D);
\draw[myptr] (node3U) to node[pos=0.25,below,scale=1.25] {$b$} (node4D);
\draw[myptr] (node3D) to node[pos=0.35,below,scale=1.25] {$d$} (node4U);
\draw[myptr] (node4U) to node[midway,below,scale=1.25] {$a$} (node5U);
\draw[myptr] (node4D) to node[midway,above,scale=1.25] {$e$} (node5D);
\draw[myptr] (node4U) to node[midway,right,scale=1.25] {$c$} (node4D);
\draw[myptr] (node4U) to node[pos=0.25,below,scale=1.25] {$b$} (node5D);
\draw[myptr] (node4D) to node[pos=0.35,below,scale=1.25] {$d$} (node5U);
\node[anchor=east,scale=2] (dots1) at (0.00,0.50) {$\dots$};
\node[anchor=west,scale=2] (dots2) at (5.00,0.50) {$\dots$};
\end{tikzpicture}}

 \caption{\label{fig:network_example}An example of a cylindrical network $\Nt$. The edge weights are $a,b,c,d,e$. The network $N$ has two vertices $u$ and $v$.}
\end{figure}
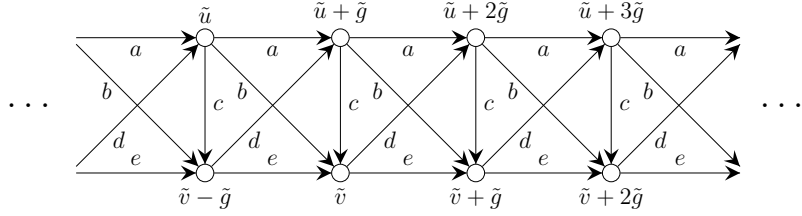

An $r$-cycle $\Cbf=(C_1,\dots,C_r)$ in $N$ is an $r$-tuple of pairwise vertex disjoint simple cycles in $N$. We set $\wt(\Cbf)=\wt(C_1)\cdots\wt(C_r)$ where $\wt(C_i)$ is the product of the edges of $C_i$. We put 
\[\wind(\Cbf)=\wind(C_1)+\wind(C_2)+\dots+\wind(C_r).\] 
The set of all $r$-cycles in $N$ is denoted by $\Cyc^r(N)$.

We now define the polynomial $Q_N(t)\in K[t]$ by
\begin{equation}\label{eq:Q_N_intro}
Q_N(t)=\sum_{r=0}^d(-1)^{d-r}\sum_{\Cbf\in\Cyc^r(N)} t^{d-\wind(\Cbf)} \wt(\Cbf).
\end{equation}
Here the degree $d$ of $Q_N(t)$ is the maximal winding number of an $r$-cycle in $N$ for $r\geq 0$. It is clear that $d$ is finite. Thus $Q_N(t)$ is a monic polynomial in $t$ of degree $d$. For instance, for the network $\Nt$ in Figure~\ref{fig:network_example}, we get $Q_N(t)=t^2-(a+e+cd)t+(ae-bd)$, see Example~\ref{ex:network}. 

Given two vertices $\ut,\vt\in\Vt$, define
\[\ee(\ut,\vt)=\sum_{\Pt} \wt(\Pt),\]
where the sum is taken over all paths $\Pt$ in $\Nt$ that start at $\ut$ and end at $\vt$, and the weight $\wt(\Pt)$ of a path is the product of weights of its edges. We are ready to state our first result:

\begin{theorem}\label{thm:single_path}
Consider a cylindrical network $\Nt$. Let $\ut$ and $\vt$ be any two vertices of $\Nt$. For $\ell\geq 0$, let $\vt_\ell=\vt+\ell\g\in\Vt$ be the shift of $\vt$, and define a sequence $f:\N\to K$ by $f(\ell):=\ee(\ut,\vt_\ell)$. Then for all but finitely many values of $\ell$, the sequence $f$ satisfies a linear recurrence with characteristic polynomial $Q_N(t)$. 
\end{theorem}

We say that $\Nt$ is a \emph{planar cylindrical network} if $\Nt$ is drawn in the strip $\TS$ without self-intersections. More specifically, we require that the edges are drawn in a way that is shift-invariant so that $N$ is also drawn in the cylinder $\TC$ without self-intersections. For planar cylindrical networks, it is easy to see that every simple cycle in $N$ has winding number $1$. Thus the formula~\eqref{eq:Q_N_intro} simplifies as follows:

\begin{equation}\label{eq:Q_N_planar_intro}
	Q_N(t)=\sum_{r=0}^d(-t)^{d-r}\sum_{\Cbf\in\Cyc^r(N)} \wt(\Cbf).
\end{equation}

Therefore for a planar cylindrical network $N$, the coefficient of $(-t)^{d-r}$ in $Q_N(t)$ is a polynomial in the edge weights with nonnegative coefficients.

We now pass to the ``Lindstr\"om-Gessel-Viennot'' part of our results. A few more definitions are in order.

\begin{definition}
	An \emph{$r$-vertex} $\vbft=(\vt_1,\dots,\vt_r)$ in $\Nt$ is an $r$-tuple of distinct vertices of $\Nt$. An \emph{$r$-path} $\Pbft=(\Pt_1,\dots,\Pt_r)$ is an $r$-tuple of paths in $\Nt$ that are pairwise vertex disjoint, and we set $\wt(\Pbft)=\wt(\Pt_1)\cdots\wt(\Pt_r)$. If for $1\leq i\leq r$, the path $\Pt_i$ starts at $\ut_i$ and ends at $\vt_i$ then $\ubft=(\ut_1,\dots,\ut_r)$ and $\vbft=(\vt_1,\dots,\vt_r)$ are called the \emph{start} and the \emph{end} of $\Pbft$. We denote by $\pathst(\ubft,\vbft)$ the collection of all $r$-paths in $\Nt$ that start at $\ubft$ and end at $\vbft$, and we set 
	\[\ee(\ubft,\vbft):=\sum_{\Pbft\in\pathst(\ubft,\vbft)} \wt(\Pt).\] 
\end{definition}

Given an $r$-vertex $\vbft=(\vt_1,\dots,\vt_r)$ and a permutation $\sigma\in\Sfr_r$ of $[r]$, we denote by $\sigma\vbft=(\vt_{\sigma(1)},\dots,\vt_{\sigma(r)})$ the action of $\sigma$ on $\vbft$.

For each $r\geq 1$, we introduce certain polynomials $Q^\pleh{r}_N(t)$ and $Q^\plee{r}_N(t)$ of degrees $\multiset{d}{r}$ and $d\choose r$ respectively. To give a precise definition, suppose that we are given the roots $\gamma_1,\gamma_2,\dots,\gamma_d\in \Kbar$ of $Q_N(t)$, where $\Kbar$ denotes the algebraic closure of $K$. Thus we can write
\[Q_N(t)=(t-\gamma_1)(t-\gamma_2)\cdots(t-\gamma_d).\]
For $r\geq 1$, we set
\begin{equation}\label{eq:plee_pleh_intro}
\begin{split}
Q^\pleh{r}(t)&=\displaystyle\prod_{1\leq i_1\leq i_2\leq \dots\leq i_r\leq d} (t-\gamma_{i_1}\gamma_{i_2}\cdots\gamma_{i_r});\\
Q^\plee{r}(t)&=\displaystyle\prod_{1\leq i_1< i_2<\dots< i_r\leq d} (t-\gamma_{i_1}\gamma_{i_2}\cdots\gamma_{i_r}).\\
\end{split}
\end{equation}

For example, $Q^\pleh{1}(t)=Q^\plee{1}(t)=Q(t)$ and $Q^\plee{d}(t)=t-\alpha_d$, where $\alpha_d$ denotes the constant term of $Q(t)$. On the other hand, $Q^\pleh{d}(t)$ is a polynomial of degree $\multiset{d}{d}={2d-1\choose d}$. It is clear from~\eqref{eq:plee_pleh_intro} that $Q^\pleh{r}(t)$ is always divisible by $Q^\plee{r}(t)$ in $\Kbar[t]$. It is non-trivial to show that $Q^\pleh{r}(t)$, $Q^\plee{r}(t)$, and their ratio all belong to $K[t]$. Their coefficients are polynomial expressions in the coefficients of $Q_N(t)$ described explicitly in terms of plethysms of symmetric functions, see Section~\ref{sect:plethysm}. 
  
\begin{theorem}\label{thm:recurrences}
	Let $\Nt$ be a cylindrical network and let $\ubft=(\ut_1,\dots,\ut_r)$ and $\vbft=(\vt_1,\dots,\vt_r)$ be two $r$-vertices in $\Nt$. For $\ell\geq 0$, let $\vbft_\ell=\vbft+\ell\g=(\vt_1+\ell\g,\dots,\vt_r+\ell\g)$. Define the sequence $f:\N\to K$ by 
	\[f(\ell)=\sum_{\sigma\in\Sfr_r} \sgn(\sigma)\ee(\ubft,\sigma\vbft_\ell),\]
	where $\sgn(\sigma)$ denotes the \emph{sign} of $\sigma$. 
	\begin{enumerate}
		\item\label{item:thm:pleh} For any cylindrical network $\Nt$, the sequence $f$ satisfies a linear recurrence with characteristic polynomial $Q_N^\pleh{r}(t)$.
		\item\label{item:thm:plee} If $\Nt$ is a planar cylindrical network, the sequence $f$ satisfies a linear recurrence with characteristic polynomial $Q_N^\plee{r}(t)$.
	\end{enumerate}
	In both of the cases, $f$ satisfies the recurrence for all but finitely many values of $\ell$.
\end{theorem}

\section{Linear recurrences for single paths}\label{sect:single_path}
In this section, we first describe some standard material on linear recurrences in Section~\ref{sect:linear_background}, then we define cylindrical networks and related notions rigorously in Section~\ref{sect:cyl_networks}. We then proceed to the proof of Theorem~\ref{thm:single_path} in Section~\ref{sect:single_path}.

\subsection{Background on linear recurrences and characteristic polynomials.}\label{sect:linear_background}
We briefly recall some well known algebraic facts, mostly following~\cite{EC1}. Let $K$ be a field of characteristic zero and $f:\N\to K$ be a sequence of elements of $K$. We say that $f$ \emph{satisfies a linear recurrence} if there exist some elements $\alpha_1,\alpha_2,\dots,\alpha_d\in K$  such that for all but finitely many values of $n\in \N$ we have
\[f(n+d)-\alpha_{1}f(n+d-1)+\dots+(-1)^d \alpha_df(n)=0.\]
The \emph{characteristic polynomial} of this linear recurrence is the polynomial $Q(t)\in K[t]$ defined by 
\[Q(t)=t^d-\alpha_{1}t^{d-1}+\dots+(-1)^d \alpha_d.\]
\begin{proposition}[{\cite[Theorem~4.1.1]{EC1}}]\label{prop:rational}
	A sequence $f:\N\to K$ satisfies a linear recurrence with characteristic polynomial $Q(t)$ if and only if there exists a polynomial $P(t)\in K[t]$ such that the generating function for $f$ is a rational formal power series:
\[\sum_{n\in\N} f(n) t^n=\frac{P(t)}{R(t)},\]
where the polynomials $R(t)$ and $Q(t)$ are \emph{palindromic images} of each other:
\[t^d Q(1/t)=R(t).\]
\end{proposition}

The following is an obvious fact which we will repeatedly use:
\begin{proposition}\label{prop:sum}
	Suppose that $f_1,f_2,\dots,f_r:\N\to K$ each satisfy a linear recurrence with characteristic polynomial $Q(t)$. Then for any $c_1,\dots,c_r\in K$, the sequence $f=(c_1f_1+c_2f_2+\dots+c_rf_r):\N\to K$ defined by 
	\[f(\ell)=c_1f_1(\ell)+c_2f_2(\ell)+\dots+c_rf_r(\ell)\]
	satisfies a linear recurrence with the same characteristic polynomial $Q(t)$.
\end{proposition}

\subsection{Cylindrical networks}\label{sect:cyl_networks}

Recall that $\TC=\TS/\Z\g$ denotes the cylinder and denote the canonical projection map by $\proj:\TS\to\TC$. Let $\gen:[0,1]\to\TC$ be a loop defined by $\gen(t)=\proj(\g t)$ for $t\in[0,1]$. Thus the element $[\gen]\in\pi_1(\TC)$ generates the fundamental group $\pi_1(\TC)$. For a loop $p:[0,1]\to\TC$, define its \emph{winding number} $\wind(p)$ to be the unique integer $\ell$ such that $[\gen]^\ell=[p]$ in $\pi_1(\TC)$. 

\begin{definition}\label{dfn:network}
A triple $\Nt=(\Vt,\Et,\wt)$ is a \emph{cylindrical network} if:
\begin{enumerate}
	\item the \emph{vertex set $\Vt$ of $\Nt$}  is a discrete subset of $\TS$ that is invariant with respect to the shift by $\g$: $\Vt=\Vt+\g$;
	\item the \emph{edge set} $\Et\subset (\Vt\times\Vt)$ is a set of ordered pairs $(\ut,\vt)$ of elements $\ut,\vt\in\Vt$ such that if $\et=(\ut,\vt)\in\Et$ then $\et+\g:=(\ut+\g,\vt+\g)\in\Et$;
	\item $\wt:\Et\to K$ is an assignment of \emph{weights} from some field $K$ of characteristic zero to the edges of $\Nt$ satisfying $\wt(\et+\g)=\wt(\et)$;
	\item\label{item:finite} for every vertex $\vt\in\Vt$ of $\Nt$, the sets $\outg(\vt)$ and $\inc(\vt)$ of outgoing and incoming edges of $\vt$ in $\Nt$ are finite.
	\item if for some $\ell\in \Z$ there exists a directed path in $\Nt$ from $\vt\in\Vt$ to $\vt+\ell\g$ then $\ell>0$. In particular, $\Nt$ is acyclic.
\end{enumerate}
\end{definition}
We do not allow multiple edges in $\Nt$, however, replacing several edges with the same endpoints by a single edge whose weight equals the sum of their weights does not change any of the quantities we are interested in.

We view every element $\et=(\ut,\vt)\in\Et$ as a linear path $\et:[0,1]\to \TS$ starting at $\ut$ and ending at $\vt$ defined by $\et(t)=(1-t)\ut+t\vt$. Note that the paths corresponding to different edges may intersect each other. 

We let $N=(V,E,\wt)$ be the \emph{projection} of $\Nt$ defined by $V=\proj(\Vt), E=\proj(\Et),$ and $\wt:E\to K$ is defined via $\wt(\proj(\et))=\wt(\et)$. Since $\Vt$ is discrete, the set $V$ is finite, and the set $E$ is finite since the degrees of the vertices of $\Nt$ are finite.

For $\ut,\vt\in\Vt$, define 
\[\ee(\ut,\vt)=\sum_{\Pt\in\pathst(\ut,\vt)} \wt(\Pt),\]
where the sum is taken over the set $\pathst(\ut,\vt)$ of all directed paths $\Pt$ from $\ut$ to $\vt$ in $\Nt$. It is easy to see that for any two vertices $\ut,\vt\in\Vt$, the set $\pathst(\ut,\vt)$ is finite.

Recall that if $C$ is a directed cycle in $N$, we say that $C$ is \emph{simple} if each vertex of $N$ occurs in it at most once. Thus $C$ has a positive winding number $\wind(C)\geq 1$. 

\subsection{Linear recurrences for single paths}\label{sect:single_path}
In this section, we prove Theorem~\ref{thm:single_path}.

\def\lift{l}
\def\Lift{L}
\def\liftt{{\tilde y}}
\def\whirl{\wind}
For each vertex $v\in V$, choose its \emph{lift} $\lift(v)\in\Vt$ arbitrarily so that $\proj(\lift(v))=v$. Let $\Lift=\{\lift(v)\mid v\in V\}$ and denote by $\liftt_1,\dots,\liftt_p$ the elements of $\Lift$ in some order. 

Recall that the map $\wt$ takes values in some field $K$. For simplicity, we assume in this section that $K=\Q(\x)$ is the field of rational functions in variables $\x=(x_e)_{e\in E}$ and $\wt(\et)=x_{\proj(\et)}$. Define a $p\times p$ matrix $B(t)$ with entries $b_{ij}(t)\in\Z[\x;t^{\pm1}]$ as follows: 
\begin{equation}\label{eq:Bt}
b_{ij}(t)=\sum_{\et=(\liftt_i,\liftt_j+\ell\g)\in\Et} t^{\ell} \wt(\et),
\end{equation}
where  the sum is taken over all edges $\et\in\Et$ that connect $\liftt_i$ to some vertex $\liftt_j+\ell\g$ in $\liftt_j+\Z\g$. By part~\eqref{item:finite} of Definition~\ref{dfn:network}, there are only finitely many such edges in $\Nt$. 

We denote 
\begin{equation}\label{eq:Q_A}
Q_\Lift(t)=\det(\Id-B(t));\quad A(t)=(\Id-B(t))^{-1}=\sum_{r\geq 0} B(t)^r.
\end{equation}
Thus $Q_\Lift(t)$ is a certain polynomial in $\Z[\x;t^{\pm1}]$ and $A(t)$ is a matrix with entries in the field of rational functions $\Q(\x;t)$. It follows that the $(i,j)$-th entry $a_{ij}(t)$ of $A(t)$ equals
\[a_{ij}(t)=\sum_\Pt t^{\ell} \wt(\Pt),\]
where  the sum is taken over all paths that start at $\liftt_i$ and end at $\liftt_j+\ell\g$ for some $\ell\in\Z$. Note also that the series expansion~\eqref{eq:Q_A} of $A(t)$ is a well-defined power series because all the cycles in $N$ have positive winding number and thus $B(t)^r$ is divisible by large powers of $t$ for large values of $r$. This also shows that $\Id-B(t)$ is an invertible matrix because its inverse $A(t)$ is well defined. In fact, one can write down a formula for the determinant $Q_\Lift(t)$ of $\Id-B(t)$ explicitly:

\begin{proposition}\label{prop:Q_N}
	The polynomial $Q_\Lift(t)$ does not depend on the choice of $\Lift$. We have
	
	\begin{equation}\label{eq:Q_Lift}
		Q_\Lift(t)=(-t)^dQ_N(1/t),
	\end{equation}
	where the polynomial $Q_N(t)$ is given by~\eqref{eq:Q_N_intro}.
\end{proposition}
\begin{proof}
By definition, we have
\[Q_\Lift(t)=\det(\Id-B(t))=\sum_{\sigma\in\Sfr_n} \sgn(\sigma) H_\sigma,\]
where $H_\sigma=\prod_{i=1}^p (\Id-B(t))_{i,\sigma(i)}$. Decomposing $\sigma$ into cycles immediately yields~\eqref{eq:Q_Lift}. 
\end{proof}

\begin{example}\label{ex:network}
	Consider the network $\Nt$ in Figure~\ref{fig:network_example} together with a choice of $\Lift=\{\liftt_1=\ut,\liftt_2=\vt\}$.

We have
\[B(t)=\begin{pmatrix}
       	at& b+ct^{-1}\\
       	dt^2 & et
       \end{pmatrix};\quad \Id-B(t)=\begin{pmatrix}
       	1-at& -b-ct^{-1}\\
       	-dt^2 & 1-et
       \end{pmatrix}.\]
Thus 
\[Q_\Lift(t)=\det(\Id-B(t))=1-(a+e+cd)t+(ae-bd)t^2.\]
We see that in~\eqref{eq:Q_N_intro}, there is one (empty) $0$-cycle with weight $t^2$, four $1$-cycles with weights $-at,-et,-cdt, -bd$, and one $2$-cycle with weight $ae$, and thus 
\[Q_N(t)=t^2-(a+e+cd)t+(ae-bd).\]
\end{example}

\begin{proof}[Proof of Theorem~\ref{thm:single_path}]
	We can choose $\Lift$ so that $\liftt_i=\ut$ and $\liftt_j=\vt$ for some $1\leq i,j\leq p$. Let $(g_{ij}(\ell_0),g_{ij}(\ell_0+1),\dots)$ be a sequence of polynomials in $\x$ defined by 
	\[a_{ij}(t)=\sum_{\ell\geq \ell_0} g_{ij}(\ell) t^\ell.\]
	Then clearly we have $f(\ell)=g_{ij}(\ell)$ for $\ell\geq 0$. Since $a_{ij}(t)$ is a rational function with denominator $Q_\Lift(t)$, it follows by Propositions~\ref{prop:rational} and~\ref{prop:Q_N} that the sequence $f$ satisfies a linear recurrence with characteristic polynomial $Q_N(t)$.
\end{proof}

\section{Linear recurrences for tuples of paths}\label{sect:LGV}
In the previous section, we have established Theorem~\ref{thm:single_path} that gave a linear recurrence relation for single paths in $\Nt$. We now want to give a proof to Theorem~\ref{thm:recurrences} for $r$-paths in $\Nt$. 

\subsection{Background on plethysms}\label{sect:plethysm}

Consider a monic polynomial 
\[Q(t)\in K[t];\quad Q(t)=t^d-\alpha_1t^{d-1}+\dots+(-1)^d\alpha_d.\] 
It factors as a product of linear terms in the algebraic closure $\Kbar$ of $K$: $Q(t)=\prod_{j=1}^d (t-\gamma_j)$, where $\gamma_1,\dots,\gamma_d\in\Kbar$. The coefficient $\alpha_k$ of $Q(t)$ equals 
\[\alpha_k=e_k(\gamma_1,\dots,\gamma_d):=\sum_{1\leq i_1<i_2<\dots<i_k\leq d} \gamma_{i_1}\cdots\gamma_{i_k}.\]
Here $e_k$ is the $k$-th \emph{elementary symmetric polynomial}, see~\cite[Section~7.4]{EC2}. Recall also that the \emph{complete homogeneous symmetric polynomial} $h_k(\gamma_1,\dots,\gamma_d)$ is given by
\[h_k(\gamma_1,\dots,\gamma_d):=\sum_{1\leq i_1\leq i_2\leq \dots\leq i_k\leq d} \gamma_{i_1}\cdots\gamma_{i_k}.\]
Let $r\geq 1$ be an integer and consider the polynomials
\begin{eqnarray}
Q^\pleh{r}(t)&:=&\prod_{1\leq j_1\leq j_2\leq \dots\leq j_k\leq d} (t-\gamma_{i_1}\gamma_{i_2}\cdots\gamma_{i_k});\\
Q^\plee{r}(t)&:=&\prod_{1\leq j_1< j_2< \dots< j_k\leq d} (t-\gamma_{i_1}\gamma_{i_2}\cdots\gamma_{i_k}).\label{eq:plee_def}
\end{eqnarray}
The coefficient of $t^{D-k}$ (where $D$ is the degree of the corresponding polynomial) in $Q^\pleh{r}(t)$ equals $(-1)^ke_k[h_r](\gamma_1,\dots,\gamma_d)$, where the polynomial $e_k[h_r]$ denotes the \emph{plethysm} of $e_k$ with $h_r$. Similarly, the coefficient of $t^{D-k}$ in of $Q^\plee{r}(t)$ equals $(-1)^ke_k[e_r](\gamma_1,\dots,\gamma_d)$. We refer the reader to~\cite[Definition~A2.6]{EC2} for the definition of a plethysm. Loosely speaking, given two symmetric functions $f$ and $g$, the \emph{plethysm} $f[g]$ of $f$ with $g$ is a symmetric function obtained from $f$ by substituting the monomials of $g$ into the variables of $f$. 

Since $e_k[h_r]$ and $e_k[e_r]$ are again symmetric functions, they can be expressed as polynomials in the elementary symmetric functions $e_i$ by the fundamental theorem of symmetric functions~\cite[Theorem~7.4.4]{EC2}. Since $e_i(\gamma_1,\dots,\gamma_d)=\alpha_i$ is the coefficient of $t^{d-i}$ in $Q(t)$, we get that the coefficients of $Q^\pleh{r}(t)$ and $Q^\plee{r}(t)$ are polynomial expressions in the coefficients of $Q(t)$ and thus $Q^\pleh{r}(t),Q^\plee{r}(t)\in K[t]$ rather than $\Kbar[t]$. 

\begin{example}
	Let $Q(t)=t^3-\alpha_1t^2+\alpha_2t-\alpha_3$. Then 
	\begin{equation}\label{eq:pleth_ex}Q^\plee
{2}(t)=t^3-\alpha_2t^2+\alpha_3\alpha_1t-\alpha_3^3.
	\end{equation}
	For example, we get $\alpha_3\alpha_1t$ because we have the identity $e_2[e_2]=e_3e_1-e_4$ for symmetric functions and since $e_4(\gamma_1,\gamma_2,\gamma_3,0,0,\dots)=0$, we have 
	\[e_2[e_2](\gamma_1,\gamma_2,\gamma_3)=(e_3e_1)(\gamma_1,\gamma_2,\gamma_3).\]
	Indeed,
	\begin{eqnarray*}
	e_2[e_2](\gamma_1,\gamma_2,\gamma_3)&=&e_2(\gamma_1\gamma_2,\gamma_1\gamma_3,\gamma_2\gamma_3)=
	\gamma_1\gamma_2\gamma_1\gamma_3+\gamma_1\gamma_2\gamma_2\gamma_3+\gamma_1\gamma_3\gamma_2\gamma_3\\
	&=&\gamma_1\gamma_2\gamma_3(\gamma_1+\gamma_2+\gamma_3)=e_3(\gamma_1,\gamma_2,\gamma_3)e_1(\gamma_1,\gamma_2,\gamma_3).
	\end{eqnarray*}
	Similarly, 
	
	\begin{eqnarray*}Q^\pleh{2}(t)&=&t^6-(\alpha_1^2-\alpha_2)t^5+(\alpha_2\alpha_1^2-\alpha_2^2-\alpha_3\alpha_1)t^4\\	
	&-&(\alpha_2^2+\alpha_3\alpha_1^3-4\alpha_3\alpha_2\alpha_1+2\alpha_3^2)t^3 +(\alpha_3\alpha_2^2\alpha_1-\alpha_3^2\alpha1^2-\alpha_3^2\alpha_2)t^2\\
	&-&(\alpha_3^2\alpha_2^2-\alpha_3^3\alpha_1)t+\alpha_3^4.
	\end{eqnarray*}
\end{example}

\def\multiset#1#2{\ensuremath{\left(\kern-.3em\left(\genfrac{}{}{0pt}{}{#1}{#2}\right)\kern-.3em\right)}}

One easily observes that the degree of $Q^\pleh{r}(t)$ is $\multiset{d}{r}={d+r-1\choose r}$ while the degree of $Q^\plee{r}(t)$ is $d\choose r$. For example, $Q^\pleh{d}(t)$ has degree ${2d-1\choose d}$ but $Q^\plee{d}(t)=t-\alpha_d$ has degree $1$. It is obvious from the definitions that $Q^\pleh{r}(t)$ is always divisible by $Q^\plee{r}(t)$ in $\Kbar[t]$, and by an argument similar to the one above, their ratio even belongs to $K[t]$.

The following proposition follows from~\cite[Propositions~4.2.2 and~4.2.5]{EC1}.
\begin{proposition}\label{prop:product}
	Suppose that the sequences $f_1,f_2,\dots,f_r:\N\to K$ each satisfy a linear recurrence with characteristic polynomial $Q(t)$. Then the sequence $f=(f_1f_2\cdots f_r):\N\to K$ defined by 
	\[f(\ell)=f_1(\ell)f_2(\ell)\cdots f_r(\ell)\]
	satisfies a linear recurrence with characteristic polynomial $Q^\pleh{r}(t)$.
\end{proposition}

\newcommand{\Wedge}[2]{\Lambda^{#1}(#2)}
\newcommand{\Wedger}[1]{\Wedge{r}{#1}}

Let now $A$ be an $n\times n$ matrix over $K$. We view $A$ as a linear map $A:W\to W$ where $W=K^n$. Let $w_1,\dots,w_n$ be the basis of $W$. The \emph{$r$-th exterior power} of $W$ is the linear space $\Wedger W$ with basis 
\[\left\{w_{i_1}\wedge w_{i_2}\wedge\dots\wedge w_{i_r}\mid 1\leq i_1<i_2<\dots<i_r\leq n\right\}.\]
The \emph{$r$-th exterior power} of $A$ is the linear map $\Wedger  A:\Wedger W\to\Wedger W$ defined on every basis element by
\begin{equation}\label{eq:wedge_dfn}
\Wedger A(w_{i_1}\wedge w_{i_2}\wedge\dots\wedge w_{i_r})=(Aw_{i_1})\wedge(Aw_{i_2})\wedge\dots\wedge(Aw_{i_r}).
\end{equation}
Here we use the multilinearity and antisymmetry of $\wedge$ to expand the right hand side of~\eqref{eq:wedge_dfn} in the basis of $\Wedger W$. 

Equivalently, $\Wedger A$ is an ${n\choose r}\times{n\choose r}$ matrix whose rows and columns are indexed by $r$-element subsets of $[n]:=\{1,2,\dots,n\}$ and for two such subsets $I,J$, the corresponding entry of $\Wedger A$ equals the value of the minor of $A$ with rows $I$ and columns $J$.

The following fact is an application of the general theory of \emph{$\l$-rings}, see e.g.~\cite[Section~4]{Soule}:
\begin{proposition}\label{prop:wedge}
	Let $Q(t)\in K[t]$ be the characteristic polynomial of $A$. Then the characteristic polynomial of $\Wedger A$ is $Q^\plee{r}(t)$.
\end{proposition}

\begin{example}
	Let 
	\[A=\begin{pmatrix}
	    	a&0&d\\
	    	0&b&e\\
	    	0&f&c
	    \end{pmatrix}.\]
	Then the characteristic polynomial of $A$ is 
	\[Q(t)=\det(t\Id-A)=t^3-(a+b+c)t^2+(ab+bc+ac-ef)t-(abc-aef).\]
	Order the two-element subsets of $\{1,2,3\}$ as $(\{1,2\},\{1,3\},\{2,3\})$. Then for this ordering we can write
	\[\Wedge{2}{A}=\begin{pmatrix}
	            	ab&ae&-bd\\
	            	af&ac&-df\\
	            	0&0&bc-ef
	            \end{pmatrix}.\]
	 Thus the characteristic polynomial of $\Wedge{2}{A}$ is
	 \[\left(t-(bc-ef)\right)\left((t-ab)(t-ac)-a^2ef\right).\]
	 We encourage the reader to check that this polynomial equals to $Q^\plee{2}(t)$, i.e. the polynomial given by~\eqref{eq:pleth_ex} for 
	 \[\alpha_1=a+b+c;\quad \alpha_2=ab+bc+ac-ef;\quad \alpha_3=abc-aef.\]
\end{example}

Another well-known property of the exterior power is its multiplicativity: for two $n\times n$ matrices $A,B$, we have
\begin{equation}\label{eq:wedge_mult}
\Wedger{AB}=(\Wedger A)(\Wedger B).
\end{equation}
This is an obvious consequence of~\eqref{eq:wedge_dfn}.

\subsection{Lindstr\"om-Gessel-Viennot theorem}
We give a short background on the Lindstr\"om-Gessel-Viennot method introduced in~\cite{GV,Lindstrom} adapted to the case of cylindrical networks. 

\def\LGV{A}
Let $\Nt$ be a cylindrical network and consider two $r$-vertices $\ubft=(\ut_1,\ut_2,\dots,\ut_r)$ and $\vbft=(\vt_1,\vt_2,\dots,\vt_r)$ in $\Nt$. Let $\LGV(\ubft,\vbft)=(a_{ij})$ be the $r\times r$ matrix whose entries are given by
\[a_{ij}=\ee(\ut_i,\vt_j).\]

The Lindstr\"om-Gessel-Viennot theorem gives a combinatorial interpretation to the determinant of $\LGV(\ubft,\vbft)$. 

\begin{theorem}[\cite{GV}]\label{thm:LGV}
	We have
	\begin{equation}\label{eq:LGV}
	\det \LGV(\ubft,\vbft)=\sum_{\sigma\in\Sfr_r} \sgn(\sigma)\ee(\ubft,\sigma\vbft).
	\end{equation}
	
\end{theorem}

This theorem can be proven using a simple path-cancelling argument. We refer the reader to~\cite{GV} for the details.

\begin{example}
	For the network $\Nt$ in Figure~\ref{fig:network_example}, consider $\ubft=(\ut,\vt-\g)$ and $\vbft=(\ut+\g,\vt)$. Then we have
	\[\LGV(\ubft,\vbft)=\begin{pmatrix}
	       	a+cd& b+ce+ac+c^2d\\
	       	d & e+cd
	       \end{pmatrix}.\]
	 Therefore
	 \[\det\LGV(\ubft,\vbft)=(a+cd)(e+cd)-d(b+ce+ac+c^2d)=ae-bd.\]
	This agrees with~\eqref{eq:LGV} since there is just one $2$-path in $\pathst(\ubft,\vbft)$ with weight $ae$ and one $2$-path in $\pathst(\ubft,\vbft')$ with weight $bd$, where $\vbft'=(\vt,\ut+\g)=\sigma \vbft$ for $\sigma=(12)$ being the unique transposition in $\Sfr_2$.

\end{example}

\subsection{A recurrence for tuples of paths}
\begin{proof}[Proof of Theorem~\ref{thm:recurrences}, part~\eqref{item:thm:pleh}]
	Let $a_{ij}^\parr \ell$ be defined so that $\LGV(\ubft,\vbft_\ell)=(a_{ij}^\parr \ell)_{i,j=1}^r$. Then by Theorem~\ref{thm:single_path}, for all $1\leq i,j\leq r$, the sequence $(a_{ij}^\parr \ell)_{\ell=0}^\infty$ satisfies a linear recurrence with characteristic polynomial $Q_N(t)$. Since by Theorem~\ref{thm:LGV}, $f(\ell)$ is a linear combination of $r$-term products of $a_{ij}^\parr \ell$'s, the result follows by Propositions~\ref{prop:product} and~\ref{prop:sum}.
\end{proof}

Recall that the polynomial $Q_N^\plee{r}(t)$ in general divides the polynomial $Q_N^\pleh{r}(t)$ and has a much smaller degree for large $r$. We now would like to prove the second part of Theorem~\ref{thm:recurrences}.

\begin{definition}
	We say that a cylindrical network $\Nt$ is \emph{local} if one can choose a lifting $\Lift$ of $V$ to $\Nt$ so that the entries of the matrix $B(t)$ given by~\eqref{eq:Bt} would be \emph{linear polynomials} in $t$, i.e. $B(t)=C+tD$ for some matrices $C,D$ whose entries do not depend on $t$. In other words, $\Nt$ is \emph{local} if there exists a lifting $\Lift$ of $V$ such that every edge that starts at $\Lift$ ends either at $\Lift$ or at $\Lift+\g$. In this case, we say that $\Lift$ is a \emph{local lifting} for $\Nt$.
\end{definition}

We will later see in Proposition~\ref{prop:local} that every planar cylindrical network $\Nt$ is local. However, this property is more general:
\begin{example}
	The cylindrical network $\Nt$ in Figure~\ref{fig:network_example} is local. Indeed, if we choose a different lift $\Lift=\{\ut+\g,\vt\}$ then the matrix $B(t)$ becomes
	\[B(t)=\begin{pmatrix}
	       	at & c+bt\\
	       	dt & et
	       \end{pmatrix}.\]
Thus we have 
\[B(t)=C+tD,\quad\text{where}\quad C=\begin{pmatrix}
	       	0 & c\\
	       	0 &0
	       \end{pmatrix},\quad D=\begin{pmatrix}
	       	a & b\\
	       	d & e
	       \end{pmatrix}.\]
\end{example}

\begin{theorem}\label{thm:tuples_plee}
	Suppose that $\Nt$ is a local network and let $\Lift$ be a local lifting for $\Nt$. Let $\ubft$ and $\vbft$ be two $r$-vertices in $\Nt$. Then the sequence $f(\ell)$ from Theorem~\ref{thm:recurrences} satisfies a linear recurrence with characteristic polynomial $Q_N^\plee{r}(t)$.
\end{theorem}
\begin{proof}
	Let $B(t)=C+tD$. Then since $\Nt$ is acyclic, the matrix $C$ has to be nilpotent, and therefore $(\Id-C)^{-1}=\Id+C+\dots+C^l$ for some $l\geq 0$.
	We are interested in the matrix $A(t)=(\Id-B(t))^{-1}=(\Id-C-tD)^{-1}$. Let 
	\begin{equation}\label{eq:S}
	S=(\Id-C)^{-1} D,
      \end{equation}
      then we have
	\[(\Id-C-tD)=(\Id-C)(\Id-t(\Id-C)^{-1}D)=(\Id-C)(\Id-tS),\]
	and thus 
	\begin{equation}\label{eq:AtS}
	A(t)=(\Id-tS)^{-1} (\Id-C)^{-1}.
	\end{equation}
	In particular, 
	\begin{equation}\label{eq:det_S}
	Q_\Lift(t)=\det(\Id-B(t))=\det(\Id-tS),
	\end{equation} 
	since $\det(\Id-C)=1$ for any nilpotent matrix $C$. Thus $Q_N(t)=t^dQ_\Lift(1/t)$ is the characteristic polynomial $\det(t\Id-S)$ of $S$, possibly multiplied by a power of $t$.
	
	Let us denote
	\[A(t)=\sum_{\ell\geq \ell_0} A^\parr \ell t^\ell.\]
	Thus the entry $a_{ij}^\parr \ell$ of $A^\parr \ell$ counts the paths from $\liftt_i$ to $\liftt_j+\ell\g$ in $\Nt$. Using~\eqref{eq:AtS}, we get  
	\[A^\parr \ell= S^\ell (\Id-C)^{-1}.\]
	
	We first consider the case when all the vertices of $\ubft$ and of $\vbft$ belong to $\Lift$. Consider the sequence $f:\N\to K$ from Theorem~\ref{thm:recurrences}. By Theorem~\ref{thm:LGV}, $f(\ell)$ is a certain $r\times r$ minor of the matrix $A^\parr \ell$, or equivalently, it is a certain entry of the matrix $\Wedger{A^\parr\ell}$. Using the multiplicativity~\eqref{eq:wedge_mult} of the exterior power, we get
	\[\Wedger{A^\parr\ell}=(\Wedger S)^\ell (\Wedger{(\Id-C)^{-1}}).\]
	\def\wedger{\parr{\wedge^r }}
	Define the matrix 
	\[A^\wedger(t)=\sum_{\ell\geq \ell_0} t^\ell(\Wedger{A^\parr\ell})=(\Id-t(\Wedger S))^{-1}(\Wedger{(\Id-C)^{-1}}).\]
	The generating function $\sum_\ell t^\ell f(\ell)$ for $f$ appears as an entry in $A^\wedger(t)$ and therefore is a rational function with denominator $\det(\Id-t(\Wedger S))$ which is just the palindromic image of the characteristic polynomial of $\Wedger S$. By Proposition~\ref{prop:wedge}, this characteristic polynomial equals $Q_N^\plee{r}(t)$ since $Q_N(t)$ is the characteristic polynomial of $S$ by~\eqref{eq:det_S}. 
	
	We are done with the case when all the vertices of $\ubft$ and $\vbft$ belong to $\Lift$. We are going to deduce the general case as a simple consequence. Consider any $r$-path $\Pbft=(\Pt_1,\dots,\Pt_r)$ from $\ubft$ to $\sigma\vbft_\ell$. We claim that it can be decomposed as a concatenation of three $r$-paths $\Pbft^\parr1,\Pbft^\parr2,\Pbft^\parr3$ in such a way that the endpoints $\ubft'$ and $\vbft'$ of $\Pbft^\parr2$ satisfy the following: all vertices of $\ubft'$ belong to $\Lift+\ell_1\g$ and all vertices of $\vbft'$ belong to $\Lift+(\ell+\ell_2)\g$, for some constants $\ell_1$ and $\ell_2$. Indeed, consider a path $\Pt_i$ of $\Pbft$ from $\ut_i$ to $\vt_{\sigma(i)}+\ell\g$. Let $(\pt_1,\pt_2,\dots,\pt_M)$ be the vertices of $\Pt_i$, and define $z(\pt_i)$ to be the unique integer $z$ so that $\pt_i\in \Lift+z\g$. Then the sequence $(z(\pt_i))_{i=1}^M$ is weakly increasing and \emph{Lipschitz}, i.e. 
	\[z(\pt_i)\leq z(\pt_{i+1})\leq z(\pt_i)+1.\]
	Let 
	\[z_0=\min\{z(\ut_i)\mid 1\leq i\leq r\},\quad z_0'=\max\{z(\ut_i)\mid 1\leq i\leq r\};\] \[z_1=\min\{z(\vt_i)\mid 1\leq i\leq r\},\quad z_1'=\max\{z(\vt_i)\mid 1\leq i\leq r\}.\]
	Then we can define the decomposition of $\Pt_i$ into three parts as follows: 
	\begin{itemize}
		\item $\Pt_i^\parr1$ consists of vertices $\pt_i$ for which $z_0\leq z(\pt_i)< z_0'$;
		\item $\Pt_i^\parr2$ consists of vertices $\pt_i$ for which $z_0'\leq z(\pt_i)\leq z_1+\ell$;
		\item $\Pt_i^\parr3$ consists of vertices $\pt_i$ for which $z_1+\ell< z(\pt_i)\leq z_1'+\ell$.
	\end{itemize}
	Indeed, if we set $\ell_1=z_0'$ and $\ell_2=z_1$ then $\Pt_i^\parr2$ starts at a vertex from $\Lift+\ell_1\g$ and ends at a vertex from $\Lift+(\ell+\ell_2)\g$. It is clear that the total number of choices for $\Pbft^\parr1$ and $\Pbft^\parr3$ is finite, and for each such choice the corresponding sequence $f'$ for the points $\ubft'$ and $\vbft'$ satisfies a linear recurrence with characteristic polynomial $Q_N^\plee{r}$. Thus we are done by Proposition~\ref{prop:sum}.
\end{proof}

Even though we can prove Theorem~\ref{thm:tuples_plee} for only local cylindrical networks, we suspect that it holds for \emph{any} cylindrical network:
\begin{conjecture}
	If $\Nt$ is a cylindrical network and $\ubft,\vbft$ are two $r$-vertices then the sequence $f(\ell)$ from Theorem~\ref{thm:recurrences} satisfies a linear recurrence with characteristic polynomial $Q^\plee{r}_N(t)$ for all but finitely many $\ell$.
\end{conjecture}

\subsection{Planar cylindrical networks}\label{sect:planar}
Even though Theorem~\ref{thm:LGV} holds for all networks, in the majority of the situations it is applied to \emph{planar} networks. 

\begin{definition}\label{dfn:network_planar}
	We say that a cylindrical network $\Nt$ is \emph{planar} if for each edge $\et=(\ut,\vt)\in\Et$ there is an embedding $h_\et:[0,1]\to\TS$ satisfying the following properties: 
	\begin{itemize}
		\item $h_{\et}(0)=\ut$ and $h_{\et}(1)=\vt$;
		\item the interior $h_\et((0,1))$ of $\et$ does not intersect the image of any other $h_{\et'}$ for $\et'\in\Et$ and does not contain any vertex of $\Nt$;
		\item shift-invariance: $h_{\et+\g}(t)=h_\et(t)+\g$ for all $\et\in\Et$ and $t\in[0,1]$.
	\end{itemize}
\end{definition}

First, we prove a simpler formula~\eqref{eq:Q_N_planar_intro} for $Q_N(t)$ in the planar case:

\begin{proposition}\label{prop:Q_N_planar}
	If $\Nt$ is a planar cylindrical network then $Q_N(t)$ is given by~\eqref{eq:Q_N_planar_intro}, that is,
	\begin{equation*}
	Q_N(t)=\sum_{r\geq 0}(-t)^{d-r}\sum_{\Cbf\in\Cyc^r(N)} \wt(\Cbf).
	\end{equation*}
\end{proposition}
\begin{proof}
	As we have already mentioned in Section~\ref{sect:main}, it suffices to show every simple cycle in a planar cylindrical network has winding number $1$. Indeed, every simple cycle in $N$ represents a non-self-intersecting loop in the cylinder $\TC$ with a positive winding number which therefore has to be equal to $1$.
\end{proof}

By Theorem~\ref{thm:tuples_plee}, in order to prove the second part of Theorem~\ref{thm:recurrences} it suffices to show the following:
\begin{proposition}\label{prop:local}
	Every planar cylindrical network $\Nt$ is local.
\end{proposition}
\begin{proof}
	Showing that $\Nt$ is local amounts to constructing a function $z:\Vt\to\Z$ such that $z(\vt+\g)=z(\vt)+1$ and such that for every edge $\et=(\ut,\vt)$ of $\Nt$ we have
	\begin{equation}\label{eq:lipschitz}
	z(\ut)\leq z(\vt)\leq z(\ut)+1.
	\end{equation}
	Given such a function, it is clear that the set $\Lift:=z^{-1}(0)$ is a local lift for $\Nt$. We prove the existence of $z$ by induction on the number of vertices in $N$. Throughout, we assume that $\Nt$ is connected since it suffices to prove the result for any connected component of $\Nt$. The base case is when there is just one vertex $v$ in $N$, so let $\vt_0$ be any of its lifts and define $\vt_\ell=\vt_0+\ell\g, \ell\in Z$ to be the remaining lifts of $v$. We claim that all the edges coming out of $\vt_0$ end at $\vt_1$. Indeed, suppose that $\et=(\vt_0,\vt_\ell)$ is an edge that ends at some $\vt_\ell$. Then its projection is a simple loop in the cylinder $\TC$ with winding number $\ell$. As we have noted earlier, this can only happen when $\ell=1$. Thus we get $\ell=1$ for every edge $\et=(\vt_0,\vt_\ell)$, and it suffices to set $z(\vt_k)=k$ for all $k\in\Z$ to complete the base case.
	
	To do the induction step, consider a general connected planar cylindrical network $\Nt$. An edge $\et=(\ut,\vt)$ is called a \emph{cover relation} in $\Nt$ if $\proj(\ut)\neq\proj(\vt)$ and there is no path in $\Nt$ with at least two edges that starts at $\ut$ and ends at $\vt$. It is clear that if $N$ is connected and has at least two vertices then such a cover relation exists in $\Nt$ by part~\eqref{item:finite} of Definition~\ref{dfn:network}. So let $\et=(\ut,\vt)$ be a cover relation in $\Nt$. We can contract this edge $\et$ and all of its shifts, and this operation produces a smaller connected planar cylindrical network $\Nt'$ for which we already have a function $z'$ satisfying~\eqref{eq:lipschitz}. Let us put $z(\wtt)=z'(\wtt')$ for any vertex $\wtt\in\Vt$, where $\wtt'$ is the vertex in $\Nt'$ corresponding to $\wtt$. It is clear that this defines a function $z:\Vt\to\Z$ satisfying~\eqref{eq:lipschitz}. We are done with the proof.
\end{proof}


\section{Applications}\label{sect:applications}

\subsection{A recurrence for Schur polynomials}\label{sect:Schur}
Since the second part of Theorem~\ref{thm:recurrences} holds for any planar cylindrical network, we first demonstrate how it yields new results in one especially well-studied case of \emph{Schur polynomials}. 

Fix two integers $n,m\geq 1$ and consider the following planar cylindrical network $\Nt_{n,m}$. Its set of vertices $\Vt$ is identified with $\Z\times [n]$. Let $\x=(x_1,\dots,x_n)$ be a family of indeterminates. A vertex $\vt=(i,j)$ with $j\in[n]=\{1,2,\dots,n\}$ is connected to $(i+1,j)$ by an edge of weight $x_j$ and, assuming $j<n$, it is connected to $(i,j+1)$ by an edge of weight $1$. We set $\g:=(m,0)$. This defines the network $\Nt_{n,m}$ whose projection $N_{n,m}$ is a cylinder $[n]\times \Z_m$ of height $n$ and width $m$. See Figure~\ref{fig:Schur}.

  \def\aa{7}
  \def\bb{4}
\begin{figure}

\scalebox{0.7}{
\begin{tikzpicture}[scale=1.5]
\tikzset{myptr/.style={decoration={markings,mark=at position 1 with %
    {\arrow[scale=1.5,>=stealth]{>}}},postaction={decorate}}}
\draw[myptr,line width=0.25mm] (1,4) to node[midway,above] {$\g=(3,0)$} (4,4);
\node[draw,ellipse] (node0x0) at (0.00,0.00) {};
\node[draw,ellipse] (node0x1) at (0.00,1.00) {};
\node[draw,ellipse] (node0x2) at (0.00,2.00) {};
\node[draw,ellipse] (node0x3) at (0.00,3.00) {};
\node[draw,ellipse] (node1x0) at (1.00,0.00) {};
\node[draw,ellipse] (node1x1) at (1.00,1.00) {};
\node[draw,ellipse] (node1x2) at (1.00,2.00) {};
\node[draw,ellipse] (node1x3) at (1.00,3.00) {};
\node[draw,ellipse] (node2x0) at (2.00,0.00) {};
\node[draw,ellipse] (node2x1) at (2.00,1.00) {};
\node[draw,ellipse] (node2x2) at (2.00,2.00) {};
\node[draw,ellipse] (node2x3) at (2.00,3.00) {};
\node[draw,ellipse] (node3x0) at (3.00,0.00) {};
\node[draw,ellipse] (node3x1) at (3.00,1.00) {};
\node[draw,ellipse] (node3x2) at (3.00,2.00) {};
\node[draw,ellipse] (node3x3) at (3.00,3.00) {};
\node[draw,ellipse] (node4x0) at (4.00,0.00) {};
\node[draw,ellipse] (node4x1) at (4.00,1.00) {};
\node[draw,ellipse] (node4x2) at (4.00,2.00) {};
\node[draw,ellipse] (node4x3) at (4.00,3.00) {};
\node[draw,ellipse] (node5x0) at (5.00,0.00) {};
\node[draw,ellipse] (node5x1) at (5.00,1.00) {};
\node[draw,ellipse] (node5x2) at (5.00,2.00) {};
\node[draw,ellipse] (node5x3) at (5.00,3.00) {};
\node[draw,ellipse] (node6x0) at (6.00,0.00) {};
\node[draw,ellipse] (node6x1) at (6.00,1.00) {};
\node[draw,ellipse] (node6x2) at (6.00,2.00) {};
\node[draw,ellipse] (node6x3) at (6.00,3.00) {};
\draw[myptr] (node0x0) to node[midway,right] {$1$} (node0x1);
\draw[myptr] (node0x0) to node[midway,above] {$x_1$} (node1x0);
\draw[myptr] (node0x1) to node[midway,right] {$1$} (node0x2);
\draw[myptr] (node0x1) to node[midway,above] {$x_2$} (node1x1);
\draw[myptr] (node0x2) to node[midway,right] {$1$} (node0x3);
\draw[myptr] (node0x2) to node[midway,above] {$x_3$} (node1x2);
\draw[myptr] (node0x3) to node[midway,above] {$x_4$} (node1x3);
\draw[myptr] (node1x0) to node[midway,right] {$1$} (node1x1);
\draw[myptr] (node1x0) to node[midway,above] {$x_1$} (node2x0);
\draw[myptr] (node1x1) to node[midway,right] {$1$} (node1x2);
\draw[myptr] (node1x1) to node[midway,above] {$x_2$} (node2x1);
\draw[myptr] (node1x2) to node[midway,right] {$1$} (node1x3);
\draw[myptr] (node1x2) to node[midway,above] {$x_3$} (node2x2);
\draw[myptr] (node1x3) to node[midway,above] {$x_4$} (node2x3);
\draw[myptr] (node2x0) to node[midway,right] {$1$} (node2x1);
\draw[myptr] (node2x0) to node[midway,above] {$x_1$} (node3x0);
\draw[myptr] (node2x1) to node[midway,right] {$1$} (node2x2);
\draw[myptr] (node2x1) to node[midway,above] {$x_2$} (node3x1);
\draw[myptr] (node2x2) to node[midway,right] {$1$} (node2x3);
\draw[myptr] (node2x2) to node[midway,above] {$x_3$} (node3x2);
\draw[myptr] (node2x3) to node[midway,above] {$x_4$} (node3x3);
\draw[myptr] (node3x0) to node[midway,right] {$1$} (node3x1);
\draw[myptr] (node3x0) to node[midway,above] {$x_1$} (node4x0);
\draw[myptr] (node3x1) to node[midway,right] {$1$} (node3x2);
\draw[myptr] (node3x1) to node[midway,above] {$x_2$} (node4x1);
\draw[myptr] (node3x2) to node[midway,right] {$1$} (node3x3);
\draw[myptr] (node3x2) to node[midway,above] {$x_3$} (node4x2);
\draw[myptr] (node3x3) to node[midway,above] {$x_4$} (node4x3);
\draw[myptr] (node4x0) to node[midway,right] {$1$} (node4x1);
\draw[myptr] (node4x0) to node[midway,above] {$x_1$} (node5x0);
\draw[myptr] (node4x1) to node[midway,right] {$1$} (node4x2);
\draw[myptr] (node4x1) to node[midway,above] {$x_2$} (node5x1);
\draw[myptr] (node4x2) to node[midway,right] {$1$} (node4x3);
\draw[myptr] (node4x2) to node[midway,above] {$x_3$} (node5x2);
\draw[myptr] (node4x3) to node[midway,above] {$x_4$} (node5x3);
\draw[myptr] (node5x0) to node[midway,right] {$1$} (node5x1);
\draw[myptr] (node5x0) to node[midway,above] {$x_1$} (node6x0);
\draw[myptr] (node5x1) to node[midway,right] {$1$} (node5x2);
\draw[myptr] (node5x1) to node[midway,above] {$x_2$} (node6x1);
\draw[myptr] (node5x2) to node[midway,right] {$1$} (node5x3);
\draw[myptr] (node5x2) to node[midway,above] {$x_3$} (node6x2);
\draw[myptr] (node5x3) to node[midway,above] {$x_4$} (node6x3);
\draw[myptr] (node6x0) to node[midway,right] {$1$} (node6x1);
\draw[myptr] (node6x1) to node[midway,right] {$1$} (node6x2);
\draw[myptr] (node6x2) to node[midway,right] {$1$} (node6x3);
\draw[myptr] (-1,0) to node[midway,above] {$x_1$} (node0x0);
\draw[myptr] (node6x0) to node[midway,above] {$x_1$} (7,0);
\draw[myptr] (-1,1) to node[midway,above] {$x_2$} (node0x1);
\draw[myptr] (node6x1) to node[midway,above] {$x_2$} (7,1);
\draw[myptr] (-1,2) to node[midway,above] {$x_3$} (node0x2);
\draw[myptr] (node6x2) to node[midway,above] {$x_3$} (7,2);
\draw[myptr] (-1,3) to node[midway,above] {$x_4$} (node0x3);
\draw[myptr] (node6x3) to node[midway,above] {$x_4$} (7,3);
\node[anchor=east,scale=2] (dots1) at (-1.20,1.50) {$\dots$};
\node[anchor=west,scale=2] (dots2) at (7.20,1.50) {$\dots$};
\end{tikzpicture}}

\caption{\label{fig:Schur}The network $\Nt_{n,m}$ for $n=4$ and $m=3$.}
\end{figure}
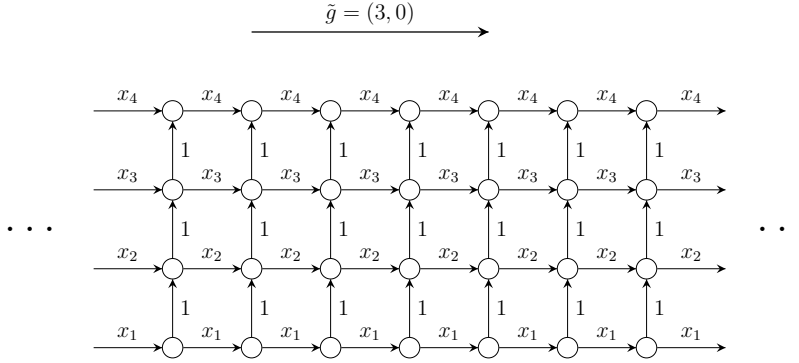

It is clear from~\eqref{eq:Q_N_planar_intro} that the polynomial $Q_{N_{n,m}}(t)$ is equal to 
\[Q_{N_{n,m}}(t)=(t-x_1^m)(t-x_2^m)\cdots(t-x_n^m).\]

The polynomial $Q_{N_{n,m}}(t)$ has degree $d=n$ and the coefficient of $t^{d-r}$ in it equals $(-1)^re_r(x_1^m,\dots,x_n^m)$.

Let now $1\leq r\leq n$ be an integer. A \emph{partition with at most $r$ parts} is a weakly decreasing sequence $\lambda=(\lambda_1\geq\lambda_2\geq\dots\geq \lambda_r\geq 0)$ of $r$ nonnegative integers. To each such partition one can associate a certain symmetric polynomial $s_\lambda(x_1,x_2,\dots,x_n)$ called the \emph{Schur polynomial} of $\l$, see~\cite[Section~7.10]{EC2} for the definition. Given a partition $\lambda$ with at most $r$ parts, we introduce an $r$-vertex $\vbft(\l)=((\l_r+1,n),(\l_{r-1}+2,n),\dots,(\l_1+r,n))$ in $\Nt_{n,m}$. We also fix an $r$-vertex $\ubft=((1,1),(2,1),\dots,(r,1))$. In terms of our network $\Nt_{n,m}$, the Schur polynomial $s_\l(x_1,\dots,x_n)$ equals
\[s_\l(x_1,x_2,\dots,x_n)=\ee(\ubft,\vbft(\l)).\]

We define the sum of two partitions $\l+\m$ componentwise, that is, $(\l+\m)_i=\l_i+\m_i$, similarly, $\ell\mu$ denotes a partition with $(\ell\mu)_i=\ell\mu_i$. We fix a partition $\m=m^r:=(m,m,\dots,m)$ with $r$ parts, in other words, the \emph{Young diagram} of $\mu$ is an $m\times r$ rectangle. 

\begin{theorem}\label{thm:Schur}
	The sequence $f(\ell)=s_{\l+\ell\m}(x_1,\dots,x_n)$ satisfies a linear recurrence for all but finitely many $\ell$ with characteristic polynomial $Q_{N_{n,m}}^\plee{r}(t)$.
\end{theorem}
\begin{proof}
	By Proposition~\ref{prop:local}, the planar cylindrical network $\Nt$ is local. It is clear that the only permutation $\sigma$ that gives a non-zero contribution in~\eqref{eq:LGV} is the identity (in this case the $r$-vertices $\ubft$ and $\vbft$ are called \emph{non-permutable}, see~\cite{GV}). Thus the sequence $f(\ell)$ is exactly a sequence satisfying the assumptions of Theorem~\ref{thm:tuples_plee} and the result follows.
\end{proof}

The polynomial $Q_{N_{n,m}}^\plee{r}(t)$ has $n\choose r$ roots of the form $(x_{i_1}x_{i_2}\cdots x_{i_r})^m$ for all $1\leq i_1<i_2<\dots<i_r\leq n$. 

\begin{remark}
	Theorem~\ref{thm:Schur} extends trivially to the sequence $f(\ell)=s_{(\l+\ell\m)/\kappa}(x_1,\dots,x_n)$ of \emph{skew Schur polynomials} for any partitions $\l,\kappa$ with at most $r$ parts. More generally, given four partitions $\l,\m,\nu,\kappa$, one can consider a sequence $f(\ell)=s_{(\l+\ell\m)/(\kappa+\ell\nu)}(x_1,\dots,x_n)$. It was shown by Alexandersson~\cite{Per} that this sequence satisfies (for all but finitely many $\ell$ under a mild condition on $\l,\m,\nu,\kappa$) a linear recurrence with characteristic polynomial 
	\[Q_{\mu/\nu}(t)=\prod_T (t-\x^{\wt(T)}),\]
	where $T$ runs over all semistandard Young tableaux of shape $\m/\nu$ with entries in $[n]$. In~\cite[Conjecture~24]{Per}, he conjectured that $f$ satisfies a shorter linear recurrence with characteristic polynomial 
	\[P_{\mu/\nu}(t)=\prod_{\wbf\in W} (t-\x^{\wbf}),\] 
	where $W$ is a certain subset of $\N^n$ defined explicitly. In the case $\mu=m^r,\nu=0^r=\emptyset$, it is easy to see that the polynomial $Q^\plee{r}_{N_{n,m}}(t)$ equals the polynomial $P_{\mu/\nu}(t)$ and they both have degree $n\choose r$. This number is much smaller\footnote{Indeed, by the Weyl dimension formula~\cite[Corollary 7.21.4]{EC2}, $Q_{\mu/\nu}(t)$ has degree $O(n^{mr})$ as $n\to\infty$. On the other hand, $Q^\plee{r}_{N_{n,m}}(t)$ has degree ${n\choose r}=O(n^r)$ as $n\to\infty$.} than the degree of $Q_{\mu/\nu}(t)$ which is the number of semistandard Young tableaux of rectangular shape $m^r$ with entries in $[n]$. Thus Theorem~\ref{thm:Schur} gives in this case a new, shorter linear recurrence which is conjectured in~\cite[Conjecture~24]{Per} to be the minimal recurrence for the sequence $f(\ell)=s_{(\l+\ell\m)/\kappa}(x_1,\dots,x_n)$.
\end{remark}

\subsection{Lozenge tilings, plane partitions, and Carlitz $q$-Fibonacci polynomials}\label{sect:lozenge}

\def\sk{\Y}
\def\sksh{\l/\m}
\def\Y{{\mathcal{Y}}}
\def\a{\alpha}
\def\b{\beta}
\def\bij{\Pbft}
Given a partition $\l=(\l_1,\l_2,\dots,\l_m)$, define its \emph{Young diagram} $\Y(\l)$ to be the set of $1\times 1$ ``boxes'' in the plane centered at $(i,j)$ for every pair $(i,j)$ satisfying $1\leq i\leq m$ and $1\leq j\leq \l_i$. We use the \emph{English notation} and draw the boxes of $\Y(\l)$ using matrix coordinates, see e.g. Figure~\ref{fig:lozenges}~(a). Consider two partitions $\l,\m$ such that $\Y(\m)\subset\Y(\l)$. Then we define the \emph{skew shape} $\Y(\l/\m)$ to be the set-theoretic difference $\Y(\l)\setminus\Y(\m)$.

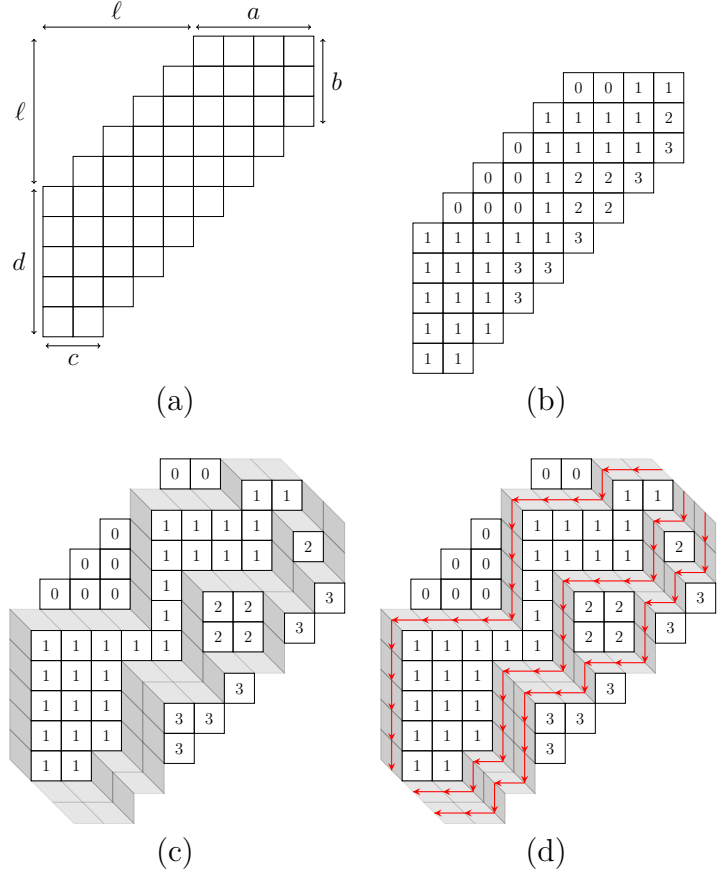
\begin{figure}
\centering

\begin{tabular}{cc}
\scalebox{0.4}{
\begin{tikzpicture}[scale=1.0,yscale=-1]
\draw [] (5.0,0.0) rectangle (6.0,1.0); \draw [] (6.0,0.0) rectangle (7.0,1.0); \draw [] (7.0,0.0) rectangle (8.0,1.0); \draw [] (8.0,0.0) rectangle (9.0,1.0); \draw [] (4.0,1.0) rectangle (5.0,2.0); \draw [] (5.0,1.0) rectangle (6.0,2.0); \draw [] (6.0,1.0) rectangle (7.0,2.0); \draw [] (7.0,1.0) rectangle (8.0,2.0); \draw [] (8.0,1.0) rectangle (9.0,2.0); \draw [] (3.0,2.0) rectangle (4.0,3.0); \draw [] (4.0,2.0) rectangle (5.0,3.0); \draw [] (5.0,2.0) rectangle (6.0,3.0); \draw [] (6.0,2.0) rectangle (7.0,3.0); \draw [] (7.0,2.0) rectangle (8.0,3.0); \draw [] (8.0,2.0) rectangle (9.0,3.0); \draw [] (2.0,3.0) rectangle (3.0,4.0); \draw [] (3.0,3.0) rectangle (4.0,4.0); \draw [] (4.0,3.0) rectangle (5.0,4.0); \draw [] (5.0,3.0) rectangle (6.0,4.0); \draw [] (6.0,3.0) rectangle (7.0,4.0); \draw [] (7.0,3.0) rectangle (8.0,4.0); \draw [] (1.0,4.0) rectangle (2.0,5.0); \draw [] (2.0,4.0) rectangle (3.0,5.0); \draw [] (3.0,4.0) rectangle (4.0,5.0); \draw [] (4.0,4.0) rectangle (5.0,5.0); \draw [] (5.0,4.0) rectangle (6.0,5.0); \draw [] (6.0,4.0) rectangle (7.0,5.0); \draw [] (0.0,5.0) rectangle (1.0,6.0); \draw [] (1.0,5.0) rectangle (2.0,6.0); \draw [] (2.0,5.0) rectangle (3.0,6.0); \draw [] (3.0,5.0) rectangle (4.0,6.0); \draw [] (4.0,5.0) rectangle (5.0,6.0); \draw [] (5.0,5.0) rectangle (6.0,6.0); \draw [] (0.0,6.0) rectangle (1.0,7.0); \draw [] (1.0,6.0) rectangle (2.0,7.0); \draw [] (2.0,6.0) rectangle (3.0,7.0); \draw [] (3.0,6.0) rectangle (4.0,7.0); \draw [] (4.0,6.0) rectangle (5.0,7.0); \draw [] (0.0,7.0) rectangle (1.0,8.0); \draw [] (1.0,7.0) rectangle (2.0,8.0); \draw [] (2.0,7.0) rectangle (3.0,8.0); \draw [] (3.0,7.0) rectangle (4.0,8.0); \draw [] (0.0,8.0) rectangle (1.0,9.0); \draw [] (1.0,8.0) rectangle (2.0,9.0); \draw [] (2.0,8.0) rectangle (3.0,9.0); \draw [] (0.0,9.0) rectangle (1.0,10.0); \draw [] (1.0,9.0) rectangle (2.0,10.0); \draw[<->,line width=0.25pt] (0.00,-0.30) to node[midway,above,scale=2] {$\ell$} (4.90,-0.30);
\draw[<->,line width=0.25pt] (-0.30,0.10) to node[midway,left,scale=2] {$\ell$} (-0.30,4.90);
\draw[<->,line width=0.25pt] (5.10,-0.30) to node[midway,above,scale=2] {$a$} (8.90,-0.30);
\draw[<->,line width=0.25pt] (9.30,0.10) to node[midway,right,scale=2] {$b$} (9.30,2.90);
\draw[<->,line width=0.25pt] (0.10,10.30) to node[midway,below,scale=2] {$c$} (1.90,10.30);
\draw[<->,line width=0.25pt] (-0.30,5.10) to node[midway,left,scale=2] {$d$} (-0.30,9.90);
\end{tikzpicture}}
&
\scalebox{0.4}{
\begin{tikzpicture}[scale=1.0,yscale=-1]
\draw [] (5.0,0.0) rectangle (6.0,1.0); \node[anchor=center,scale=1.3] (n0x5) at (5.50,0.50) {$0$};
\draw [] (6.0,0.0) rectangle (7.0,1.0); \node[anchor=center,scale=1.3] (n0x6) at (6.50,0.50) {$0$};
\draw [] (7.0,0.0) rectangle (8.0,1.0); \node[anchor=center,scale=1.3] (n0x7) at (7.50,0.50) {$1$};
\draw [] (8.0,0.0) rectangle (9.0,1.0); \node[anchor=center,scale=1.3] (n0x8) at (8.50,0.50) {$1$};
\draw [] (4.0,1.0) rectangle (5.0,2.0); \node[anchor=center,scale=1.3] (n1x4) at (4.50,1.50) {$1$};
\draw [] (5.0,1.0) rectangle (6.0,2.0); \node[anchor=center,scale=1.3] (n1x5) at (5.50,1.50) {$1$};
\draw [] (6.0,1.0) rectangle (7.0,2.0); \node[anchor=center,scale=1.3] (n1x6) at (6.50,1.50) {$1$};
\draw [] (7.0,1.0) rectangle (8.0,2.0); \node[anchor=center,scale=1.3] (n1x7) at (7.50,1.50) {$1$};
\draw [] (8.0,1.0) rectangle (9.0,2.0); \node[anchor=center,scale=1.3] (n1x8) at (8.50,1.50) {$2$};
\draw [] (3.0,2.0) rectangle (4.0,3.0); \node[anchor=center,scale=1.3] (n2x3) at (3.50,2.50) {$0$};
\draw [] (4.0,2.0) rectangle (5.0,3.0); \node[anchor=center,scale=1.3] (n2x4) at (4.50,2.50) {$1$};
\draw [] (5.0,2.0) rectangle (6.0,3.0); \node[anchor=center,scale=1.3] (n2x5) at (5.50,2.50) {$1$};
\draw [] (6.0,2.0) rectangle (7.0,3.0); \node[anchor=center,scale=1.3] (n2x6) at (6.50,2.50) {$1$};
\draw [] (7.0,2.0) rectangle (8.0,3.0); \node[anchor=center,scale=1.3] (n2x7) at (7.50,2.50) {$1$};
\draw [] (8.0,2.0) rectangle (9.0,3.0); \node[anchor=center,scale=1.3] (n2x8) at (8.50,2.50) {$3$};
\draw [] (2.0,3.0) rectangle (3.0,4.0); \node[anchor=center,scale=1.3] (n3x2) at (2.50,3.50) {$0$};
\draw [] (3.0,3.0) rectangle (4.0,4.0); \node[anchor=center,scale=1.3] (n3x3) at (3.50,3.50) {$0$};
\draw [] (4.0,3.0) rectangle (5.0,4.0); \node[anchor=center,scale=1.3] (n3x4) at (4.50,3.50) {$1$};
\draw [] (5.0,3.0) rectangle (6.0,4.0); \node[anchor=center,scale=1.3] (n3x5) at (5.50,3.50) {$2$};
\draw [] (6.0,3.0) rectangle (7.0,4.0); \node[anchor=center,scale=1.3] (n3x6) at (6.50,3.50) {$2$};
\draw [] (7.0,3.0) rectangle (8.0,4.0); \node[anchor=center,scale=1.3] (n3x7) at (7.50,3.50) {$3$};
\draw [] (1.0,4.0) rectangle (2.0,5.0); \node[anchor=center,scale=1.3] (n4x1) at (1.50,4.50) {$0$};
\draw [] (2.0,4.0) rectangle (3.0,5.0); \node[anchor=center,scale=1.3] (n4x2) at (2.50,4.50) {$0$};
\draw [] (3.0,4.0) rectangle (4.0,5.0); \node[anchor=center,scale=1.3] (n4x3) at (3.50,4.50) {$0$};
\draw [] (4.0,4.0) rectangle (5.0,5.0); \node[anchor=center,scale=1.3] (n4x4) at (4.50,4.50) {$1$};
\draw [] (5.0,4.0) rectangle (6.0,5.0); \node[anchor=center,scale=1.3] (n4x5) at (5.50,4.50) {$2$};
\draw [] (6.0,4.0) rectangle (7.0,5.0); \node[anchor=center,scale=1.3] (n4x6) at (6.50,4.50) {$2$};
\draw [] (0.0,5.0) rectangle (1.0,6.0); \node[anchor=center,scale=1.3] (n5x0) at (0.50,5.50) {$1$};
\draw [] (1.0,5.0) rectangle (2.0,6.0); \node[anchor=center,scale=1.3] (n5x1) at (1.50,5.50) {$1$};
\draw [] (2.0,5.0) rectangle (3.0,6.0); \node[anchor=center,scale=1.3] (n5x2) at (2.50,5.50) {$1$};
\draw [] (3.0,5.0) rectangle (4.0,6.0); \node[anchor=center,scale=1.3] (n5x3) at (3.50,5.50) {$1$};
\draw [] (4.0,5.0) rectangle (5.0,6.0); \node[anchor=center,scale=1.3] (n5x4) at (4.50,5.50) {$1$};
\draw [] (5.0,5.0) rectangle (6.0,6.0); \node[anchor=center,scale=1.3] (n5x5) at (5.50,5.50) {$3$};
\draw [] (0.0,6.0) rectangle (1.0,7.0); \node[anchor=center,scale=1.3] (n6x0) at (0.50,6.50) {$1$};
\draw [] (1.0,6.0) rectangle (2.0,7.0); \node[anchor=center,scale=1.3] (n6x1) at (1.50,6.50) {$1$};
\draw [] (2.0,6.0) rectangle (3.0,7.0); \node[anchor=center,scale=1.3] (n6x2) at (2.50,6.50) {$1$};
\draw [] (3.0,6.0) rectangle (4.0,7.0); \node[anchor=center,scale=1.3] (n6x3) at (3.50,6.50) {$3$};
\draw [] (4.0,6.0) rectangle (5.0,7.0); \node[anchor=center,scale=1.3] (n6x4) at (4.50,6.50) {$3$};
\draw [] (0.0,7.0) rectangle (1.0,8.0); \node[anchor=center,scale=1.3] (n7x0) at (0.50,7.50) {$1$};
\draw [] (1.0,7.0) rectangle (2.0,8.0); \node[anchor=center,scale=1.3] (n7x1) at (1.50,7.50) {$1$};
\draw [] (2.0,7.0) rectangle (3.0,8.0); \node[anchor=center,scale=1.3] (n7x2) at (2.50,7.50) {$1$};
\draw [] (3.0,7.0) rectangle (4.0,8.0); \node[anchor=center,scale=1.3] (n7x3) at (3.50,7.50) {$3$};
\draw [] (0.0,8.0) rectangle (1.0,9.0); \node[anchor=center,scale=1.3] (n8x0) at (0.50,8.50) {$1$};
\draw [] (1.0,8.0) rectangle (2.0,9.0); \node[anchor=center,scale=1.3] (n8x1) at (1.50,8.50) {$1$};
\draw [] (2.0,8.0) rectangle (3.0,9.0); \node[anchor=center,scale=1.3] (n8x2) at (2.50,8.50) {$1$};
\draw [] (0.0,9.0) rectangle (1.0,10.0); \node[anchor=center,scale=1.3] (n9x0) at (0.50,9.50) {$1$};
\draw [] (1.0,9.0) rectangle (2.0,10.0); \node[anchor=center,scale=1.3] (n9x1) at (1.50,9.50) {$1$};
\end{tikzpicture}}
\\

(a)
&
(b)
\\

&
 
\\

\scalebox{0.4}{
\begin{tikzpicture}[scale=1.0,yscale=-1]
\draw [] (5.0,0.0) rectangle (6.0,1.0); \node[anchor=center,scale=1.3] (nn0x5) at (5.50,0.50) {$0$};
\draw [] (6.0,0.0) rectangle (7.0,1.0); \node[anchor=center,scale=1.3] (nn0x6) at (6.50,0.50) {$0$};
\draw [fill=black, opacity=0.2] (7.0,0.0) -- (7.71,0.71) -- (7.71,1.71) -- (7.0,1.0) -- cycle;
\draw [] (7.71,0.71) rectangle (8.71,1.71); \node[anchor=center,scale=1.3] (nn0x7) at (8.21,1.21) {$1$};
\draw [] (8.71,0.71) rectangle (9.71,1.71); \node[anchor=center,scale=1.3] (nn0x8) at (9.21,1.21) {$1$};
\draw [fill=black, opacity=0.2] (9.71,0.71) -- (10.420000000000002,1.42) -- (10.420000000000002,2.42) -- (9.71,1.71) -- cycle;
\draw [fill=black, opacity=0.2] (10.42,1.42) -- (11.129999999999999,2.13) -- (11.129999999999999,3.13) -- (10.42,2.42) -- cycle;
\draw [fill=black, opacity=0.2] (4.0,1.0) -- (4.71,1.71) -- (4.71,2.71) -- (4.0,2.0) -- cycle;
\draw [] (4.71,1.71) rectangle (5.71,2.71); \node[anchor=center,scale=1.3] (nn1x4) at (5.21,2.21) {$1$};
\draw [] (5.71,1.71) rectangle (6.71,2.71); \node[anchor=center,scale=1.3] (nn1x5) at (6.21,2.21) {$1$};
\draw [] (6.71,1.71) rectangle (7.71,2.71); \node[anchor=center,scale=1.3] (nn1x6) at (7.21,2.21) {$1$};
\draw [] (7.71,1.71) rectangle (8.71,2.71); \node[anchor=center,scale=1.3] (nn1x7) at (8.21,2.21) {$1$};
\draw [fill=black, opacity=0.2] (8.71,1.71) -- (9.420000000000002,2.42) -- (9.420000000000002,3.42) -- (8.71,2.71) -- cycle;
\draw [] (9.42,2.42) rectangle (10.42,3.42); \node[anchor=center,scale=1.3] (nn1x8) at (9.92,2.92) {$2$};
\draw [fill=black, opacity=0.2] (10.42,2.42) -- (11.129999999999999,3.13) -- (11.129999999999999,4.13) -- (10.42,3.42) -- cycle;
\draw [] (3.0,2.0) rectangle (4.0,3.0); \node[anchor=center,scale=1.3] (nn2x3) at (3.50,2.50) {$0$};
\draw [fill=black, opacity=0.2] (4.0,2.0) -- (4.71,2.71) -- (4.71,3.71) -- (4.0,3.0) -- cycle;
\draw [] (4.71,2.71) rectangle (5.71,3.71); \node[anchor=center,scale=1.3] (nn2x4) at (5.21,3.21) {$1$};
\draw [] (5.71,2.71) rectangle (6.71,3.71); \node[anchor=center,scale=1.3] (nn2x5) at (6.21,3.21) {$1$};
\draw [] (6.71,2.71) rectangle (7.71,3.71); \node[anchor=center,scale=1.3] (nn2x6) at (7.21,3.21) {$1$};
\draw [] (7.71,2.71) rectangle (8.71,3.71); \node[anchor=center,scale=1.3] (nn2x7) at (8.21,3.21) {$1$};
\draw [fill=black, opacity=0.2] (8.71,2.71) -- (9.420000000000002,3.42) -- (9.420000000000002,4.42) -- (8.71,3.71) -- cycle;
\draw [fill=black, opacity=0.2] (9.42,3.42) -- (10.129999999999999,4.13) -- (10.129999999999999,5.13) -- (9.42,4.42) -- cycle;
\draw [] (10.129999999999999,4.13) rectangle (11.129999999999999,5.13); \node[anchor=center,scale=1.3] (nn2x8) at (10.63,4.63) {$3$};
\draw [] (2.0,3.0) rectangle (3.0,4.0); \node[anchor=center,scale=1.3] (nn3x2) at (2.50,3.50) {$0$};
\draw [] (3.0,3.0) rectangle (4.0,4.0); \node[anchor=center,scale=1.3] (nn3x3) at (3.50,3.50) {$0$};
\draw [fill=black, opacity=0.2] (4.0,3.0) -- (4.71,3.71) -- (4.71,4.71) -- (4.0,4.0) -- cycle;
\draw [] (4.71,3.71) rectangle (5.71,4.71); \node[anchor=center,scale=1.3] (nn3x4) at (5.21,4.21) {$1$};
\draw [fill=black, opacity=0.2] (5.71,3.71) -- (6.42,4.42) -- (6.42,5.42) -- (5.71,4.71) -- cycle;
\draw [] (6.42,4.42) rectangle (7.42,5.42); \node[anchor=center,scale=1.3] (nn3x5) at (6.92,4.92) {$2$};
\draw [] (7.42,4.42) rectangle (8.42,5.42); \node[anchor=center,scale=1.3] (nn3x6) at (7.92,4.92) {$2$};
\draw [fill=black, opacity=0.2] (8.42,4.42) -- (9.129999999999999,5.13) -- (9.129999999999999,6.13) -- (8.42,5.42) -- cycle;
\draw [] (9.129999999999999,5.13) rectangle (10.129999999999999,6.13); \node[anchor=center,scale=1.3] (nn3x7) at (9.63,5.63) {$3$};
\draw [] (1.0,4.0) rectangle (2.0,5.0); \node[anchor=center,scale=1.3] (nn4x1) at (1.50,4.50) {$0$};
\draw [] (2.0,4.0) rectangle (3.0,5.0); \node[anchor=center,scale=1.3] (nn4x2) at (2.50,4.50) {$0$};
\draw [] (3.0,4.0) rectangle (4.0,5.0); \node[anchor=center,scale=1.3] (nn4x3) at (3.50,4.50) {$0$};
\draw [fill=black, opacity=0.2] (4.0,4.0) -- (4.71,4.71) -- (4.71,5.71) -- (4.0,5.0) -- cycle;
\draw [] (4.71,4.71) rectangle (5.71,5.71); \node[anchor=center,scale=1.3] (nn4x4) at (5.21,5.21) {$1$};
\draw [fill=black, opacity=0.2] (5.71,4.71) -- (6.42,5.42) -- (6.42,6.42) -- (5.71,5.71) -- cycle;
\draw [] (6.42,5.42) rectangle (7.42,6.42); \node[anchor=center,scale=1.3] (nn4x5) at (6.92,5.92) {$2$};
\draw [] (7.42,5.42) rectangle (8.42,6.42); \node[anchor=center,scale=1.3] (nn4x6) at (7.92,5.92) {$2$};
\draw [fill=black, opacity=0.2] (8.42,5.42) -- (9.129999999999999,6.13) -- (9.129999999999999,7.13) -- (8.42,6.42) -- cycle;
\draw [fill=black, opacity=0.2] (0.0,5.0) -- (0.71,5.71) -- (0.71,6.71) -- (0.0,6.0) -- cycle;
\draw [] (0.71,5.71) rectangle (1.71,6.71); \node[anchor=center,scale=1.3] (nn5x0) at (1.21,6.21) {$1$};
\draw [] (1.71,5.71) rectangle (2.71,6.71); \node[anchor=center,scale=1.3] (nn5x1) at (2.21,6.21) {$1$};
\draw [] (2.71,5.71) rectangle (3.71,6.71); \node[anchor=center,scale=1.3] (nn5x2) at (3.21,6.21) {$1$};
\draw [] (3.71,5.71) rectangle (4.71,6.71); \node[anchor=center,scale=1.3] (nn5x3) at (4.21,6.21) {$1$};
\draw [] (4.71,5.71) rectangle (5.71,6.71); \node[anchor=center,scale=1.3] (nn5x4) at (5.21,6.21) {$1$};
\draw [fill=black, opacity=0.2] (5.71,5.71) -- (6.42,6.42) -- (6.42,7.42) -- (5.71,6.71) -- cycle;
\draw [fill=black, opacity=0.2] (6.42,6.42) -- (7.13,7.13) -- (7.13,8.129999999999999) -- (6.42,7.42) -- cycle;
\draw [] (7.13,7.13) rectangle (8.129999999999999,8.129999999999999); \node[anchor=center,scale=1.3] (nn5x5) at (7.63,7.63) {$3$};
\draw [fill=black, opacity=0.2] (0.0,6.0) -- (0.71,6.71) -- (0.71,7.71) -- (0.0,7.0) -- cycle;
\draw [] (0.71,6.71) rectangle (1.71,7.71); \node[anchor=center,scale=1.3] (nn6x0) at (1.21,7.21) {$1$};
\draw [] (1.71,6.71) rectangle (2.71,7.71); \node[anchor=center,scale=1.3] (nn6x1) at (2.21,7.21) {$1$};
\draw [] (2.71,6.71) rectangle (3.71,7.71); \node[anchor=center,scale=1.3] (nn6x2) at (3.21,7.21) {$1$};
\draw [fill=black, opacity=0.2] (3.71,6.71) -- (4.42,7.42) -- (4.42,8.42) -- (3.71,7.71) -- cycle;
\draw [fill=black, opacity=0.2] (4.42,7.42) -- (5.13,8.129999999999999) -- (5.13,9.129999999999999) -- (4.42,8.42) -- cycle;
\draw [] (5.13,8.129999999999999) rectangle (6.13,9.129999999999999); \node[anchor=center,scale=1.3] (nn6x3) at (5.63,8.63) {$3$};
\draw [] (6.13,8.129999999999999) rectangle (7.13,9.129999999999999); \node[anchor=center,scale=1.3] (nn6x4) at (6.63,8.63) {$3$};
\draw [fill=black, opacity=0.2] (0.0,7.0) -- (0.71,7.71) -- (0.71,8.71) -- (0.0,8.0) -- cycle;
\draw [] (0.71,7.71) rectangle (1.71,8.71); \node[anchor=center,scale=1.3] (nn7x0) at (1.21,8.21) {$1$};
\draw [] (1.71,7.71) rectangle (2.71,8.71); \node[anchor=center,scale=1.3] (nn7x1) at (2.21,8.21) {$1$};
\draw [] (2.71,7.71) rectangle (3.71,8.71); \node[anchor=center,scale=1.3] (nn7x2) at (3.21,8.21) {$1$};
\draw [fill=black, opacity=0.2] (3.71,7.71) -- (4.42,8.42) -- (4.42,9.42) -- (3.71,8.71) -- cycle;
\draw [fill=black, opacity=0.2] (4.42,8.42) -- (5.13,9.129999999999999) -- (5.13,10.129999999999999) -- (4.42,9.42) -- cycle;
\draw [] (5.13,9.129999999999999) rectangle (6.13,10.129999999999999); \node[anchor=center,scale=1.3] (nn7x3) at (5.63,9.63) {$3$};
\draw [fill=black, opacity=0.2] (0.0,8.0) -- (0.71,8.71) -- (0.71,9.71) -- (0.0,9.0) -- cycle;
\draw [] (0.71,8.71) rectangle (1.71,9.71); \node[anchor=center,scale=1.3] (nn8x0) at (1.21,9.21) {$1$};
\draw [] (1.71,8.71) rectangle (2.71,9.71); \node[anchor=center,scale=1.3] (nn8x1) at (2.21,9.21) {$1$};
\draw [] (2.71,8.71) rectangle (3.71,9.71); \node[anchor=center,scale=1.3] (nn8x2) at (3.21,9.21) {$1$};
\draw [fill=black, opacity=0.2] (3.71,8.71) -- (4.42,9.420000000000002) -- (4.42,10.420000000000002) -- (3.71,9.71) -- cycle;
\draw [fill=black, opacity=0.2] (4.42,9.42) -- (5.13,10.129999999999999) -- (5.13,11.129999999999999) -- (4.42,10.42) -- cycle;
\draw [fill=black, opacity=0.2] (0.0,9.0) -- (0.71,9.71) -- (0.71,10.71) -- (0.0,10.0) -- cycle;
\draw [] (0.71,9.71) rectangle (1.71,10.71); \node[anchor=center,scale=1.3] (nn9x0) at (1.21,10.21) {$1$};
\draw [] (1.71,9.71) rectangle (2.71,10.71); \node[anchor=center,scale=1.3] (nn9x1) at (2.21,10.21) {$1$};
\draw [fill=black, opacity=0.2] (2.71,9.71) -- (3.42,10.420000000000002) -- (3.42,11.420000000000002) -- (2.71,10.71) -- cycle;
\draw [fill=black, opacity=0.2] (3.42,10.42) -- (4.13,11.129999999999999) -- (4.13,12.129999999999999) -- (3.42,11.42) -- cycle;
\draw [fill=black, opacity=0.1] (0.0,5.0) -- (0.71,5.71) -- (1.71,5.71) -- (1.0,5.0) -- cycle;
\draw [fill=black, opacity=0.1] (0.71,10.71) -- (1.42,11.420000000000002) -- (2.42,11.420000000000002) -- (1.71,10.71) -- cycle;
\draw [fill=black, opacity=0.1] (1.42,11.42) -- (2.13,12.129999999999999) -- (3.13,12.129999999999999) -- (2.42,11.42) -- cycle;
\draw [fill=black, opacity=0.1] (1.0,5.0) -- (1.71,5.71) -- (2.71,5.71) -- (2.0,5.0) -- cycle;
\draw [fill=black, opacity=0.1] (1.71,10.71) -- (2.42,11.420000000000002) -- (3.42,11.420000000000002) -- (2.71,10.71) -- cycle;
\draw [fill=black, opacity=0.1] (2.42,11.42) -- (3.13,12.129999999999999) -- (4.13,12.129999999999999) -- (3.42,11.42) -- cycle;
\draw [fill=black, opacity=0.1] (2.0,5.0) -- (2.71,5.71) -- (3.71,5.71) -- (3.0,5.0) -- cycle;
\draw [fill=black, opacity=0.1] (2.71,9.71) -- (3.42,10.420000000000002) -- (4.42,10.420000000000002) -- (3.71,9.71) -- cycle;
\draw [fill=black, opacity=0.1] (3.42,10.42) -- (4.13,11.129999999999999) -- (5.13,11.129999999999999) -- (4.42,10.42) -- cycle;
\draw [fill=black, opacity=0.1] (3.0,5.0) -- (3.71,5.71) -- (4.71,5.71) -- (4.0,5.0) -- cycle;
\draw [fill=black, opacity=0.1] (3.71,6.71) -- (4.42,7.42) -- (5.42,7.42) -- (4.71,6.71) -- cycle;
\draw [fill=black, opacity=0.1] (4.42,7.42) -- (5.13,8.129999999999999) -- (6.13,8.129999999999999) -- (5.42,7.42) -- cycle;
\draw [fill=black, opacity=0.1] (4.0,1.0) -- (4.71,1.71) -- (5.71,1.71) -- (5.0,1.0) -- cycle;
\draw [fill=black, opacity=0.1] (4.71,6.71) -- (5.42,7.42) -- (6.42,7.42) -- (5.71,6.71) -- cycle;
\draw [fill=black, opacity=0.1] (5.42,7.42) -- (6.13,8.129999999999999) -- (7.13,8.129999999999999) -- (6.42,7.42) -- cycle;
\draw [fill=black, opacity=0.1] (5.0,1.0) -- (5.71,1.71) -- (6.71,1.71) -- (6.0,1.0) -- cycle;
\draw [fill=black, opacity=0.1] (5.71,3.71) -- (6.42,4.42) -- (7.42,4.42) -- (6.71,3.71) -- cycle;
\draw [fill=black, opacity=0.1] (6.42,6.42) -- (7.13,7.13) -- (8.129999999999999,7.13) -- (7.42,6.42) -- cycle;
\draw [fill=black, opacity=0.1] (6.0,1.0) -- (6.71,1.71) -- (7.71,1.71) -- (7.0,1.0) -- cycle;
\draw [fill=black, opacity=0.1] (6.71,3.71) -- (7.42,4.42) -- (8.42,4.42) -- (7.71,3.71) -- cycle;
\draw [fill=black, opacity=0.1] (7.42,6.42) -- (8.129999999999999,7.13) -- (9.129999999999999,7.13) -- (8.42,6.42) -- cycle;
\draw [fill=black, opacity=0.1] (7.0,0.0) -- (7.71,0.71) -- (8.71,0.71) -- (8.0,0.0) -- cycle;
\draw [fill=black, opacity=0.1] (7.71,3.71) -- (8.42,4.42) -- (9.42,4.42) -- (8.71,3.71) -- cycle;
\draw [fill=black, opacity=0.1] (8.42,4.42) -- (9.129999999999999,5.13) -- (10.129999999999999,5.13) -- (9.42,4.42) -- cycle;
\draw [fill=black, opacity=0.1] (8.0,0.0) -- (8.71,0.71) -- (9.71,0.71) -- (9.0,0.0) -- cycle;
\draw [fill=black, opacity=0.1] (8.71,1.71) -- (9.420000000000002,2.42) -- (10.420000000000002,2.42) -- (9.71,1.71) -- cycle;
\draw [fill=black, opacity=0.1] (9.42,3.42) -- (10.129999999999999,4.13) -- (11.129999999999999,4.13) -- (10.42,3.42) -- cycle;
\end{tikzpicture}}
&
\scalebox{0.4}{
\begin{tikzpicture}[scale=1.0,yscale=-1]
\draw [] (5.0,0.0) rectangle (6.0,1.0); \node[anchor=center,scale=1.3] (nn0x5) at (5.50,0.50) {$0$};
\draw [] (6.0,0.0) rectangle (7.0,1.0); \node[anchor=center,scale=1.3] (nn0x6) at (6.50,0.50) {$0$};
\draw [fill=black, opacity=0.2] (7.0,0.0) -- (7.71,0.71) -- (7.71,1.71) -- (7.0,1.0) -- cycle;
\draw [] (7.71,0.71) rectangle (8.71,1.71); \node[anchor=center,scale=1.3] (nn0x7) at (8.21,1.21) {$1$};
\draw [] (8.71,0.71) rectangle (9.71,1.71); \node[anchor=center,scale=1.3] (nn0x8) at (9.21,1.21) {$1$};
\draw [fill=black, opacity=0.2] (9.71,0.71) -- (10.420000000000002,1.42) -- (10.420000000000002,2.42) -- (9.71,1.71) -- cycle;
\draw [fill=black, opacity=0.2] (10.42,1.42) -- (11.129999999999999,2.13) -- (11.129999999999999,3.13) -- (10.42,2.42) -- cycle;
\draw [fill=black, opacity=0.2] (4.0,1.0) -- (4.71,1.71) -- (4.71,2.71) -- (4.0,2.0) -- cycle;
\draw [] (4.71,1.71) rectangle (5.71,2.71); \node[anchor=center,scale=1.3] (nn1x4) at (5.21,2.21) {$1$};
\draw [] (5.71,1.71) rectangle (6.71,2.71); \node[anchor=center,scale=1.3] (nn1x5) at (6.21,2.21) {$1$};
\draw [] (6.71,1.71) rectangle (7.71,2.71); \node[anchor=center,scale=1.3] (nn1x6) at (7.21,2.21) {$1$};
\draw [] (7.71,1.71) rectangle (8.71,2.71); \node[anchor=center,scale=1.3] (nn1x7) at (8.21,2.21) {$1$};
\draw [fill=black, opacity=0.2] (8.71,1.71) -- (9.420000000000002,2.42) -- (9.420000000000002,3.42) -- (8.71,2.71) -- cycle;
\draw [] (9.42,2.42) rectangle (10.42,3.42); \node[anchor=center,scale=1.3] (nn1x8) at (9.92,2.92) {$2$};
\draw [fill=black, opacity=0.2] (10.42,2.42) -- (11.129999999999999,3.13) -- (11.129999999999999,4.13) -- (10.42,3.42) -- cycle;
\draw [] (3.0,2.0) rectangle (4.0,3.0); \node[anchor=center,scale=1.3] (nn2x3) at (3.50,2.50) {$0$};
\draw [fill=black, opacity=0.2] (4.0,2.0) -- (4.71,2.71) -- (4.71,3.71) -- (4.0,3.0) -- cycle;
\draw [] (4.71,2.71) rectangle (5.71,3.71); \node[anchor=center,scale=1.3] (nn2x4) at (5.21,3.21) {$1$};
\draw [] (5.71,2.71) rectangle (6.71,3.71); \node[anchor=center,scale=1.3] (nn2x5) at (6.21,3.21) {$1$};
\draw [] (6.71,2.71) rectangle (7.71,3.71); \node[anchor=center,scale=1.3] (nn2x6) at (7.21,3.21) {$1$};
\draw [] (7.71,2.71) rectangle (8.71,3.71); \node[anchor=center,scale=1.3] (nn2x7) at (8.21,3.21) {$1$};
\draw [fill=black, opacity=0.2] (8.71,2.71) -- (9.420000000000002,3.42) -- (9.420000000000002,4.42) -- (8.71,3.71) -- cycle;
\draw [fill=black, opacity=0.2] (9.42,3.42) -- (10.129999999999999,4.13) -- (10.129999999999999,5.13) -- (9.42,4.42) -- cycle;
\draw [] (10.129999999999999,4.13) rectangle (11.129999999999999,5.13); \node[anchor=center,scale=1.3] (nn2x8) at (10.63,4.63) {$3$};
\draw [] (2.0,3.0) rectangle (3.0,4.0); \node[anchor=center,scale=1.3] (nn3x2) at (2.50,3.50) {$0$};
\draw [] (3.0,3.0) rectangle (4.0,4.0); \node[anchor=center,scale=1.3] (nn3x3) at (3.50,3.50) {$0$};
\draw [fill=black, opacity=0.2] (4.0,3.0) -- (4.71,3.71) -- (4.71,4.71) -- (4.0,4.0) -- cycle;
\draw [] (4.71,3.71) rectangle (5.71,4.71); \node[anchor=center,scale=1.3] (nn3x4) at (5.21,4.21) {$1$};
\draw [fill=black, opacity=0.2] (5.71,3.71) -- (6.42,4.42) -- (6.42,5.42) -- (5.71,4.71) -- cycle;
\draw [] (6.42,4.42) rectangle (7.42,5.42); \node[anchor=center,scale=1.3] (nn3x5) at (6.92,4.92) {$2$};
\draw [] (7.42,4.42) rectangle (8.42,5.42); \node[anchor=center,scale=1.3] (nn3x6) at (7.92,4.92) {$2$};
\draw [fill=black, opacity=0.2] (8.42,4.42) -- (9.129999999999999,5.13) -- (9.129999999999999,6.13) -- (8.42,5.42) -- cycle;
\draw [] (9.129999999999999,5.13) rectangle (10.129999999999999,6.13); \node[anchor=center,scale=1.3] (nn3x7) at (9.63,5.63) {$3$};
\draw [] (1.0,4.0) rectangle (2.0,5.0); \node[anchor=center,scale=1.3] (nn4x1) at (1.50,4.50) {$0$};
\draw [] (2.0,4.0) rectangle (3.0,5.0); \node[anchor=center,scale=1.3] (nn4x2) at (2.50,4.50) {$0$};
\draw [] (3.0,4.0) rectangle (4.0,5.0); \node[anchor=center,scale=1.3] (nn4x3) at (3.50,4.50) {$0$};
\draw [fill=black, opacity=0.2] (4.0,4.0) -- (4.71,4.71) -- (4.71,5.71) -- (4.0,5.0) -- cycle;
\draw [] (4.71,4.71) rectangle (5.71,5.71); \node[anchor=center,scale=1.3] (nn4x4) at (5.21,5.21) {$1$};
\draw [fill=black, opacity=0.2] (5.71,4.71) -- (6.42,5.42) -- (6.42,6.42) -- (5.71,5.71) -- cycle;
\draw [] (6.42,5.42) rectangle (7.42,6.42); \node[anchor=center,scale=1.3] (nn4x5) at (6.92,5.92) {$2$};
\draw [] (7.42,5.42) rectangle (8.42,6.42); \node[anchor=center,scale=1.3] (nn4x6) at (7.92,5.92) {$2$};
\draw [fill=black, opacity=0.2] (8.42,5.42) -- (9.129999999999999,6.13) -- (9.129999999999999,7.13) -- (8.42,6.42) -- cycle;
\draw [fill=black, opacity=0.2] (0.0,5.0) -- (0.71,5.71) -- (0.71,6.71) -- (0.0,6.0) -- cycle;
\draw [] (0.71,5.71) rectangle (1.71,6.71); \node[anchor=center,scale=1.3] (nn5x0) at (1.21,6.21) {$1$};
\draw [] (1.71,5.71) rectangle (2.71,6.71); \node[anchor=center,scale=1.3] (nn5x1) at (2.21,6.21) {$1$};
\draw [] (2.71,5.71) rectangle (3.71,6.71); \node[anchor=center,scale=1.3] (nn5x2) at (3.21,6.21) {$1$};
\draw [] (3.71,5.71) rectangle (4.71,6.71); \node[anchor=center,scale=1.3] (nn5x3) at (4.21,6.21) {$1$};
\draw [] (4.71,5.71) rectangle (5.71,6.71); \node[anchor=center,scale=1.3] (nn5x4) at (5.21,6.21) {$1$};
\draw [fill=black, opacity=0.2] (5.71,5.71) -- (6.42,6.42) -- (6.42,7.42) -- (5.71,6.71) -- cycle;
\draw [fill=black, opacity=0.2] (6.42,6.42) -- (7.13,7.13) -- (7.13,8.129999999999999) -- (6.42,7.42) -- cycle;
\draw [] (7.13,7.13) rectangle (8.129999999999999,8.129999999999999); \node[anchor=center,scale=1.3] (nn5x5) at (7.63,7.63) {$3$};
\draw [fill=black, opacity=0.2] (0.0,6.0) -- (0.71,6.71) -- (0.71,7.71) -- (0.0,7.0) -- cycle;
\draw [] (0.71,6.71) rectangle (1.71,7.71); \node[anchor=center,scale=1.3] (nn6x0) at (1.21,7.21) {$1$};
\draw [] (1.71,6.71) rectangle (2.71,7.71); \node[anchor=center,scale=1.3] (nn6x1) at (2.21,7.21) {$1$};
\draw [] (2.71,6.71) rectangle (3.71,7.71); \node[anchor=center,scale=1.3] (nn6x2) at (3.21,7.21) {$1$};
\draw [fill=black, opacity=0.2] (3.71,6.71) -- (4.42,7.42) -- (4.42,8.42) -- (3.71,7.71) -- cycle;
\draw [fill=black, opacity=0.2] (4.42,7.42) -- (5.13,8.129999999999999) -- (5.13,9.129999999999999) -- (4.42,8.42) -- cycle;
\draw [] (5.13,8.129999999999999) rectangle (6.13,9.129999999999999); \node[anchor=center,scale=1.3] (nn6x3) at (5.63,8.63) {$3$};
\draw [] (6.13,8.129999999999999) rectangle (7.13,9.129999999999999); \node[anchor=center,scale=1.3] (nn6x4) at (6.63,8.63) {$3$};
\draw [fill=black, opacity=0.2] (0.0,7.0) -- (0.71,7.71) -- (0.71,8.71) -- (0.0,8.0) -- cycle;
\draw [] (0.71,7.71) rectangle (1.71,8.71); \node[anchor=center,scale=1.3] (nn7x0) at (1.21,8.21) {$1$};
\draw [] (1.71,7.71) rectangle (2.71,8.71); \node[anchor=center,scale=1.3] (nn7x1) at (2.21,8.21) {$1$};
\draw [] (2.71,7.71) rectangle (3.71,8.71); \node[anchor=center,scale=1.3] (nn7x2) at (3.21,8.21) {$1$};
\draw [fill=black, opacity=0.2] (3.71,7.71) -- (4.42,8.42) -- (4.42,9.42) -- (3.71,8.71) -- cycle;
\draw [fill=black, opacity=0.2] (4.42,8.42) -- (5.13,9.129999999999999) -- (5.13,10.129999999999999) -- (4.42,9.42) -- cycle;
\draw [] (5.13,9.129999999999999) rectangle (6.13,10.129999999999999); \node[anchor=center,scale=1.3] (nn7x3) at (5.63,9.63) {$3$};
\draw [fill=black, opacity=0.2] (0.0,8.0) -- (0.71,8.71) -- (0.71,9.71) -- (0.0,9.0) -- cycle;
\draw [] (0.71,8.71) rectangle (1.71,9.71); \node[anchor=center,scale=1.3] (nn8x0) at (1.21,9.21) {$1$};
\draw [] (1.71,8.71) rectangle (2.71,9.71); \node[anchor=center,scale=1.3] (nn8x1) at (2.21,9.21) {$1$};
\draw [] (2.71,8.71) rectangle (3.71,9.71); \node[anchor=center,scale=1.3] (nn8x2) at (3.21,9.21) {$1$};
\draw [fill=black, opacity=0.2] (3.71,8.71) -- (4.42,9.420000000000002) -- (4.42,10.420000000000002) -- (3.71,9.71) -- cycle;
\draw [fill=black, opacity=0.2] (4.42,9.42) -- (5.13,10.129999999999999) -- (5.13,11.129999999999999) -- (4.42,10.42) -- cycle;
\draw [fill=black, opacity=0.2] (0.0,9.0) -- (0.71,9.71) -- (0.71,10.71) -- (0.0,10.0) -- cycle;
\draw [] (0.71,9.71) rectangle (1.71,10.71); \node[anchor=center,scale=1.3] (nn9x0) at (1.21,10.21) {$1$};
\draw [] (1.71,9.71) rectangle (2.71,10.71); \node[anchor=center,scale=1.3] (nn9x1) at (2.21,10.21) {$1$};
\draw [fill=black, opacity=0.2] (2.71,9.71) -- (3.42,10.420000000000002) -- (3.42,11.420000000000002) -- (2.71,10.71) -- cycle;
\draw [fill=black, opacity=0.2] (3.42,10.42) -- (4.13,11.129999999999999) -- (4.13,12.129999999999999) -- (3.42,11.42) -- cycle;
\draw [fill=black, opacity=0.1] (0.0,5.0) -- (0.71,5.71) -- (1.71,5.71) -- (1.0,5.0) -- cycle;
\draw [fill=black, opacity=0.1] (0.71,10.71) -- (1.42,11.420000000000002) -- (2.42,11.420000000000002) -- (1.71,10.71) -- cycle;
\draw [fill=black, opacity=0.1] (1.42,11.42) -- (2.13,12.129999999999999) -- (3.13,12.129999999999999) -- (2.42,11.42) -- cycle;
\draw [fill=black, opacity=0.1] (1.0,5.0) -- (1.71,5.71) -- (2.71,5.71) -- (2.0,5.0) -- cycle;
\draw [fill=black, opacity=0.1] (1.71,10.71) -- (2.42,11.420000000000002) -- (3.42,11.420000000000002) -- (2.71,10.71) -- cycle;
\draw [fill=black, opacity=0.1] (2.42,11.42) -- (3.13,12.129999999999999) -- (4.13,12.129999999999999) -- (3.42,11.42) -- cycle;
\draw [fill=black, opacity=0.1] (2.0,5.0) -- (2.71,5.71) -- (3.71,5.71) -- (3.0,5.0) -- cycle;
\draw [fill=black, opacity=0.1] (2.71,9.71) -- (3.42,10.420000000000002) -- (4.42,10.420000000000002) -- (3.71,9.71) -- cycle;
\draw [fill=black, opacity=0.1] (3.42,10.42) -- (4.13,11.129999999999999) -- (5.13,11.129999999999999) -- (4.42,10.42) -- cycle;
\draw [fill=black, opacity=0.1] (3.0,5.0) -- (3.71,5.71) -- (4.71,5.71) -- (4.0,5.0) -- cycle;
\draw [fill=black, opacity=0.1] (3.71,6.71) -- (4.42,7.42) -- (5.42,7.42) -- (4.71,6.71) -- cycle;
\draw [fill=black, opacity=0.1] (4.42,7.42) -- (5.13,8.129999999999999) -- (6.13,8.129999999999999) -- (5.42,7.42) -- cycle;
\draw [fill=black, opacity=0.1] (4.0,1.0) -- (4.71,1.71) -- (5.71,1.71) -- (5.0,1.0) -- cycle;
\draw [fill=black, opacity=0.1] (4.71,6.71) -- (5.42,7.42) -- (6.42,7.42) -- (5.71,6.71) -- cycle;
\draw [fill=black, opacity=0.1] (5.42,7.42) -- (6.13,8.129999999999999) -- (7.13,8.129999999999999) -- (6.42,7.42) -- cycle;
\draw [fill=black, opacity=0.1] (5.0,1.0) -- (5.71,1.71) -- (6.71,1.71) -- (6.0,1.0) -- cycle;
\draw [fill=black, opacity=0.1] (5.71,3.71) -- (6.42,4.42) -- (7.42,4.42) -- (6.71,3.71) -- cycle;
\draw [fill=black, opacity=0.1] (6.42,6.42) -- (7.13,7.13) -- (8.129999999999999,7.13) -- (7.42,6.42) -- cycle;
\draw [fill=black, opacity=0.1] (6.0,1.0) -- (6.71,1.71) -- (7.71,1.71) -- (7.0,1.0) -- cycle;
\draw [fill=black, opacity=0.1] (6.71,3.71) -- (7.42,4.42) -- (8.42,4.42) -- (7.71,3.71) -- cycle;
\draw [fill=black, opacity=0.1] (7.42,6.42) -- (8.129999999999999,7.13) -- (9.129999999999999,7.13) -- (8.42,6.42) -- cycle;
\draw [fill=black, opacity=0.1] (7.0,0.0) -- (7.71,0.71) -- (8.71,0.71) -- (8.0,0.0) -- cycle;
\draw [fill=black, opacity=0.1] (7.71,3.71) -- (8.42,4.42) -- (9.42,4.42) -- (8.71,3.71) -- cycle;
\draw [fill=black, opacity=0.1] (8.42,4.42) -- (9.129999999999999,5.13) -- (10.129999999999999,5.13) -- (9.42,4.42) -- cycle;
\draw [fill=black, opacity=0.1] (8.0,0.0) -- (8.71,0.71) -- (9.71,0.71) -- (9.0,0.0) -- cycle;
\draw [fill=black, opacity=0.1] (8.71,1.71) -- (9.420000000000002,2.42) -- (10.420000000000002,2.42) -- (9.71,1.71) -- cycle;
\draw [fill=black, opacity=0.1] (9.42,3.42) -- (10.129999999999999,4.13) -- (11.129999999999999,4.13) -- (10.42,3.42) -- cycle;
\tikzset{myptr/.style={decoration={markings,mark=at position 1 with %
    {\arrow[scale=1.5,>=stealth]{>}}},postaction={decorate}}}
\draw[line width=1pt, red,myptr] (7.36,0.36) -- (7.36,1.36);
\draw[line width=1pt, red,myptr] (10.07,1.07) -- (10.07,2.07);
\draw[line width=1pt, red,myptr] (10.78,1.78) -- (10.78,2.78);
\draw[line width=1pt, red,myptr] (4.36,1.36) -- (4.36,2.36);
\draw[line width=1pt, red,myptr] (9.07,2.07) -- (9.07,3.07);
\draw[line width=1pt, red,myptr] (10.78,2.78) -- (10.78,3.78);
\draw[line width=1pt, red,myptr] (4.36,2.36) -- (4.36,3.36);
\draw[line width=1pt, red,myptr] (9.07,3.07) -- (9.07,4.06);
\draw[line width=1pt, red,myptr] (9.78,3.78) -- (9.78,4.78);
\draw[line width=1pt, red,myptr] (4.36,3.36) -- (4.36,4.36);
\draw[line width=1pt, red,myptr] (6.06,4.06) -- (6.06,5.06);
\draw[line width=1pt, red,myptr] (8.78,4.78) -- (8.78,5.78);
\draw[line width=1pt, red,myptr] (4.36,4.36) -- (4.36,5.36);
\draw[line width=1pt, red,myptr] (6.06,5.06) -- (6.06,6.06);
\draw[line width=1pt, red,myptr] (8.78,5.78) -- (8.78,6.78);
\draw[line width=1pt, red,myptr] (0.36,5.36) -- (0.36,6.36);
\draw[line width=1pt, red,myptr] (6.06,6.06) -- (6.06,7.06);
\draw[line width=1pt, red,myptr] (6.78,6.78) -- (6.78,7.78);
\draw[line width=1pt, red,myptr] (0.36,6.36) -- (0.36,7.36);
\draw[line width=1pt, red,myptr] (4.06,7.06) -- (4.06,8.07);
\draw[line width=1pt, red,myptr] (4.78,7.78) -- (4.78,8.78);
\draw[line width=1pt, red,myptr] (0.36,7.36) -- (0.36,8.36);
\draw[line width=1pt, red,myptr] (4.06,8.07) -- (4.06,9.07);
\draw[line width=1pt, red,myptr] (4.78,8.78) -- (4.78,9.78);
\draw[line width=1pt, red,myptr] (0.36,8.36) -- (0.36,9.36);
\draw[line width=1pt, red,myptr] (4.06,9.07) -- (4.06,10.07);
\draw[line width=1pt, red,myptr] (4.78,9.78) -- (4.78,10.78);
\draw[line width=1pt, red,myptr] (0.36,9.36) -- (0.36,10.36);
\draw[line width=1pt, red,myptr] (3.07,10.07) -- (3.07,11.07);
\draw[line width=1pt, red,myptr] (3.78,10.78) -- (3.78,11.78);
\draw[line width=1pt, red,myptr] (1.36,5.36) -- (0.36,5.36);
\draw[line width=1pt, red,myptr] (2.07,11.07) -- (1.07,11.07);
\draw[line width=1pt, red,myptr] (2.78,11.78) -- (1.78,11.78);
\draw[line width=1pt, red,myptr] (2.36,5.36) -- (1.36,5.36);
\draw[line width=1pt, red,myptr] (3.07,11.07) -- (2.07,11.07);
\draw[line width=1pt, red,myptr] (3.78,11.78) -- (2.78,11.78);
\draw[line width=1pt, red,myptr] (3.36,5.36) -- (2.36,5.36);
\draw[line width=1pt, red,myptr] (4.06,10.07) -- (3.07,10.07);
\draw[line width=1pt, red,myptr] (4.78,10.78) -- (3.78,10.78);
\draw[line width=1pt, red,myptr] (4.36,5.36) -- (3.36,5.36);
\draw[line width=1pt, red,myptr] (5.06,7.06) -- (4.06,7.06);
\draw[line width=1pt, red,myptr] (5.78,7.78) -- (4.78,7.78);
\draw[line width=1pt, red,myptr] (5.36,1.36) -- (4.36,1.36);
\draw[line width=1pt, red,myptr] (6.06,7.06) -- (5.06,7.06);
\draw[line width=1pt, red,myptr] (6.78,7.78) -- (5.78,7.78);
\draw[line width=1pt, red,myptr] (6.36,1.36) -- (5.36,1.36);
\draw[line width=1pt, red,myptr] (7.06,4.06) -- (6.06,4.06);
\draw[line width=1pt, red,myptr] (7.78,6.78) -- (6.78,6.78);
\draw[line width=1pt, red,myptr] (7.36,1.36) -- (6.36,1.36);
\draw[line width=1pt, red,myptr] (8.07,4.06) -- (7.06,4.06);
\draw[line width=1pt, red,myptr] (8.78,6.78) -- (7.78,6.78);
\draw[line width=1pt, red,myptr] (8.36,0.36) -- (7.36,0.36);
\draw[line width=1pt, red,myptr] (9.07,4.06) -- (8.07,4.06);
\draw[line width=1pt, red,myptr] (9.78,4.78) -- (8.78,4.78);
\draw[line width=1pt, red,myptr] (9.36,0.36) -- (8.36,0.36);
\draw[line width=1pt, red,myptr] (10.07,2.07) -- (9.07,2.07);
\draw[line width=1pt, red,myptr] (10.78,3.78) -- (9.78,3.78);
\end{tikzpicture}}
\\

(c)
&
(d)
\\

\end{tabular}

\caption{\label{fig:lozenges} (a) The skew shape $\sk=\sk(a,b,c,d,\ell)$ for $a=4,b=3,c=2,d=5,$ and $\ell=5$. (b) A weak reverse plane partition $\pi$ of shape $\sk$ with $r=3$. (c) A lozenge tiling corresponding to $\pi$.  (d) An $r$-path $\bij(\pi)$ in $\Nt_m$ corresponding to $\pi$.}
\end{figure}

Fix four integers $a,b,c,d\geq 1$ such that $a+b=c+d$. Then for each $\ell\geq 0$, define $\sk(a,b,c,d,\ell)$ to be the \emph{skew shape} depicted in Figure~\ref{fig:lozenges}~(a). Explicitly, we define $\sk(a,b,c,d,\ell)$ to be the difference of two Young diagrams $\sk=\Y(\l/\m)$ where 
\[\l=((\ell+a)^b,\ell+a-1,\ell+a-2,\dots,c),\quad \m=(\ell,\ell-1,\dots,1).\]
Here $(\ell+a)^b$ denotes the number $\ell+a$ repeated $b$ times. 

Fix some integer $r\geq 1$.
\begin{definition}
Given a skew shape $\sksh$, a \emph{(weak) reverse plane partition of shape} $\sksh$ is a filling $\pi$ of the boxes of $\Y(\sksh)$ with integers from $0$ to $r$ such that the numbers increase weakly along every row and column of $\sksh$. Define $|\pi|$ to be the sum of values of $\pi$.
\end{definition}

An example of a weak reverse plane partition $\pi$ of shape $\sk(a,b,c,d,\ell)$ is shown in Figure~\ref{fig:lozenges}~(b).  We have 
\[|\pi|=8\times0+28\times1+5\times2+6\times 3=56.\]
Let $q$ be an indeterminate. For each $\ell\geq 0$, we define
\begin{equation}\label{eq:f_lozenge}
f(\ell)=\sum_\pi q^{|\pi|},
\end{equation}
where the sum is taken over all weak reverse plane partitions of shape $\sk(a,b,c,d,\ell)$ with values from $0$ to $r$.

Let us now put $m=a+b+r$. For example, in the case of Figure~\ref{fig:lozenges}, we have $m=4+3+3=10$. For each $m\geq 2$, we introduce a planar cylindrical network $\Nt_m$ as follows. We put $\g=(-1,1)$, which is not a horizontal vector but we can rotate the whole picture by a $45$ degree angle. Next, the vertex set $\Vt$ of $\Nt_m$ consists of all points $(i,j)\in\R^2$ with integer coordinates satisfying $0\leq j-i\leq m-1$. If two vertices $(i,j)$ and $(i+1,j)$ both belong to $\Vt$ then $\Nt_m$ contains an edge $(i+1,j)\to(i,j)$ with weight $1$. If two vertices $(i,j)$ and $(i,j+1)$ both belong to $\Vt$ then $\Nt_m$ contains an edge $(i,j+1)\to (i,j)$ with weight $q^{j-i}$. Thus the weights of edges in $\Nt_m$ range from $q^0=1$ to $q^{m-2}$. An example of the network $\Nt_m$ for $m=4$ is shown in Figure~\ref{fig:Nt_m}.

\begin{figure}

\scalebox{0.4}{
\begin{tikzpicture}[scale=3.0]
\tikzset{myptr/.style={decoration={markings,mark=at position 1 with %
    {\arrow[scale=3,>=stealth]{>}}},postaction={decorate}}}
\node[draw=none] (n4x6) at (4.00,6.00) {};
\node[draw,circle,scale=0.8] (n4x7) at (4.00,7.00) {};
\node[draw=none] (n5x5) at (5.00,5.00) {};
\node[draw,circle,scale=0.8] (n5x6) at (5.00,6.00) {};
\node[draw,circle,scale=0.8] (n5x7) at (5.00,7.00) {};
\node[draw,circle,scale=0.8] (n5x8) at (5.00,8.00) {};
\node[draw,circle,scale=0.8] (n6x6) at (6.00,6.00) {};
\node[draw,circle,scale=0.8] (n6x7) at (6.00,7.00) {};
\node[draw,circle,scale=0.8] (n6x8) at (6.00,8.00) {};
\node[draw,circle,scale=0.8] (n6x9) at (6.00,9.00) {};
\node[draw,circle,scale=0.8] (n7x7) at (7.00,7.00) {};
\node[draw,circle,scale=0.8] (n7x8) at (7.00,8.00) {};
\node[draw,circle,scale=0.8] (n7x9) at (7.00,9.00) {};
\node[draw,circle,scale=0.8] (n7x10) at (7.00,10.00) {};
\node[draw,circle,scale=0.8] (n8x8) at (8.00,8.00) {};
\node[draw,circle,scale=0.8] (n8x9) at (8.00,9.00) {};
\node[draw,circle,scale=0.8] (n8x10) at (8.00,10.00) {};
\node[draw,circle,scale=0.8] (n8x11) at (8.00,11.00) {};
\node[draw,circle,scale=0.8] (n9x9) at (9.00,9.00) {};
\node[draw,circle,scale=0.8] (n9x10) at (9.00,10.00) {};
\node[draw=none] (n9x11) at (9.00,11.00) {};
\node[draw=none] (n10x10) at (10.00,10.00) {};
\draw[myptr] (n4x7) to node[midway,right,scale=2] {$q^2$} (n4x6);
\draw[myptr] (n5x6) to node[midway,above,scale=2] {$1$} (n4x6);
\draw[myptr] (n5x7) to node[midway,above,scale=2] {$1$} (n4x7);
\draw[myptr] (n5x6) to node[midway,right,scale=2] {$q^0$} (n5x5);
\draw[myptr] (n5x7) to node[midway,right,scale=2] {$q^1$} (n5x6);
\draw[myptr] (n6x6) to node[midway,above,scale=2] {$1$} (n5x6);
\draw[myptr] (n5x8) to node[midway,right,scale=2] {$q^2$} (n5x7);
\draw[myptr] (n6x7) to node[midway,above,scale=2] {$1$} (n5x7);
\draw[myptr] (n6x8) to node[midway,above,scale=2] {$1$} (n5x8);
\draw[myptr] (n6x7) to node[midway,right,scale=2] {$q^0$} (n6x6);
\draw[myptr] (n6x8) to node[midway,right,scale=2] {$q^1$} (n6x7);
\draw[myptr] (n7x7) to node[midway,above,scale=2] {$1$} (n6x7);
\draw[myptr] (n6x9) to node[midway,right,scale=2] {$q^2$} (n6x8);
\draw[myptr] (n7x8) to node[midway,above,scale=2] {$1$} (n6x8);
\draw[myptr] (n7x9) to node[midway,above,scale=2] {$1$} (n6x9);
\draw[myptr] (n7x8) to node[midway,right,scale=2] {$q^0$} (n7x7);
\draw[myptr] (n7x9) to node[midway,right,scale=2] {$q^1$} (n7x8);
\draw[myptr] (n8x8) to node[midway,above,scale=2] {$1$} (n7x8);
\draw[myptr] (n7x10) to node[midway,right,scale=2] {$q^2$} (n7x9);
\draw[myptr] (n8x9) to node[midway,above,scale=2] {$1$} (n7x9);
\draw[myptr] (n8x10) to node[midway,above,scale=2] {$1$} (n7x10);
\draw[myptr] (n8x9) to node[midway,right,scale=2] {$q^0$} (n8x8);
\draw[myptr] (n8x10) to node[midway,right,scale=2] {$q^1$} (n8x9);
\draw[myptr] (n9x9) to node[midway,above,scale=2] {$1$} (n8x9);
\draw[myptr] (n8x11) to node[midway,right,scale=2] {$q^2$} (n8x10);
\draw[myptr] (n9x10) to node[midway,above,scale=2] {$1$} (n8x10);
\draw[myptr] (n9x11) to node[midway,above,scale=2] {$1$} (n8x11);
\draw[myptr] (n9x10) to node[midway,right,scale=2] {$q^0$} (n9x9);
\draw[myptr] (n9x11) to node[midway,right,scale=2] {$q^1$} (n9x10);
\draw[myptr] (n10x10) to node[midway,above,scale=2] {$1$} (n9x10);
\node[scale=3,anchor=south west] (dotsUp) at (9.25,10.75) {$\iddots$};
\node[scale=3,anchor=north east] (dotsDown) at (4.00,6.00) {$\iddots$};
\draw[line width=1pt,myptr] (8.20,7.80) to node[midway,below,scale=3] {$\g$} (7.20,6.80);
\end{tikzpicture}}

	\caption{\label{fig:Nt_m}The cylindrical network $\Nt_m$ for $m=4$.}
\end{figure}
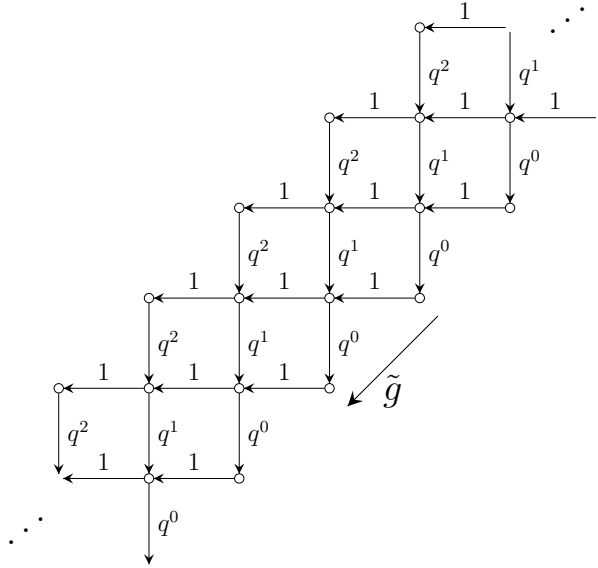

\begin{proposition}
	For each $a,b,c,d,r$, there exist two $r$-vertices $\ubft,\vbft$ in $\Nt_m$ such that for any $\ell\geq 1$, there is a bijection $\pi\to\bij(\pi)$ from the set of all weak reverse plane partitions of shape $\sk(a,b,c,d,\ell)$ with values from $0$ to $r$ to the set $\pathst(\ubft,\vbft+\ell\g)$ in $\Nt_m$. Moreover, there exist two integers $\a,\b$ depending on $a,b,c,d,r$ such that for all $\ell$ we have 
	\begin{equation}\label{eq:q}
	q^{|\pi|}=\wt(\bij(\pi)) q^\a q^{\ell\b}
	\end{equation}
	for any weak reverse plane partition $\pi$ of shape $\sk(a,b,c,d,\ell)$.
\end{proposition}
\begin{proof}

As Figure~\ref{fig:lozenges}~(c) demonstrates, each weak reverse plane partition (with parts bounded by $r$) corresponds to a lozenge tiling of a certain planar region. We refer the reader to~\cite[Figure~1]{CLP} for an explicit description of this bijection. As one can see from Figure~\ref{fig:lozenges}~(d), each such lozenge tiling corresponds to a unique $r$-path in $\Nt_m$ for $m=a+b+r$. The endpoints of this path are precisely $\ubft$ and $\vbft+\ell\g$ for some $\ubft,\vbft$ that do not depend on $\ell$. Thus the only claim we need to prove is~\eqref{eq:q}. It is clear that all lozenge tilings are connected by the local move shown in Figure~\ref{fig:local_move}. Moreover, this local move increases the power of $q$ by exactly $1$ in both $\pi$ and $\bij(\pi)$. It suffices to analyze the image of the reverse plane partition $\pi_0$ whose all parts are equal to zero, and it is straightforward to check that the weight of $\bij(\pi_0)$ is of the form $q^\a q^{\ell\b}$ for some $\a,\b\in\Z$. This finishes the proof.
	\begin{figure}
	\begin{center}
	\begin{tabular}{ccccc}
		
	    \scalebox{0.7}{
	    \begin{tikzpicture}[scale=1.0,yscale=-1]
	    \draw [] (0.0,0.0) rectangle (1.0,1.0); \node[anchor=center,scale=1] (nn0x0) at (0.50,0.50) {$i$};
	    \draw [fill=black, opacity=0.2] (1.0,0.0) -- (1.71,0.71) -- (1.71,1.71) -- (1.0,1.0) -- cycle;
	    \draw [fill=black, opacity=0.1] (0.0,1.0) -- (0.71,1.71) -- (1.71,1.71) -- (1.0,1.0) -- cycle;
	    \tikzset{myptr/.style={decoration={markings,mark=at position 1 with %
		{\arrow[scale=1.5,>=stealth]{>}}},postaction={decorate}}}
	    \draw[line width=1pt, red,myptr] (1.36,0.36) -- (1.36,1.36);
	    \draw[line width=1pt, red,myptr] (1.36,1.36) -- (0.36,1.36);
	    \draw[<->,line width=1pt] (2.71,0.8) -- (3.71,0.8);
	   \begin{scope}[shift={(4.71,0)}]
	  \draw [fill=black, opacity=0.2] (0.0,0.0) -- (0.71,0.71) -- (0.71,1.71) -- (0.0,1.0) -- cycle;
	  \draw [] (0.71,0.71) rectangle (1.71,1.71); \node[anchor=center,scale=1] (nn0x0) at (1.21,1.21) {$i+1$};
	  \draw [fill=black, opacity=0.1] (0.0,0.0) -- (0.71,0.71) -- (1.71,0.71) -- (1.0,0.0) -- cycle;
	  \tikzset{myptr/.style={decoration={markings,mark=at position 1 with %
	      {\arrow[scale=1.5,>=stealth]{>}}},postaction={decorate}}}
	  \draw[line width=1pt, red,myptr] (0.36,0.36) -- (0.36,1.36);
	  \draw[line width=1pt, red,myptr] (1.36,0.36) -- (0.36,0.36);
	   \end{scope}
	  \end{tikzpicture}}
	\end{tabular}
      \end{center}
      \caption{\label{fig:local_move}A local move that connects any two lozenge tilings with each other.}
      \end{figure}
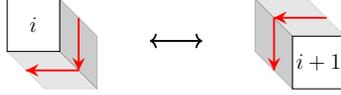
\end{proof}

Thus we can now apply Theorem~\ref{thm:recurrences} to the network $\Nt_m$ in order to obtain the main result of this section:

\begin{theorem}
	The sequence $f$ given by~\eqref{eq:f_lozenge} satisfies a linear recurrence with characteristic polynomial $Q_{N_m}^\plee{r}(tq^{\b})$.
\end{theorem}
\begin{proof}
	Follows immediately from the previous proposition combined with the second part of Theorem~\ref{thm:recurrences}.
\end{proof}

Let us now analyze the polynomial $Q_{N_m}(t)$. There are $m-1$ cycles $C_0,C_1,\dots,C_{m-2}$ in $N$ with respective weights $q^0,q^1,\dots,q^{m-2}$. Two cycles $C_i$ and $C_j$ are vertex disjoint if and only if $|i-j|>1$. Thus by~\eqref{eq:Q_N_planar_intro}, we get that the degree $d$ of $Q_{N_m}(t)$ is equal to $\lfloor m/2\rfloor$ and 
\[Q_{N_m}(t)=\sum_{r=0}^d (-t)^{d-r} H_r,\]
where 
\begin{equation}\label{eq:H_lozenge}
H_r=\sum_{(i_1,i_2,\dots,i_r)} q^{i_1+i_2+\dots+i_r}
\end{equation}
and the sum is taken over all $r$-tuples $0\leq i_1<i_2<\dots<i_r\leq m-2$ of integers satisfying $i_k+1<i_{k+1}$ for all $k=1,2,\dots,r-1$. For example, for the network $\Nt_4$ in Figure~\ref{fig:Nt_m}, we get 
\[Q_{N_4}(t)=t^2-(1+q+q^2)t+q^{0+2}.\]

It turns out that these polynomials have already been extensively studied under the name \emph{Carlitz $q$-Fibonacci polynomials}.

\begin{definition}[{see~\cite{Carlitz} or~\cite{Cigler}}]
	For $n\geq 0$, define the \emph{Carlitz $q$-Fibonacci polynomial} $F_n(t)$ by 
	\[F_0(t)=0,\quad F_1(t)=1,\quad F_n(t)=F_{n-1}(t)+q^{n-3}tF_{n-2}(t).\]
\end{definition}
The first few values of $F_n(t)$ are therefore
\[F_2(t)=1,\quad F_3(t)=t+1,\quad F_4(t)=(1+q)t+1,\quad F_5(t)=q^2t^2+(1+q+q^2)t+1.\]

\begin{proposition}
	For each $m\geq 2$, we have 
	\[Q_{N_m}(t)=(-t)^d F_{m+1}(-1/t),\]
	where $d=\lfloor m/2\rfloor$ is the degree of both polynomials.
\end{proposition}
\begin{proof}
	It follows from~\eqref{eq:H_lozenge} that both sides satisfy the same recurrence relation with the same initial conditions.
\end{proof}

\subsection{Domino tilings and the octahedron recurrence}\label{sect:domino}

In this section, we reprove (and slightly generalize) the results of~\cite[Section~3]{GP2} using our machinery. Let $m\geq 2$ be an integer, and consider the strip $\TS_m=\{(x,y)\in\R^2\mid 0\leq y\leq m\}$. We are going to be interested in domino tilings of regions inside $\TS_m$.

Let $n\geq 1$ be another integer and define the vector $\g=(2n,0)$. To each lattice point $(i,j)\in\Z^2, 0\leq j\leq m$ of $\TS_m$ we assign a \emph{weight} $w_{i,j}$ that satisfies $w_{i,j}=w_{i+2n,j}$ for all $(i,j)$ and $w_{0,j}=w_{m,j}=1$. For each lattice point $(i,j)\in\Z^2$ such that $0\leq i<2n$ and $0<j<m$, we put $w_{i,j}=x_{ij}$ to be an indeterminate and let $\x$ be the set of all these $2n(m-1)$ indeterminates. This defines $w_{i,j}$ for all lattice points $(i,j)\in\TS_m$. 

Given integers $i,j$, a \emph{square $S$ with center $(i+1/2,j+1/2)$} is the convex hull of four lattice points with coordinates $(i+1/2,j+1/2)+(\pm1/2,\pm1/2)$. We say that $S$ is \emph{white} (resp., black) if $i+j$ is even (resp., odd).

\def\AZ{\Acal}
\def\TAZ{{\Acal_m}}
\def\pol{Z_m}
For integers $i,j,\ell$, we define the \emph{Aztec diamond $\AZ(i,j,\ell)$} to be the union of all squares $S$ that are fully contained in the region 
\[\{(x,y)\in\R^2: |x-i|+|y-j|\leq \ell+1\}.\]
Thus for example $\AZ(i,j,1)$ is the union of four squares. Define the \emph{truncated Aztec diamond $\TAZ(i,j,\ell)$} to be the intersection of $\AZ(i,j,\ell)$ with $\TS_m$. A \emph{domino tiling of $\TAZ(i,j,\ell)$} is a covering of $\TAZ(i,j,\ell)$ by $2\times 1$ rectangles such that their vertices belong to $\Z^2$ and their interiors do not intersect each other. An example of a domino tiling of~$\TAZ(i,j,\ell)$ is given in Figure~\ref{fig:domino}. 

\begin{figure}
	
\scalebox{0.8}{
\begin{tikzpicture}[scale=1.0]
\draw[->] (-0.10,0.00) -- (13.00,0.00);
\draw[->] (0.00,-0.10) -- (0.00,7.00);
\node[anchor=north, scale=0.8] (bottom0) at (0.00,-0.10) {$0$};
\draw[] (0.00,-0.10) -- (0.00,0.10);
\node[anchor=north, scale=0.8] (bottom1) at (1.00,-0.10) {$1$};
\draw[] (1.00,-0.10) -- (1.00,0.10);
\node[anchor=north, scale=0.8] (bottom2) at (2.00,-0.10) {$2$};
\draw[] (2.00,-0.10) -- (2.00,0.10);
\node[anchor=north, scale=0.8] (bottom3) at (3.00,-0.10) {$3$};
\draw[] (3.00,-0.10) -- (3.00,0.10);
\node[anchor=north, scale=0.8] (bottom4) at (4.00,-0.10) {$4$};
\draw[] (4.00,-0.10) -- (4.00,0.10);
\node[anchor=north, scale=0.8] (bottom5) at (5.00,-0.10) {$5$};
\draw[] (5.00,-0.10) -- (5.00,0.10);
\node[anchor=north, scale=0.8] (bottom6) at (6.00,-0.10) {$6$};
\draw[] (6.00,-0.10) -- (6.00,0.10);
\node[anchor=north, scale=0.8] (bottom7) at (7.00,-0.10) {$i=7$};
\draw[] (7.00,-0.10) -- (7.00,0.10);
\node[anchor=north, scale=0.8] (bottom8) at (8.00,-0.10) {$8$};
\draw[] (8.00,-0.10) -- (8.00,0.10);
\node[anchor=north, scale=0.8] (bottom9) at (9.00,-0.10) {$9$};
\draw[] (9.00,-0.10) -- (9.00,0.10);
\node[anchor=north, scale=0.8] (bottom10) at (10.00,-0.10) {$10$};
\draw[] (10.00,-0.10) -- (10.00,0.10);
\node[anchor=north, scale=0.8] (bottom11) at (11.00,-0.10) {$11$};
\draw[] (11.00,-0.10) -- (11.00,0.10);
\node[anchor=north, scale=0.8] (bottom12) at (12.00,-0.10) {$12$};
\draw[] (12.00,-0.10) -- (12.00,0.10);
\node[anchor=east, scale=0.8] (left0) at (-0.10,0.00) {$0$};
\draw[] (-0.10,0.00) -- (0.10,0.00);
\node[anchor=east, scale=0.8] (left1) at (-0.10,1.00) {$1$};
\draw[] (-0.10,1.00) -- (0.10,1.00);
\node[anchor=east, scale=0.8] (left2) at (-0.10,2.00) {$2$};
\draw[] (-0.10,2.00) -- (0.10,2.00);
\node[anchor=east, scale=0.8] (left3) at (-0.10,3.00) {$j=3$};
\draw[] (-0.10,3.00) -- (0.10,3.00);
\node[anchor=east, scale=0.8] (left4) at (-0.10,4.00) {$4$};
\draw[] (-0.10,4.00) -- (0.10,4.00);
\node[anchor=east, scale=0.8] (left5) at (-0.10,5.00) {$m=5$};
\draw[] (-0.10,5.00) -- (0.10,5.00);
\node[anchor=north west] (i) at (13.00,-0.10) {$i$};
\node[anchor=south east] (j) at (-0.10,7.00) {$j$};
\draw [draw=none,fill=black,opacity=0.1] (0.0,1.0) rectangle (1.0,2.0); \draw [draw=none,fill=black,opacity=0.1] (0.0,3.0) rectangle (1.0,4.0); \draw [draw=none,fill=black,opacity=0.1] (1.0,0.0) rectangle (2.0,1.0); \draw [draw=none,fill=black,opacity=0.1] (1.0,2.0) rectangle (2.0,3.0); \draw [draw=none,fill=black,opacity=0.1] (1.0,4.0) rectangle (2.0,5.0); \draw [draw=none,fill=black,opacity=0.1] (2.0,1.0) rectangle (3.0,2.0); \draw [draw=none,fill=black,opacity=0.1] (2.0,3.0) rectangle (3.0,4.0); \draw [draw=none,fill=black,opacity=0.1] (3.0,0.0) rectangle (4.0,1.0); \draw [draw=none,fill=black,opacity=0.1] (3.0,2.0) rectangle (4.0,3.0); \draw [draw=none,fill=black,opacity=0.1] (3.0,4.0) rectangle (4.0,5.0); \draw [draw=none,fill=black,opacity=0.1] (4.0,1.0) rectangle (5.0,2.0); \draw [draw=none,fill=black,opacity=0.1] (4.0,3.0) rectangle (5.0,4.0); \draw [draw=none,fill=black,opacity=0.1] (5.0,0.0) rectangle (6.0,1.0); \draw [draw=none,fill=black,opacity=0.1] (5.0,2.0) rectangle (6.0,3.0); \draw [draw=none,fill=black,opacity=0.1] (5.0,4.0) rectangle (6.0,5.0); \draw [draw=none,fill=black,opacity=0.1] (6.0,1.0) rectangle (7.0,2.0); \draw [draw=none,fill=black,opacity=0.1] (6.0,3.0) rectangle (7.0,4.0); \draw [draw=none,fill=black,opacity=0.1] (7.0,0.0) rectangle (8.0,1.0); \draw [draw=none,fill=black,opacity=0.1] (7.0,2.0) rectangle (8.0,3.0); \draw [draw=none,fill=black,opacity=0.1] (7.0,4.0) rectangle (8.0,5.0); \draw [draw=none,fill=black,opacity=0.1] (8.0,1.0) rectangle (9.0,2.0); \draw [draw=none,fill=black,opacity=0.1] (8.0,3.0) rectangle (9.0,4.0); \draw [draw=none,fill=black,opacity=0.1] (9.0,0.0) rectangle (10.0,1.0); \draw [draw=none,fill=black,opacity=0.1] (9.0,2.0) rectangle (10.0,3.0); \draw [draw=none,fill=black,opacity=0.1] (9.0,4.0) rectangle (10.0,5.0); \draw [draw=none,fill=black,opacity=0.1] (10.0,1.0) rectangle (11.0,2.0); \draw [draw=none,fill=black,opacity=0.1] (10.0,3.0) rectangle (11.0,4.0); \draw [draw=none,fill=black,opacity=0.1] (11.0,0.0) rectangle (12.0,1.0); \draw [draw=none,fill=black,opacity=0.1] (11.0,2.0) rectangle (12.0,3.0); \draw [draw=none,fill=black,opacity=0.1] (11.0,4.0) rectangle (12.0,5.0); \draw[line width=0.5pt,dotted] (-0.10,5.00) -- (13.00,5.00);
\draw[line width=1.5pt] (7.00,-1.00) -- (6.00,-1.00);
\draw[line width=1.5pt] (6.00,-1.00) -- (6.00,0.00);
\draw[line width=1.5pt] (6.00,0.00) -- (5.00,0.00);
\draw[line width=1.5pt] (5.00,0.00) -- (5.00,1.00);
\draw[line width=1.5pt] (5.00,1.00) -- (4.00,1.00);
\draw[line width=1.5pt] (4.00,1.00) -- (4.00,2.00);
\draw[line width=1.5pt] (4.00,2.00) -- (3.00,2.00);
\draw[line width=1.5pt] (3.00,2.00) -- (3.00,3.00);
\draw[line width=1.5pt] (7.00,7.00) -- (6.00,7.00);
\draw[line width=1.5pt] (6.00,7.00) -- (6.00,6.00);
\draw[line width=1.5pt] (6.00,6.00) -- (5.00,6.00);
\draw[line width=1.5pt] (5.00,6.00) -- (5.00,5.00);
\draw[line width=1.5pt] (5.00,5.00) -- (4.00,5.00);
\draw[line width=1.5pt] (4.00,5.00) -- (4.00,4.00);
\draw[line width=1.5pt] (4.00,4.00) -- (3.00,4.00);
\draw[line width=1.5pt] (3.00,4.00) -- (3.00,3.00);
\draw[line width=1.5pt] (7.00,-1.00) -- (8.00,-1.00);
\draw[line width=1.5pt] (8.00,-1.00) -- (8.00,0.00);
\draw[line width=1.5pt] (8.00,0.00) -- (9.00,0.00);
\draw[line width=1.5pt] (9.00,0.00) -- (9.00,1.00);
\draw[line width=1.5pt] (9.00,1.00) -- (10.00,1.00);
\draw[line width=1.5pt] (10.00,1.00) -- (10.00,2.00);
\draw[line width=1.5pt] (10.00,2.00) -- (11.00,2.00);
\draw[line width=1.5pt] (11.00,2.00) -- (11.00,3.00);
\draw[line width=1.5pt] (7.00,7.00) -- (8.00,7.00);
\draw[line width=1.5pt] (8.00,7.00) -- (8.00,6.00);
\draw[line width=1.5pt] (8.00,6.00) -- (9.00,6.00);
\draw[line width=1.5pt] (9.00,6.00) -- (9.00,5.00);
\draw[line width=1.5pt] (9.00,5.00) -- (10.00,5.00);
\draw[line width=1.5pt] (10.00,5.00) -- (10.00,4.00);
\draw[line width=1.5pt] (10.00,4.00) -- (11.00,4.00);
\draw[line width=1.5pt] (11.00,4.00) -- (11.00,3.00);
\draw [rounded corners, line width=0.5pt] (4.025,4.025) rectangle (5.975,4.975); \draw [rounded corners, line width=0.5pt] (3.025,3.025) rectangle (4.975,3.975); \draw [rounded corners, line width=0.5pt] (3.025,2.025) rectangle (4.975,2.975); \draw [rounded corners, line width=0.5pt] (4.025,1.025) rectangle (5.975,1.975); \draw [rounded corners, line width=0.5pt] (5.025,0.025) rectangle (6.975,0.975); \draw [rounded corners, line width=0.5pt] (7.025,0.025) rectangle (8.975,0.975); \draw [rounded corners, line width=0.5pt] (7.025,3.025) rectangle (8.975,3.975); \draw [rounded corners, line width=0.5pt] (7.025,4.025) rectangle (8.975,4.975); \draw [rounded corners, line width=0.5pt] (6.025,3.025) rectangle (6.975,4.975); \draw [rounded corners, line width=0.5pt] (5.025,2.025) rectangle (5.975,3.975); \draw [rounded corners, line width=0.5pt] (6.025,1.025) rectangle (6.975,2.975); \draw [rounded corners, line width=0.5pt] (7.025,1.025) rectangle (7.975,2.975); \draw [rounded corners, line width=0.5pt] (8.025,1.025) rectangle (8.975,2.975); \draw [rounded corners, line width=0.5pt] (9.025,1.025) rectangle (9.975,2.975); \draw [rounded corners, line width=0.5pt] (10.025,2.025) rectangle (10.975,3.975); \draw [rounded corners, line width=0.5pt] (9.025,3.025) rectangle (9.975,4.975); \node[anchor=south west] (TAZ) at (8.00,7.00) {$\AZ(i,j,\ell)$};
\end{tikzpicture}}

	\caption{\label{fig:domino} A domino tiling of $\TAZ(i,j,\ell)$ for $i=7,j=3,\ell=4,$ and $m=5$.}
\end{figure}
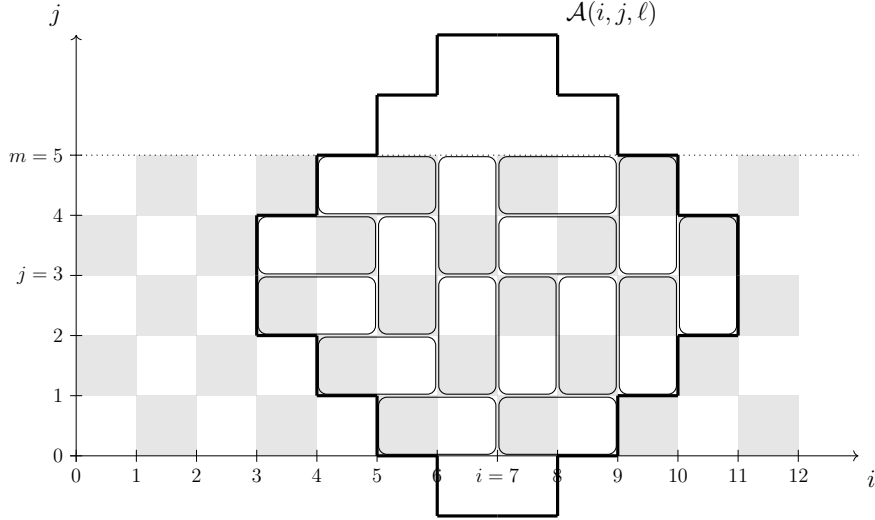

Given a domino tiling $T$ of $\TAZ(i,j,\ell)$, the \emph{weight} $\wt(T)$ of $T$ is the product over all lattice points $(i',j')$ in the interior of $\TAZ(i,j,\ell)$ of $w_{i',j'}^\e$, where $\e\in\{-1,0,+1\}$ is defined as follows:
\[\e= \begin{cases}
      	+1,&\text{if $(i',j')$ is adjacent to $4$ dominoes of $T$};\\
      	0,&\text{if $(i',j')$ is adjacent to $3$ dominoes of $T$};\\
      	-1,&\text{if $(i',j')$ is adjacent to $2$ dominoes of $T$}.
      \end{cases}\]
This assignment of weights was introduced by Speyer~\cite{Sp} to give a formula for the values of the octahedron recurrence. There is a more complicated rule that assigns the weights to vertices on the boundary of $\TAZ(i,j,\ell)$, however, we will just omit them from $\wt(T)$ for simplicity. 

Define $\pol(i,j,\ell)$ to be the sum of weights of all domino tilings of $\TAZ(i,j,\ell)$.

Let us now introduce the following planar cylindrical network $\Nt_{n,m}$. Its vertex set $\Vt$ is the set of all centers of black squares $S$ that lie fully inside $\TS_m$. Suppose that $S'$ is a black square with center $(i'+1/2,j'+1/2)$ and $S$ is a black square with center $(i+1/2,j+1/2)$ and that both of them belong to $\TS_m$. Then $\Nt_{n,m}$ contains an edge $e=(i'+1/2,j'+1/2)\to (i+1/2,j+1/2)$ if and only if either $i=i'+1,j=j'+1$ or $i=i'+2,j=j'$ or $i=i'+1,j=j'-1$. For each of these three cases, we assign the \emph{weight} $\wt(e)$ to $e$ as follows:
\[\wt(e)= \begin{cases}
          	\frac{w_{i+1,j+1}}{w_{i,j}},&\text{ if $i=i'+1,j=j'+1$};\\
          	\frac{w_{i+1,j+1}w_{i+1,j}}{w_{i,j}w_{i,j+1}},&\text{ if $i=i'+2,j=j'$};\\
          	\frac{w_{i+1,j}}{w_{i,j+1}},&\text{ if $i=i'+1,j=j'-1$}.
          \end{cases}\]
This defines the network $\Nt_{n,m}$. An example of $\Nt_{n,m}$ is given in Figure~\ref{fig:network_domino}. 

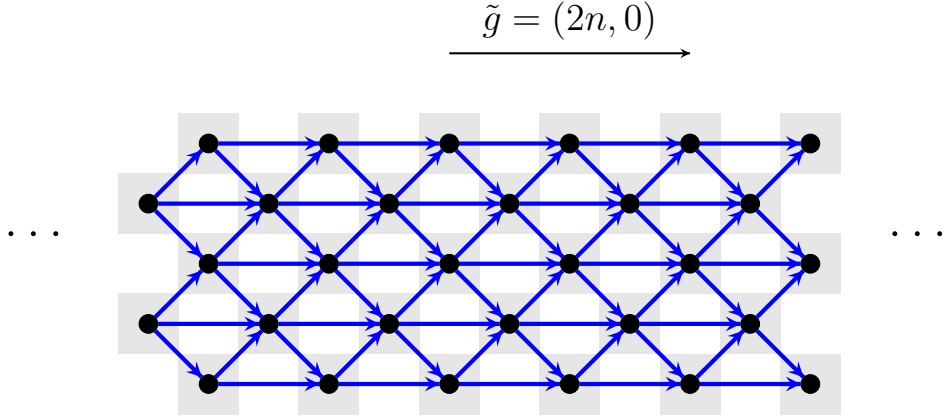
\begin{figure}
	
	\scalebox{0.8}{
\begin{tikzpicture}[scale=1.0]
\draw [draw=none,fill=black,opacity=0.1] (0.0,1.0) rectangle (1.0,2.0); \draw [draw=none,fill=black,opacity=0.1] (0.0,3.0) rectangle (1.0,4.0); \draw [draw=none,fill=black,opacity=0.1] (1.0,0.0) rectangle (2.0,1.0); \draw [draw=none,fill=black,opacity=0.1] (1.0,2.0) rectangle (2.0,3.0); \draw [draw=none,fill=black,opacity=0.1] (1.0,4.0) rectangle (2.0,5.0); \draw [draw=none,fill=black,opacity=0.1] (2.0,1.0) rectangle (3.0,2.0); \draw [draw=none,fill=black,opacity=0.1] (2.0,3.0) rectangle (3.0,4.0); \draw [draw=none,fill=black,opacity=0.1] (3.0,0.0) rectangle (4.0,1.0); \draw [draw=none,fill=black,opacity=0.1] (3.0,2.0) rectangle (4.0,3.0); \draw [draw=none,fill=black,opacity=0.1] (3.0,4.0) rectangle (4.0,5.0); \draw [draw=none,fill=black,opacity=0.1] (4.0,1.0) rectangle (5.0,2.0); \draw [draw=none,fill=black,opacity=0.1] (4.0,3.0) rectangle (5.0,4.0); \draw [draw=none,fill=black,opacity=0.1] (5.0,0.0) rectangle (6.0,1.0); \draw [draw=none,fill=black,opacity=0.1] (5.0,2.0) rectangle (6.0,3.0); \draw [draw=none,fill=black,opacity=0.1] (5.0,4.0) rectangle (6.0,5.0); \draw [draw=none,fill=black,opacity=0.1] (6.0,1.0) rectangle (7.0,2.0); \draw [draw=none,fill=black,opacity=0.1] (6.0,3.0) rectangle (7.0,4.0); \draw [draw=none,fill=black,opacity=0.1] (7.0,0.0) rectangle (8.0,1.0); \draw [draw=none,fill=black,opacity=0.1] (7.0,2.0) rectangle (8.0,3.0); \draw [draw=none,fill=black,opacity=0.1] (7.0,4.0) rectangle (8.0,5.0); \draw [draw=none,fill=black,opacity=0.1] (8.0,1.0) rectangle (9.0,2.0); \draw [draw=none,fill=black,opacity=0.1] (8.0,3.0) rectangle (9.0,4.0); \draw [draw=none,fill=black,opacity=0.1] (9.0,0.0) rectangle (10.0,1.0); \draw [draw=none,fill=black,opacity=0.1] (9.0,2.0) rectangle (10.0,3.0); \draw [draw=none,fill=black,opacity=0.1] (9.0,4.0) rectangle (10.0,5.0); \draw [draw=none,fill=black,opacity=0.1] (10.0,1.0) rectangle (11.0,2.0); \draw [draw=none,fill=black,opacity=0.1] (10.0,3.0) rectangle (11.0,4.0); \draw [draw=none,fill=black,opacity=0.1] (11.0,0.0) rectangle (12.0,1.0); \draw [draw=none,fill=black,opacity=0.1] (11.0,2.0) rectangle (12.0,3.0); \draw [draw=none,fill=black,opacity=0.1] (11.0,4.0) rectangle (12.0,5.0); \node[draw,fill=black,circle,scale=0.8] (n0x1) at (0.50,1.50) { };
\node[draw,fill=black,circle,scale=0.8] (n0x3) at (0.50,3.50) { };
\node[draw,fill=black,circle,scale=0.8] (n1x0) at (1.50,0.50) { };
\node[draw,fill=black,circle,scale=0.8] (n1x2) at (1.50,2.50) { };
\node[draw,fill=black,circle,scale=0.8] (n1x4) at (1.50,4.50) { };
\node[draw,fill=black,circle,scale=0.8] (n2x1) at (2.50,1.50) { };
\node[draw,fill=black,circle,scale=0.8] (n2x3) at (2.50,3.50) { };
\node[draw,fill=black,circle,scale=0.8] (n3x0) at (3.50,0.50) { };
\node[draw,fill=black,circle,scale=0.8] (n3x2) at (3.50,2.50) { };
\node[draw,fill=black,circle,scale=0.8] (n3x4) at (3.50,4.50) { };
\node[draw,fill=black,circle,scale=0.8] (n4x1) at (4.50,1.50) { };
\node[draw,fill=black,circle,scale=0.8] (n4x3) at (4.50,3.50) { };
\node[draw,fill=black,circle,scale=0.8] (n5x0) at (5.50,0.50) { };
\node[draw,fill=black,circle,scale=0.8] (n5x2) at (5.50,2.50) { };
\node[draw,fill=black,circle,scale=0.8] (n5x4) at (5.50,4.50) { };
\node[draw,fill=black,circle,scale=0.8] (n6x1) at (6.50,1.50) { };
\node[draw,fill=black,circle,scale=0.8] (n6x3) at (6.50,3.50) { };
\node[draw,fill=black,circle,scale=0.8] (n7x0) at (7.50,0.50) { };
\node[draw,fill=black,circle,scale=0.8] (n7x2) at (7.50,2.50) { };
\node[draw,fill=black,circle,scale=0.8] (n7x4) at (7.50,4.50) { };
\node[draw,fill=black,circle,scale=0.8] (n8x1) at (8.50,1.50) { };
\node[draw,fill=black,circle,scale=0.8] (n8x3) at (8.50,3.50) { };
\node[draw,fill=black,circle,scale=0.8] (n9x0) at (9.50,0.50) { };
\node[draw,fill=black,circle,scale=0.8] (n9x2) at (9.50,2.50) { };
\node[draw,fill=black,circle,scale=0.8] (n9x4) at (9.50,4.50) { };
\node[draw,fill=black,circle,scale=0.8] (n10x1) at (10.50,1.50) { };
\node[draw,fill=black,circle,scale=0.8] (n10x3) at (10.50,3.50) { };
\node[draw,fill=black,circle,scale=0.8] (n11x0) at (11.50,0.50) { };
\node[draw,fill=black,circle,scale=0.8] (n11x2) at (11.50,2.50) { };
\node[draw,fill=black,circle,scale=0.8] (n11x4) at (11.50,4.50) { };
\tikzset{myptr/.style={decoration={markings,mark=at position 1 with %
    {\arrow[scale=1,>=stealth]{>}}},postaction={decorate}}}
\draw[myptr,line width=2pt,blue] (n0x1) -- (n1x2);
\draw[myptr,line width=2pt,blue] (n0x1) -- (n1x0);
\draw[myptr,line width=2pt,blue] (n0x1) -- (n2x1);
\draw[myptr,line width=2pt,blue] (n0x3) -- (n1x4);
\draw[myptr,line width=2pt,blue] (n0x3) -- (n1x2);
\draw[myptr,line width=2pt,blue] (n0x3) -- (n2x3);
\draw[myptr,line width=2pt,blue] (n1x0) -- (n2x1);
\draw[myptr,line width=2pt,blue] (n1x0) -- (n3x0);
\draw[myptr,line width=2pt,blue] (n1x2) -- (n2x3);
\draw[myptr,line width=2pt,blue] (n1x2) -- (n2x1);
\draw[myptr,line width=2pt,blue] (n1x2) -- (n3x2);
\draw[myptr,line width=2pt,blue] (n1x4) -- (n2x3);
\draw[myptr,line width=2pt,blue] (n1x4) -- (n3x4);
\draw[myptr,line width=2pt,blue] (n2x1) -- (n3x2);
\draw[myptr,line width=2pt,blue] (n2x1) -- (n3x0);
\draw[myptr,line width=2pt,blue] (n2x1) -- (n4x1);
\draw[myptr,line width=2pt,blue] (n2x3) -- (n3x4);
\draw[myptr,line width=2pt,blue] (n2x3) -- (n3x2);
\draw[myptr,line width=2pt,blue] (n2x3) -- (n4x3);
\draw[myptr,line width=2pt,blue] (n3x0) -- (n4x1);
\draw[myptr,line width=2pt,blue] (n3x0) -- (n5x0);
\draw[myptr,line width=2pt,blue] (n3x2) -- (n4x3);
\draw[myptr,line width=2pt,blue] (n3x2) -- (n4x1);
\draw[myptr,line width=2pt,blue] (n3x2) -- (n5x2);
\draw[myptr,line width=2pt,blue] (n3x4) -- (n4x3);
\draw[myptr,line width=2pt,blue] (n3x4) -- (n5x4);
\draw[myptr,line width=2pt,blue] (n4x1) -- (n5x2);
\draw[myptr,line width=2pt,blue] (n4x1) -- (n5x0);
\draw[myptr,line width=2pt,blue] (n4x1) -- (n6x1);
\draw[myptr,line width=2pt,blue] (n4x3) -- (n5x4);
\draw[myptr,line width=2pt,blue] (n4x3) -- (n5x2);
\draw[myptr,line width=2pt,blue] (n4x3) -- (n6x3);
\draw[myptr,line width=2pt,blue] (n5x0) -- (n6x1);
\draw[myptr,line width=2pt,blue] (n5x0) -- (n7x0);
\draw[myptr,line width=2pt,blue] (n5x2) -- (n6x3);
\draw[myptr,line width=2pt,blue] (n5x2) -- (n6x1);
\draw[myptr,line width=2pt,blue] (n5x2) -- (n7x2);
\draw[myptr,line width=2pt,blue] (n5x4) -- (n6x3);
\draw[myptr,line width=2pt,blue] (n5x4) -- (n7x4);
\draw[myptr,line width=2pt,blue] (n6x1) -- (n7x2);
\draw[myptr,line width=2pt,blue] (n6x1) -- (n7x0);
\draw[myptr,line width=2pt,blue] (n6x1) -- (n8x1);
\draw[myptr,line width=2pt,blue] (n6x3) -- (n7x4);
\draw[myptr,line width=2pt,blue] (n6x3) -- (n7x2);
\draw[myptr,line width=2pt,blue] (n6x3) -- (n8x3);
\draw[myptr,line width=2pt,blue] (n7x0) -- (n8x1);
\draw[myptr,line width=2pt,blue] (n7x0) -- (n9x0);
\draw[myptr,line width=2pt,blue] (n7x2) -- (n8x3);
\draw[myptr,line width=2pt,blue] (n7x2) -- (n8x1);
\draw[myptr,line width=2pt,blue] (n7x2) -- (n9x2);
\draw[myptr,line width=2pt,blue] (n7x4) -- (n8x3);
\draw[myptr,line width=2pt,blue] (n7x4) -- (n9x4);
\draw[myptr,line width=2pt,blue] (n8x1) -- (n9x2);
\draw[myptr,line width=2pt,blue] (n8x1) -- (n9x0);
\draw[myptr,line width=2pt,blue] (n8x1) -- (n10x1);
\draw[myptr,line width=2pt,blue] (n8x3) -- (n9x4);
\draw[myptr,line width=2pt,blue] (n8x3) -- (n9x2);
\draw[myptr,line width=2pt,blue] (n8x3) -- (n10x3);
\draw[myptr,line width=2pt,blue] (n9x0) -- (n10x1);
\draw[myptr,line width=2pt,blue] (n9x0) -- (n11x0);
\draw[myptr,line width=2pt,blue] (n9x2) -- (n10x3);
\draw[myptr,line width=2pt,blue] (n9x2) -- (n10x1);
\draw[myptr,line width=2pt,blue] (n9x2) -- (n11x2);
\draw[myptr,line width=2pt,blue] (n9x4) -- (n10x3);
\draw[myptr,line width=2pt,blue] (n9x4) -- (n11x4);
\draw[myptr,line width=2pt,blue] (n10x1) -- (n11x2);
\draw[myptr,line width=2pt,blue] (n10x1) -- (n11x0);
\draw[myptr,line width=2pt,blue] (n10x3) -- (n11x4);
\draw[myptr,line width=2pt,blue] (n10x3) -- (n11x2);
\node[scale=2,anchor=east] (ldots) at (-0.50,3.00) {$\dots$};
\node[scale=2,anchor=west] (rdots) at (12.50,3.00) {$\dots$};
\draw[line width=1pt,myptr] (5.50,6.00) to node[midway,above,scale=1.5] {$\g=(2n,0)$} (9.50,6.00);
\end{tikzpicture}}

	\caption{\label{fig:network_domino} The network $\Nt_{n,m}$ for $m=5$ and $n=2$.}
\end{figure}

Now given a lattice point $(i,j)\in \TS_m$ and an integer $\ell$, there is a simple bijection that associates to each domino tiling $T$ of $\TAZ(i,j,\ell)$ an $r$-path $\bij(T)$ between two $r$-vertices $\ubft$ and $\vbft$. Here $r,\ubft,\vbft$ only depend on $i,j,\ell$ and not on $T$. 

We construct $\bij(T)$ using a local rule in Figure~\ref{fig:local_rules}. Consider a black square $S$ with center $(i'+1/2,j'+1/2)$ such that the white square $S'$ with center $(i'+3/2,j'+1/2)$ to the right of $S$ lies inside $\TAZ(i,j,\ell)$. There is exactly one domino $D$ of $T$ that contains $S'$, and there are four possibilities for the black square $S''$ contained in $D$. Let $(i''+1/2,j''+1/2)$ be the center of $S''$. Our local rule reads as follows:
\begin{enumerate}
	\item If $i''=i',j''=j'$ then there is no edge in $\bij(T)$ coming out of $(i'+1/2,j'+1/2)$;
	\item otherwise there an edge $(i'+1/2,j'+1/2)\to(i''+1/2,j''+1/2)$ in $\bij(T)$.
\end{enumerate}
See Figure~\ref{fig:local_rules} for an illustration. 

\setlength{\tabcolsep}{12pt}
\begin{figure}
\centering

\begin{tabular}{cccc}
\scalebox{0.8}{
\begin{tikzpicture}[scale=1.0]
\tikzset{myptr/.style={decoration={markings,mark=at position 1 with %
    {\arrow[scale=1,>=stealth]{>}}},postaction={decorate}}}
\draw [draw=none,fill=black,opacity=0.1] (0.0,0.0) rectangle (1.0,1.0); \node[draw,fill=black,circle,scale=0.8] (from) at (0.50,0.50) { };
\node[anchor=north west] (S) at (0.00,0.00) {$S$};
\draw [rounded corners, line width=0.5pt] (1.025,0.025) rectangle (2.975,0.975); \draw [draw=none,fill=black,opacity=0.1] (2.0,0.0) rectangle (3.0,1.0); \node[draw,fill=black,circle,scale=0.8] (to) at (2.50,0.50) { };
\draw[myptr,line width=2pt,blue] (from) -- (to);
\node[anchor=north] (SSS) at (2.50,0.00) {$S''$};
\node[anchor=north] (SS) at (1.50,0.00) {$S'$};
\end{tikzpicture}}
&
\scalebox{0.8}{
\begin{tikzpicture}[scale=1.0]
\tikzset{myptr/.style={decoration={markings,mark=at position 1 with %
    {\arrow[scale=1,>=stealth]{>}}},postaction={decorate}}}
\draw [draw=none,fill=black,opacity=0.1] (0.0,0.0) rectangle (1.0,1.0); \node[draw,fill=black,circle,scale=0.8] (from) at (0.50,0.50) { };
\node[anchor=north west] (S) at (0.00,0.00) {$S$};
\node[anchor=north] (SS) at (1.50,0.00) {$S'$};
\draw [rounded corners, line width=0.5pt] (1.025,0.025) rectangle (1.975,1.975); \draw [draw=none,fill=black,opacity=0.1] (1.0,1.0) rectangle (2.0,2.0); \node[draw,fill=black,circle,scale=0.8] (to) at (1.50,1.50) { };
\node[anchor=west] (SSS) at (2.00,1.50) {$S''$};
\draw[myptr,line width=2pt,blue] (from) -- (to);
\end{tikzpicture}}
&
\scalebox{0.8}{
\begin{tikzpicture}[scale=1.0]
\tikzset{myptr/.style={decoration={markings,mark=at position 1 with %
    {\arrow[scale=1,>=stealth]{>}}},postaction={decorate}}}
\draw [draw=none,fill=black,opacity=0.1] (0.0,0.0) rectangle (1.0,1.0); \node[draw,fill=black,circle,scale=0.8] (from) at (0.50,0.50) { };
\node[anchor=north west] (S) at (0.00,0.00) {$S$};
\node[anchor=south] (SS) at (1.50,1.00) {$S'$};
\draw [rounded corners, line width=0.5pt] (1.025,-0.975) rectangle (1.975,0.975); \draw [draw=none,fill=black,opacity=0.1] (1.0,-1.0) rectangle (2.0,0.0); \node[draw,fill=black,circle,scale=0.8] (to) at (1.50,-0.50) { };
\draw[myptr,line width=2pt,blue] (from) -- (to);
\node[anchor=west] (SSS) at (2.00,-0.50) {$S''$};
\end{tikzpicture}}
&
\scalebox{0.8}{
\begin{tikzpicture}[scale=1.0]
\tikzset{myptr/.style={decoration={markings,mark=at position 1 with %
    {\arrow[scale=1,>=stealth]{>}}},postaction={decorate}}}
\draw [draw=none,fill=black,opacity=0.1] (0.0,0.0) rectangle (1.0,1.0); \node[draw,fill=black,circle,scale=0.8] (from) at (0.50,0.50) { };
\node[anchor=west] (SS) at (2.00,0.50) {$S'$};
\draw [rounded corners, line width=0.5pt] (0.025,0.025) rectangle (1.975,0.975); \node[draw,fill=black,circle,scale=0.8] (to) at (0.50,0.50) { };
\node[anchor=north east] (S) at (1.00,0.00) {$S=S''$};
\end{tikzpicture}}
\\

\end{tabular}

	\caption{\label{fig:local_rules} Local rules for constructing an $r$-path $\bij(T)$ in $\Nt_{n,m}$ from a given domino tiling $T$.}
\end{figure}
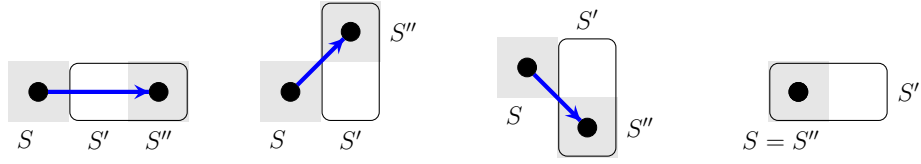

It is easy to see that in the second case, the edge $(i'+1/2,j'+1/2)\to(i''+1/2,j''+1/2)$ always is an edge of $\Nt_{n,m}$. Thus our local rule defines $\bij(T)$ as a collection of edges of $\Nt_{n,m}$ and it follows that every vertex $(i'+1/2,j'+1/2)$ in the interior of $\TAZ(i,j,\ell-1)$ either is isolated in $\bij(T)$ or has indegree and outdegree $1$. Thus $\bij(T)$ is an $r$-path for some $r\geq 0$. Clearly, the start $\ubft$ and end $\vbft$ of $\bij(T)$ do not actually depend on $T$ and thus the same is true for $r$. It is also straightforward to check that given any $r$-path $\Pbft$ in $\Nt_{n,m}$ that starts at $\ubft$ and ends at $\vbft$, there is a unique domino tiling $T$ of $\TAZ(i,j,\ell)$ such that $\Pbft=\Pbft(T)$. For example, if $T$ is the domino tiling from Figure~\ref{fig:domino} then $\bij(T)$ is shown in Figure~\ref{fig:together}.

\begin{figure}
	\scalebox{0.8}{
\begin{tikzpicture}[scale=1.0]
\draw [draw=none,fill=black,opacity=0.1] (0.0,1.0) rectangle (1.0,2.0); \draw [draw=none,fill=black,opacity=0.1] (0.0,3.0) rectangle (1.0,4.0); \draw [draw=none,fill=black,opacity=0.1] (1.0,0.0) rectangle (2.0,1.0); \draw [draw=none,fill=black,opacity=0.1] (1.0,2.0) rectangle (2.0,3.0); \draw [draw=none,fill=black,opacity=0.1] (1.0,4.0) rectangle (2.0,5.0); \draw [draw=none,fill=black,opacity=0.1] (2.0,1.0) rectangle (3.0,2.0); \draw [draw=none,fill=black,opacity=0.1] (2.0,3.0) rectangle (3.0,4.0); \draw [draw=none,fill=black,opacity=0.1] (3.0,0.0) rectangle (4.0,1.0); \draw [draw=none,fill=black,opacity=0.1] (3.0,2.0) rectangle (4.0,3.0); \draw [draw=none,fill=black,opacity=0.1] (3.0,4.0) rectangle (4.0,5.0); \draw [draw=none,fill=black,opacity=0.1] (4.0,1.0) rectangle (5.0,2.0); \draw [draw=none,fill=black,opacity=0.1] (4.0,3.0) rectangle (5.0,4.0); \draw [draw=none,fill=black,opacity=0.1] (5.0,0.0) rectangle (6.0,1.0); \draw [draw=none,fill=black,opacity=0.1] (5.0,2.0) rectangle (6.0,3.0); \draw [draw=none,fill=black,opacity=0.1] (5.0,4.0) rectangle (6.0,5.0); \draw [draw=none,fill=black,opacity=0.1] (6.0,1.0) rectangle (7.0,2.0); \draw [draw=none,fill=black,opacity=0.1] (6.0,3.0) rectangle (7.0,4.0); \draw [draw=none,fill=black,opacity=0.1] (7.0,0.0) rectangle (8.0,1.0); \draw [draw=none,fill=black,opacity=0.1] (7.0,2.0) rectangle (8.0,3.0); \draw [draw=none,fill=black,opacity=0.1] (7.0,4.0) rectangle (8.0,5.0); \draw [draw=none,fill=black,opacity=0.1] (8.0,1.0) rectangle (9.0,2.0); \draw [draw=none,fill=black,opacity=0.1] (8.0,3.0) rectangle (9.0,4.0); \draw [draw=none,fill=black,opacity=0.1] (9.0,0.0) rectangle (10.0,1.0); \draw [draw=none,fill=black,opacity=0.1] (9.0,2.0) rectangle (10.0,3.0); \draw [draw=none,fill=black,opacity=0.1] (9.0,4.0) rectangle (10.0,5.0); \draw [draw=none,fill=black,opacity=0.1] (10.0,1.0) rectangle (11.0,2.0); \draw [draw=none,fill=black,opacity=0.1] (10.0,3.0) rectangle (11.0,4.0); \draw [draw=none,fill=black,opacity=0.1] (11.0,0.0) rectangle (12.0,1.0); \draw [draw=none,fill=black,opacity=0.1] (11.0,2.0) rectangle (12.0,3.0); \draw [draw=none,fill=black,opacity=0.1] (11.0,4.0) rectangle (12.0,5.0); \draw[line width=0.5pt,dotted] (-0.10,5.00) -- (13.00,5.00);
\draw[->] (-0.10,0.00) -- (13.00,0.00);
\draw[->] (0.00,-0.10) -- (0.00,7.00);
\node[anchor=north, scale=0.8] (bottom0) at (0.00,-0.10) {$0$};
\draw[] (0.00,-0.10) -- (0.00,0.10);
\node[anchor=north, scale=0.8] (bottom1) at (1.00,-0.10) {$1$};
\draw[] (1.00,-0.10) -- (1.00,0.10);
\node[anchor=north, scale=0.8] (bottom2) at (2.00,-0.10) {$2$};
\draw[] (2.00,-0.10) -- (2.00,0.10);
\node[anchor=north, scale=0.8] (bottom3) at (3.00,-0.10) {$3$};
\draw[] (3.00,-0.10) -- (3.00,0.10);
\node[anchor=north, scale=0.8] (bottom4) at (4.00,-0.10) {$4$};
\draw[] (4.00,-0.10) -- (4.00,0.10);
\node[anchor=north, scale=0.8] (bottom5) at (5.00,-0.10) {$5$};
\draw[] (5.00,-0.10) -- (5.00,0.10);
\node[anchor=north, scale=0.8] (bottom6) at (6.00,-0.10) {$6$};
\draw[] (6.00,-0.10) -- (6.00,0.10);
\node[anchor=north, scale=0.8] (bottom7) at (7.00,-0.10) {$i=7$};
\draw[] (7.00,-0.10) -- (7.00,0.10);
\node[anchor=north, scale=0.8] (bottom8) at (8.00,-0.10) {$8$};
\draw[] (8.00,-0.10) -- (8.00,0.10);
\node[anchor=north, scale=0.8] (bottom9) at (9.00,-0.10) {$9$};
\draw[] (9.00,-0.10) -- (9.00,0.10);
\node[anchor=north, scale=0.8] (bottom10) at (10.00,-0.10) {$10$};
\draw[] (10.00,-0.10) -- (10.00,0.10);
\node[anchor=north, scale=0.8] (bottom11) at (11.00,-0.10) {$11$};
\draw[] (11.00,-0.10) -- (11.00,0.10);
\node[anchor=north, scale=0.8] (bottom12) at (12.00,-0.10) {$12$};
\draw[] (12.00,-0.10) -- (12.00,0.10);
\node[anchor=east, scale=0.8] (left0) at (-0.10,0.00) {$0$};
\draw[] (-0.10,0.00) -- (0.10,0.00);
\node[anchor=east, scale=0.8] (left1) at (-0.10,1.00) {$1$};
\draw[] (-0.10,1.00) -- (0.10,1.00);
\node[anchor=east, scale=0.8] (left2) at (-0.10,2.00) {$2$};
\draw[] (-0.10,2.00) -- (0.10,2.00);
\node[anchor=east, scale=0.8] (left3) at (-0.10,3.00) {$j=3$};
\draw[] (-0.10,3.00) -- (0.10,3.00);
\node[anchor=east, scale=0.8] (left4) at (-0.10,4.00) {$4$};
\draw[] (-0.10,4.00) -- (0.10,4.00);
\node[anchor=east, scale=0.8] (left5) at (-0.10,5.00) {$m=5$};
\draw[] (-0.10,5.00) -- (0.10,5.00);
\node[anchor=north west] (i) at (13.00,-0.10) {$i$};
\node[anchor=south east] (j) at (-0.10,7.00) {$j$};
\draw[line width=1.5pt] (7.00,-1.00) -- (6.00,-1.00);
\draw[line width=1.5pt] (6.00,-1.00) -- (6.00,0.00);
\draw[line width=1.5pt] (6.00,0.00) -- (5.00,0.00);
\draw[line width=1.5pt] (5.00,0.00) -- (5.00,1.00);
\draw[line width=1.5pt] (5.00,1.00) -- (4.00,1.00);
\draw[line width=1.5pt] (4.00,1.00) -- (4.00,2.00);
\draw[line width=1.5pt] (4.00,2.00) -- (3.00,2.00);
\draw[line width=1.5pt] (3.00,2.00) -- (3.00,3.00);
\draw[line width=1.5pt] (7.00,7.00) -- (6.00,7.00);
\draw[line width=1.5pt] (6.00,7.00) -- (6.00,6.00);
\draw[line width=1.5pt] (6.00,6.00) -- (5.00,6.00);
\draw[line width=1.5pt] (5.00,6.00) -- (5.00,5.00);
\draw[line width=1.5pt] (5.00,5.00) -- (4.00,5.00);
\draw[line width=1.5pt] (4.00,5.00) -- (4.00,4.00);
\draw[line width=1.5pt] (4.00,4.00) -- (3.00,4.00);
\draw[line width=1.5pt] (3.00,4.00) -- (3.00,3.00);
\draw[line width=1.5pt] (7.00,-1.00) -- (8.00,-1.00);
\draw[line width=1.5pt] (8.00,-1.00) -- (8.00,0.00);
\draw[line width=1.5pt] (8.00,0.00) -- (9.00,0.00);
\draw[line width=1.5pt] (9.00,0.00) -- (9.00,1.00);
\draw[line width=1.5pt] (9.00,1.00) -- (10.00,1.00);
\draw[line width=1.5pt] (10.00,1.00) -- (10.00,2.00);
\draw[line width=1.5pt] (10.00,2.00) -- (11.00,2.00);
\draw[line width=1.5pt] (11.00,2.00) -- (11.00,3.00);
\draw[line width=1.5pt] (7.00,7.00) -- (8.00,7.00);
\draw[line width=1.5pt] (8.00,7.00) -- (8.00,6.00);
\draw[line width=1.5pt] (8.00,6.00) -- (9.00,6.00);
\draw[line width=1.5pt] (9.00,6.00) -- (9.00,5.00);
\draw[line width=1.5pt] (9.00,5.00) -- (10.00,5.00);
\draw[line width=1.5pt] (10.00,5.00) -- (10.00,4.00);
\draw[line width=1.5pt] (10.00,4.00) -- (11.00,4.00);
\draw[line width=1.5pt] (11.00,4.00) -- (11.00,3.00);
\draw [rounded corners, line width=0.5pt] (4.025,4.025) rectangle (5.975,4.975); \draw [rounded corners, line width=0.5pt] (3.025,3.025) rectangle (4.975,3.975); \draw [rounded corners, line width=0.5pt] (3.025,2.025) rectangle (4.975,2.975); \draw [rounded corners, line width=0.5pt] (4.025,1.025) rectangle (5.975,1.975); \draw [rounded corners, line width=0.5pt] (5.025,0.025) rectangle (6.975,0.975); \draw [rounded corners, line width=0.5pt] (7.025,0.025) rectangle (8.975,0.975); \draw [rounded corners, line width=0.5pt] (7.025,3.025) rectangle (8.975,3.975); \draw [rounded corners, line width=0.5pt] (7.025,4.025) rectangle (8.975,4.975); \draw [rounded corners, line width=0.5pt] (6.025,3.025) rectangle (6.975,4.975); \draw [rounded corners, line width=0.5pt] (5.025,2.025) rectangle (5.975,3.975); \draw [rounded corners, line width=0.5pt] (6.025,1.025) rectangle (6.975,2.975); \draw [rounded corners, line width=0.5pt] (7.025,1.025) rectangle (7.975,2.975); \draw [rounded corners, line width=0.5pt] (8.025,1.025) rectangle (8.975,2.975); \draw [rounded corners, line width=0.5pt] (9.025,1.025) rectangle (9.975,2.975); \draw [rounded corners, line width=0.5pt] (10.025,2.025) rectangle (10.975,3.975); \draw [rounded corners, line width=0.5pt] (9.025,3.025) rectangle (9.975,4.975); \node[anchor=south west] (TAZ) at (8.00,7.00) {$\AZ(i,j,\ell)$};
\tikzset{myptr/.style={decoration={markings,mark=at position 1 with %
    {\arrow[scale=1,>=stealth]{>}}},postaction={decorate}}}
\draw [fill=white,opacity=0.3,draw=none] (-1.0,-1.0) rectangle (14.0,8.0); \node[draw,fill=black,circle,scale=0.8] (n3x4) at (3.50,4.50) { };
\node[anchor=east,scale=1.0] (u1) at (3.50,4.50) {$\ut_1$};
\node[draw,fill=black,circle,scale=0.8] (n5x4) at (5.50,4.50) { };
\draw[myptr,line width=2pt,blue] (n3x4) -- (n5x4);
\node[draw,fill=black,circle,scale=0.8] (n6x3) at (6.50,3.50) { };
\draw[myptr,line width=2pt,blue] (n5x4) -- (n6x3);
\node[draw,fill=black,circle,scale=0.8] (n8x3) at (8.50,3.50) { };
\draw[myptr,line width=2pt,blue] (n6x3) -- (n8x3);
\node[draw,fill=black,circle,scale=0.8] (n9x4) at (9.50,4.50) { };
\draw[myptr,line width=2pt,blue] (n8x3) -- (n9x4);
\node[draw,fill=black,circle,scale=0.8] (n2x3) at (2.50,3.50) { };
\node[anchor=north,inner sep=9pt,scale=1.0] (v2) at (9.50,4.50) {$\vt_1$};
\node[anchor=east,scale=1.0] (u2) at (2.50,3.50) {$\ut_0$};
\node[draw,fill=black,circle,scale=0.8] (n4x3) at (4.50,3.50) { };
\draw[myptr,line width=2pt,blue] (n2x3) -- (n4x3);
\node[draw,fill=black,circle,scale=0.8] (n5x2) at (5.50,2.50) { };
\draw[myptr,line width=2pt,blue] (n4x3) -- (n5x2);
\node[draw,fill=black,circle,scale=0.8] (n6x1) at (6.50,1.50) { };
\draw[myptr,line width=2pt,blue] (n5x2) -- (n6x1);
\node[draw,fill=black,circle,scale=0.8] (n7x2) at (7.50,2.50) { };
\draw[myptr,line width=2pt,blue] (n6x1) -- (n7x2);
\node[draw,fill=black,circle,scale=0.8] (n8x1) at (8.50,1.50) { };
\draw[myptr,line width=2pt,blue] (n7x2) -- (n8x1);
\node[draw,fill=black,circle,scale=0.8] (n9x2) at (9.50,2.50) { };
\draw[myptr,line width=2pt,blue] (n8x1) -- (n9x2);
\node[draw,fill=black,circle,scale=0.8] (n10x3) at (10.50,3.50) { };
\draw[myptr,line width=2pt,blue] (n9x2) -- (n10x3);
\node[anchor=north,inner sep=9pt,scale=1.0] (u2) at (10.50,3.50) {$\vt_0$};
\end{tikzpicture}}

	\caption{\label{fig:together} The $r$-path $\bij(T)$ where $T$ is the domino tiling from Figure~\ref{fig:domino}.}
\end{figure}
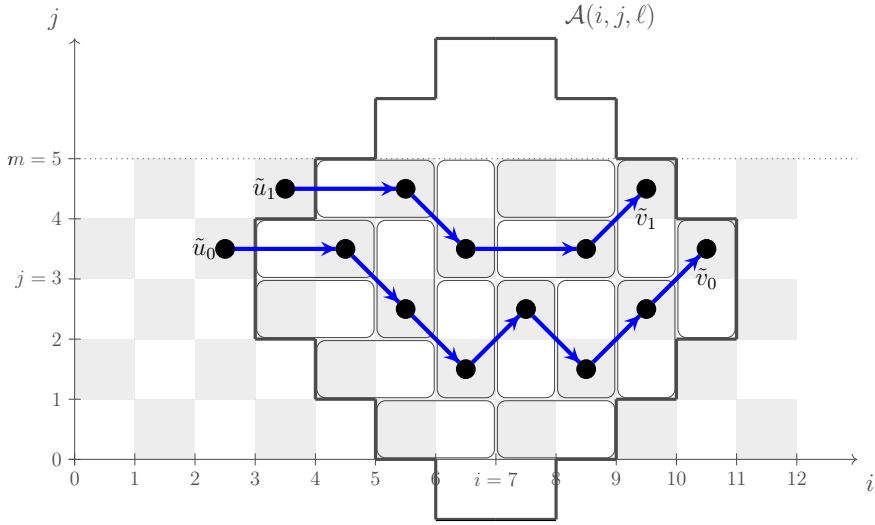

Let us describe the vertices $\ubft,\vbft$ explicitly as functions of $i,j,\ell$. By symmetry, we may assume that $i+j-\ell$ is even. Then the square $S$ with center $(i-\ell+1/2,j+1/2)$ on the boundary of $\TAZ(i,j,\ell)$ is white. Let $r$ be the minimum of $m-j$ and $\ell$. Then it is easy to check that the start $\ubft=(\ut_0,\ut_1,\dots,\ut_{r-1})$ and end $\vbft=(\vt_0,\vt_1,\dots,\vt_{r-1})$ of $\bij(T)$ are given by
\begin{equation}\label{eq:ut_k_vt_k}
\ut_k=(i-\ell+k-1/2,j+k+1/2),\quad \vt_k=(i+\ell-k-1/2,j+k+1/2)
\end{equation}
for $k=0,1,\dots,r-1$. In particular, the difference $\vt_k-\ut_k$ equals $(2\ell-2k,0)$.

\def\mon{\tau}
\begin{theorem}\label{thm:domino}
	For any lattice point $(i,j)$ of $\TS_m$ and any integer $\ell_0$, the sequence $f:\N\to\Z[\x]$ given by 
	\[f(\ell)=\pol(i+\ell,j,\ell n+\ell_0)\]
	satisfies a linear recurrence with characteristic polynomial $Q_{N_{n,m}}^\plee{m-j}(t)$ for all but finitely many values of $\ell$. Here the polynomial $Q_{N_{n,m}}(t)$ is given by~\eqref{eq:Q_N_planar_intro}.
\end{theorem}
\begin{proof}
	We will deduce the result from the second part of Theorem~\ref{thm:recurrences}. In order to do so, we need to show that 
	\begin{enumerate}[(i)]
		\item there exist vertices $\ubft$ and $\vbft_0$ such that for any domino tiling $T$ of $\TAZ(i+\ell,j,\ell n+\ell_0)$, the $r$-path $\bij(T)$ starts at $\ubft$ and ends at $\vbft_\ell:=\vbft_0+\ell\g$.
		\item the only permutation $\sigma\in\Sfr_r$ such that there exists an $r$-path from $\ubft$ to $\sigma\vbft_\ell$ is the identity;
		\item there is a monomial $\mon$ in $\x^{\pm1}$ that depends only on $i,j,\ell_0$ such that we have $\wt(T)=\mon \wt(\bij(T))$.
	\end{enumerate}
	The first claim follows from~\eqref{eq:ut_k_vt_k}. The second claim is obvious by inspection. The third claim can be proved as follows. First note that by Thurston's theorem~\cite{Th}, all domino tilings of $\TAZ(i+\ell,j,\ell n+\ell_0)$ are connected by flips and that if $T'$ is obtained from $T$ by applying a flip then $\wt(\bij(T))=\wt(\bij(T'))$ which is easy to check directly. Thus we can put $\mon$ to be the ratio of weights of $T_0$ and $\bij(T_0)$ for some fixed tiling $T_0$ of $\TAZ(i+\ell,j,\ell n+\ell_0)$. Let $T_0$ be the unique tiling that consists entirely of horizontal dominoes. Then $\wt(T_0)=1$ while $\wt(\bij(T_0))$ only contains the vertices that belong to the boundary of $\TAZ(i+\ell,j,\ell n+\ell_0)$ and therefore does not depend on $\ell$. Thus $\mon$ does not depend on $\ell$ either and we are done with the proof.
\end{proof}

It remains to note that one can give a nice combinatorial interpretation to the coefficients of $Q_{N_{n,m}}(t)$ in terms of \emph{domino tilings of the cylinder}. Namely, by~\eqref{eq:Q_N_planar_intro}, the coefficient of $(-t)^{d-r}$ is the sum $H_r$ of weights of all $r$-cycles in $N_{n,m}$. Applying the inverse of the local rules in Figure~\ref{fig:local_rules} to any such $r$-cycle yields a domino tiling of the cylinder with \emph{Thurston height} equal to $r$. Thus $H_r$ can be interpreted as the sum of weights of all domino tilings of the cylinder with a fixed Thurston height. We refer the reader to~\cite[Section~3]{GP2} for the details.

\begin{remark}
	One can extend this approach to more general regions inside $\TS_m$. Namely, let $R$ be a region inside $\TS_m$ that lies between two paths $P_l$ and $P_r$ that connect the upper boundary of $\TS_m$ with the lower boundary of $\TS_m$ and consist of left, right, and down steps. Define the region $R_\ell$ to be the region between $P_l$ and $P_r+\ell\g$ for all $\ell\geq 0$. Suppose in addition that there exists a domino tiling of $R_\ell$ for each sufficiently large $\ell$. Then the sum $f(\ell)$ of weights of domino tilings of $R_\ell$ satisfies a linear recurrence with characteristic polynomial $Q_{N_{n,m}}^\plee{r}(t)$ for some $r\geq 0$. The proof is analogous to that of Theorem~\ref{thm:domino}.
\end{remark}

\begin{remark}
	If $n=1$ then $N$ is an $m\times 2$ cylinder. There are $2m-1$ cycles $C_1,C_2,\dots,C_{2m-1}$ in $N$ labeled in weakly increasing order (on the cylinder) and just as in the previous section, $C_i$ and $C_j$ are vertex disjoint if and only if $|i-j|>1$. Thus the total number of $r$-cycles in $N$ will be the $2m-1$'th Fibonacci number $F_{2m-1}$. Hence the polynomial $Q_{N_{1,m}}(t)$ has $F_{2m-1}$ terms, however, it is a multivariate polynomial in $\x$ unlike the polynomial $F_n(t)$ in the previous section. 
\end{remark}


\subsection{Other applications}
Many other objects correspond to $r$-paths in various planar cylindrical networks. We briefly list several important examples and refer the reader to specific places where the corresponding bijections are described in the literature.

\begin{enumerate}
	\item \emph{Vicious walkers between two walls}, see~\cite[Figure~1]{VW3}. 
	\item \emph{$Q$-Schur functions}, see~\cite[Figure~4a]{StembridgeQ}
	\item \emph{Super-Schur functions}, see~\cite[Figure~1]{BrentiS}.
	\item \emph{Cube recurrence in a cylinder}. In our work in progress~\cite{GPCube}, we give a way to transform a formula from~\cite{CS} in the language of $r$-paths in a certain network.
	\item \emph{States of the six vertex model} $\leftrightarrow$ \emph{cylindric packed loops}, see~\cite[Figures~11,~19]{ZJ}. Note that both of these objects are in bijection with domino tilings of the Aztec diamond.
\end{enumerate}

\begin{remark}
 In order to get a cylindrical network for domino tilings and related objects, we had to restrict the Aztec diamond to the strip $\TS_m$. Note that however for the remaining four items in the above list, the underlying network is naturally cylindrical and there is no need to impose further restrictions on it. Thus a linear recurrence result for them is an immediate consequence of the second part of Theorem~\ref{thm:recurrences}.
\end{remark}

\def\H{\Hcal}
\section{Conjectures}\label{sect:conjectures}
In this section, we give some additional conjectures for the case when $\Nt$ is a planar cylindrical network. Let us denote 
\[H_r=\sum_{\Cbf\in\Cyc^r(N)} \wt(\Cbf),\]
so that $Q_N(t)=\sum_{r=0}^d (-t)^{d-r} H_r$. In particular, we set $H_r=0$ for $r>d$. 

We say that a sequence $(H_0,H_1,\dots)$ of polynomials is a \emph{P\'olya frequency sequence} if all minors of the following \emph{infinite Toeplitz matrix} $\H=(\H_{ij})_{i,j\geq1}$ defined by
\[\H_{ij}= \begin{cases}
	      H_{j-i},&\text{ if $j\geq i$};\\
	      0,&\text{ if $j<i$};
           \end{cases}\quad \H=\begin{pmatrix}
                                	H_0 & H_1 & H_2 & \dots\\
                                	0 & H_0 & H_1 & \dots\\
                                	0 & 0 & H_0 & \dots\\
                                	\vdots &\vdots & \vdots & \ddots
                                \end{pmatrix}\]
are polynomials in $\x$ with nonnegative integer coefficients (in other words, the matrix $\H$ is required to be \emph{totally positive}). For example, the fact that all $1\times 1$ minors have nonnegative coefficients means that each $H_i$ has nonnegative coefficients, and the fact that all $2\times 2$ minors have nonnegative coefficients implies that the sequence $(H_0,H_1,\dots)$ is \emph{strongly log-concave} meaning that the polynomial $H_i^2-H_{i+1}H_{i-1}$ has nonnegative coefficients for each $i>0$.

\begin{conjecture}\label{conj:Polya}
	The polynomials $H_0,H_1,\dots$ form a P\'olya frequency sequence.
\end{conjecture}

Let $\Lambda$ be the ring of symmetric functions (see~\cite[Chapter~7]{EC2}), and consider the ring homomorphism $\psi:\Lambda\to \Z[\x]$ defined by 
\[\psi(e_r)=H_r,\quad r\geq 1.\]
In other words, $\psi(f)$ is obtained from $f\in\Lambda$ by specializing it to the roots of $Q_N(t)$.

By the dual Jacobi-Trudi identity~\cite[Corollary~7.16.2]{EC2}, the image $\psi(s_\l)$ of a Schur function $s_\l$ is given by a row-solid minor of $\H$. Arbitrary minors of $\H$ are images $\psi(s_{\l/\m})$ of skew-Schur functions and thus by the Littlewood-Richardson rule~\cite[Section~A1.3]{EC2} are nonnegative integer combinations of the row-solid minors of $\H$. Thus Conjecture~\ref{conj:Polya} can be equivalently stated as follows:

\begin{conjecture}\label{conj:Schur_Polya}
	The images $\psi(s_\l)$ of Schur functions are polynomials in $\Z[\x]$ with nonnegative coefficients.
\end{conjecture}

In particular, Conjecture~\ref{conj:Schur_Polya} would imply that the coefficient of $(-1)^kt^{D-k}$ in $Q_N^\plee{r}(t)$ is a nonnegative polynomial in $\x$, where $D={d\choose r}$ is the degree of $Q_N^\plee{r}(t)$. This is the case since the plethysm $e_k[e_r]$ of two Schur positive functions $e_k$ and $e_r$ is again Schur positive, but its image $\psi(e_k[e_r])$ is the desired coefficient of $Q_N^\plee{r}(t)$.

\begin{conjecture}\label{conj:roots}
	Fix some $r\geq 1$ and substitute positive real numbers for the variables in $\x$. After such a substitution, the polynomials $Q_N^\pleh{r}$ and $Q_N^\plee{r}$ have positive real roots.
\end{conjecture}
Of course, by~\eqref{eq:plee_pleh_intro}, it is enough to prove Conjecture~\ref{conj:roots} for $r=1$. By~\cite[Theorem~4.5.3]{Brenti}, Conjecture~\ref{conj:roots} is a special case of Conjecture~\ref{conj:Polya}. One way to prove Conjecture~\ref{conj:roots} would be to show the follwing statement.

\begin{conjecture}\label{conj:tp_matrix}
	Substitute positive real numbers for the variables in $\x$. Then there exists a local lift $L$ such that the matrix $S$ defined by~\eqref{eq:S} is \emph{totally positive}, that is, all minors of $S$ are positive.
      \end{conjecture}

      Indeed, by~\eqref{eq:det_S} the roots of $Q_N(t)$ are the eigenvalues of $S$ so Conjecture~\ref{conj:roots} follows since totally positive matrices are known to have positive real eigenvalues. We thank Richard Stanley for suggesting this way of proving Conjecture~\ref{conj:roots}.

We finish with another conjecture that potentially increases the value of the second part of Theorem~\ref{thm:recurrences}. Let us say that a cylindrical network $\Nt$ is \emph{strongly connected} if for any two vertices $\ut,\vt$ of $\Nt$ there exists an integer $\ell\in\Z$ and a directed path in $\Nt$ from $\ut$ to $\vt+\ell\g$. Equivalently, $\Nt$ is strongly connected if and only if the directed graph $N$ is strongly connected in the usual sense.

\begin{conjecture}\label{conj:minimal}
	Suppose we are given a strongly connected planar cylindrical network $\Nt$ such that the weights of the edges of $N$ are algebraically independent. Then for any integer $r\geq 1$ and the sequence $f$ from Theorem~\ref{thm:recurrences} the polynomial $Q_N^\plee{r}(t)$ is the \emph{minimal} recurrence polynomial for $f$.
\end{conjecture}

\bibliographystyle{plain}
\bibliography{networks}

\end{document}